\documentclass[a4paper,10pt]{article}
\usepackage{psfrag}
\usepackage{amsmath, amsfonts, amscd, amssymb, amsthm, enumerate}
\usepackage[all]{xy}
\usepackage[francais,english]{babel}
\usepackage[latin1]{inputenc}

\title{A classifying space for Phillips' equivariant K-theory}
\author{Clément de Seguins Pazzis
\footnote{e-mail adress: dsp.prof@gmail.com}
\footnote{This work is part of the author's PhD thesis that he completed at the Institut Galil\'ee in Universit\'e Paris Nord,
99 avenue Jean-Baptiste Cl\'ement, 93430 Villetaneuse, FRANCE}}

\def\calE{\mathcal{E}}
\def\defterm{\textbf}
\def\N{\mathbb{N}}

\def\R{\mathbb{R}}
\def\C{\mathbb{C}}
\def\diff{\text{d}}
\def\calC{\mathcal{C}}
\def\calD{\mathcal{D}}

\newcommand{\Hom}{\operatorname{Hom}}

\newcommand{\vEc}{\operatorname{Vec}}
\newcommand{\VEct}{\operatorname{\mathbb{V}ect}}
\newcommand{\Func}{\operatorname{Func}}
\newcommand{\im}{\operatorname{Im}}

\newcommand{\ktimes}{\operatorname{\underset{k}{\times}}}

\newcommand{\Ob}{\operatorname{Ob}}
\newcommand{\Mor}{\operatorname{Hom}}
\newcommand{\In}{\operatorname{In}}
\newcommand{\Fin}{\operatorname{Fin}}
\newcommand{\Id}{\operatorname{Id}}
\newcommand{\id}{\operatorname{id}}

\newcommand{\Comp}{\operatorname{Comp}}

\newcommand{\Fred}{\operatorname{Fred}}
\newcommand{\Ker}{\operatorname{Ker}}
\newcommand{\sub}{\operatorname{sub}}

\newcommand{\Lat}{\operatorname{Lat}}
\newcommand{\Mod}{\operatorname{\text{-}mod}}

\newcommand{\smod}{\operatorname{\text{-}smod}}
\newcommand{\sframe}{\operatorname{\text{-}sframe}}
\newcommand{\sBdl}{\operatorname{\text{-}sBdl}}
\newcommand{\Sim}{\operatorname{GU}}

\newcommand{\colim}{\operatorname{colim}}
\newcommand{\Indlim}{\operatorname{\underset{\longrightarrow}{\colim}}}
\newcommand{\Shift}{\operatorname{Sh}}
\renewcommand{\setminus}{\smallsetminus}
\renewcommand{\epsilon}{\varepsilon}

\newcommand{\longleft}[1]{\;{\leftarrow%
    \count255=0 \loop \mathrel{\mkern-6mu}%
    \relbar\advance\count255 by1\ifnum\count255<#1\repeat}\;}
\newcommand{\longright}[1]{\;{\count255=0 \loop \relbar\mathrel{\mkern-6mu}%
    \advance\count255 by1\ifnum\count255<#1\repeat\rightarrow}\;}
\newcommand{\Right}[2]{\overset{#2}{\longright#1}}
\newcommand{\RIGHT}[3]{\mathrel{\mathop{\kern0pt\longright#1}
   \limits^{#2}_{#3}}}

\newcommand{\dRIGHT}[3]{\mathrel{%
   \mathop{\vcenter{\baselineskip=0pt\hbox{$\kern0pt\longright#1$}%
   \hbox{$\kern0pt\longright#1$}}}\limits^{#2}_{#3}}}
\newcommand{\LRIGHT}[3]{\mathrel{%
   \mathop{\vcenter{\baselineskip=0pt\hbox{$\kern0pt\longleft#1$}%
   \hbox{$\kern0pt\longright#1$}}}\limits^{#2}_{#3}}}

\theoremstyle{plain}
\newtheorem{theo}{Theorem}[section]
\newtheorem{prop}[theo]{Proposition}
\newtheorem{cor}[theo]{Corollary}
\newtheorem{lemme}[theo]{Lemma}
\newtheorem{conj}[theo]{Conjecture}

\theoremstyle{definition}
\newtheorem{Def}[theo]{Definition}

\theoremstyle{remark}

\begin{document}

\maketitle

\begin{abstract}
In a previous paper, we have constructed, for an arbitrary Lie group $G$ and any of the fields $F=\R$ or $\C$,
a good equivariant cohomology theory $KF_G^*(-)$ on the category of proper $G$-CW-complex
and have justified why it deserved the label ``equivariant K-theory". It was shown in particular how
this theory was a logical extension of the construction of Lück and Oliver for discrete groups
and coincided with Segal's classical K-theory when $G$ is a compact group and only finite $G$-CW-complexes
are considered. Here, we compare our new equivariant K-theory with that of N.C. Phillips:
it is shown how a natural transformation from ours to his may be constructed which
gives rises to an isomorphism when $G$ is second-countable and only finite proper $G$-CW-complexes are considered.
This solves the long-standing issue of the existence of a classifying space for Phillips' equivariant K-theory.
\end{abstract}

\section{Introduction}

\subsection{The problem}

In this paper, $G$ will denote a second-countable Lie group, $\mu$ a right Haar
measure on $G$, and $F$ one of the fields $\R$ or $\C$.
We wish to compare our equivariant K-theory \cite{Ktheo1}, which is defined on the category of proper $G$-CW-complexes, with
Phillips' equivariant K-theory $KF_G^{\text{Ph}}(-)$ which is defined on the category of proper locally
compact Hausdorff $G$-spaces (see \cite{Phillips}).

To do this, we will construct a natural transformation $\eta : KF_G(-) \longrightarrow KF_G^{\text{Ph}}(-)$
on the category of finite proper $G$-CW-complexes, such that the diagram
$$\xymatrix{
KF_G(-) \ar[rr]^{\eta} & & KF_G^{\text{Ph}}(-) \\
& \mathbb{K}F_G(-) \ar[ul]^{\gamma} \ar[ur]
}
$$
commutes (the natural transformation $\mathbb{K}F_G(-) \longrightarrow  KF_G^{\text{Ph}}(-)$ is the one defined by Phillips
p.40 of \cite{Phillips}).
Once we will have done so, it will be easy to check that $\eta$ is an isomorphism on coefficients
(i.e. for spaces of the type $G/H \times Y$ where $Y$ is a finite CW-complex with trivial action of $G$),
and it will easily follow that $\eta_X$ is an isomorphism for every finite proper $G$-CW-complex $X$.
Before explaining
the construction, we will start with a quick recollection of the definition of our version of equivariant K-theory.

\subsection{A review of equivariant K-theory for proper $G$-CW-complexes}

\subsubsection{$\Gamma$-spaces}

The simplicial category is denoted by $\Delta$ (cf.\ \cite{G-Z}).
Recall that the category $\Gamma$ (see \cite{Segal-cat}) has
the finite sets as objects, a morphism from $S$ to $T$ being a map from
$\mathcal{P}(S)$ to $\mathcal{P}(T)$ which preserves disjoint unions (with obvious composition of morphisms);
this is equivalent to having a map $f$ from $S$ to $\mathcal{P}(T)$
such that $f(s) \cap f(s') =\emptyset$ whenever $s \neq s'$.

For every $n \in \N$, we set $\mathbf{n}:=\{1,\dots,n\}$ and $[n]:=\{0,\dots,n\}$.
Recall the canonical functor $\Delta \rightarrow \Gamma$ obtained
by mapping $[n]$ to $\mathbf{n}$ and the morphism $\delta: [n] \rightarrow [m]$
to $$\begin{cases}
\mathbf{n} & \longrightarrow \mathcal{P}(\mathbf{m}) \\
k & \longmapsto \{j \in \N: \bigl\{\delta(k-1) < j \leq \delta(k)\bigr\}.
\end{cases}$$

By a $\Gamma$-space, we mean a
\emph{contravariant} functor $\underline{A}: \Gamma \rightarrow \text{CG}$
to the category of k-spaces
such that $\underline{A}(\mathbf{0})$ is a well-pointed contractible space.
The space $\underline{A}(\mathbf{1})$ is then simply denoted by $A$.
We say that $\underline{A}$ is a  \defterm{good} $\Gamma$-space when, in addition, for all
$n \in \mathbb{N}^*$, the continuous map $\underline{A}(\mathbf{n}) \rightarrow \underset{i=1}{\overset{n}{\prod}} A$, induced by all morphisms
$\mathbf{1} \rightarrow \mathbf{n}$ which map $1$ to $\{i\}$, is a homotopy equivalence.
From now on, when we talk of $\Gamma$-spaces, we will actually mean good $\Gamma$-spaces.

When $\underline{A}$ is a $\Gamma$-space, composition with the previously defined functor
$\Delta \rightarrow \Gamma$ yields a simplicial space, which we still write $\underline{A}$,
and we can take its thick geometric realization (as defined in appendix A of \cite{Segal-cat}), which we write $BA$.
Since $\underline{A}(\mathbf{0})$ is well-pointed and contractible, we have a map $A \rightarrow \Omega B\underline{A}$
that is ``canonical up to homotopy''. Recall that we have an H-space structure on
$A$ by composing the map $\underline{A}(\mathbf{2}) \rightarrow A$ induced by
$\begin{cases}
\{1\} & \rightarrow \mathcal{P}(\mathbf{2}) \\
1 & \mapsto \{1,2\}
\end{cases}$ and a homotopy inverse of the map $\underline{A}(\mathbf{2}) \rightarrow A \times A$ mentioned earlier.

Given a topological group $G$, a \defterm{$\Gamma-G$-space} is a contravariant functor
$\underline{A}: \Gamma \rightarrow \text{CG}_G$ such that:
\begin{enumerate}[(i)]
\item $\underline{A}(\mathbf{0})$ is equivariantly well-pointed and equivariantly contractible;
\item For any $n \in \N^*$, the canonical map $\underline{A}(\mathbf{n}) \rightarrow \underset{i=1}{\overset{n}{\prod}} A$
is an equivariant homotopy equivalence.
\end{enumerate}
When $\underline{A}$ is a $\Gamma-G$-space, we may define as before
a $G$-map $A \rightarrow \Omega BA$.

\subsubsection{k-categories}

\label{struct}If $\mathcal{C}$ is a small category, then:
\begin{itemize}
\item $\Ob(\mathcal{C})$ (resp. $\Mor(\mathcal{C})$) will denote its set of objects (resp. of morphisms).
\item The structural maps of $\calC$ i.e. the initial, final, identity and composition maps are respectively denoted by
$$\In_{\mathcal{C}}: \Mor(\mathcal{C}) \rightarrow \Ob(\mathcal{C}) \quad ; \quad
\Fin_{\mathcal{C}}: \Mor(\mathcal{C}) \rightarrow \Ob(\mathcal{C});$$
$$\Id_{\mathcal{C}}: \Ob(\mathcal{C}) \rightarrow \Mor(\mathcal{C}) \quad \text{and} \quad
\Comp_{\mathcal{C}}: \Mor(\mathcal{C}) \underset{\vartriangle}{\times} \Mor(\mathcal{C}) \rightarrow \Mor(\mathcal{C}).$$
\item The nerve of $\calC$ is denoted by $\mathcal{N}(\calC)$, whilst
$\mathcal{N}(\calC)_m$ will denote its $m$-th component for any $m\in \N$.
\end{itemize}
By a \defterm{k-category}, we mean a small category with k-space topologies on the sets
of objects and spaces, such that the structural maps induce continuous maps in the category of k-spaces.
To every topological category $\calC$, we assign a k-category whose space of objects and space of morphisms are
respectively $\Ob(\calC)_{(k)}$ and $\Mor(\calC)_{(k)}$.

To a k-category, we may assign its \emph{nerve} in the category of k-spaces,
and then take one of the two geometric realizations $\| \quad \|$ (the ``thick realization") or $| \quad |$
(the ``thin realization") of it in the category of k-spaces (see \cite{Segal-cat}).
When $\calC$ and $\calD$ are two k-categories, we may define another k-category, denoted by
$\Func(\calC,\calD)$, whose objects are the topological
functors from $\calC$ to $\calD$, and
whose morphisms are the continuous natural transformations between continuous functors from
$\calC$ to $\calD$. The structural maps are obvious, as are the topologies on the sets of objects and morphisms
(see \cite{Ktheo1} for more details).

Given a $k$-space $X$, we define a $k$-category $\calE X$ with $X$ as space of objects, $X \ktimes X$
as space of morphisms, and $(x,y)$ as the only morphism from $x$ to $y$ in $\calE X$.

The category of k-categories is denoted by $\text{kCat}$.

\subsubsection{Proper $G$-CW-complexes}

A $G$-space $X$ is called a $G$\defterm{-CW-complex} when it is obtained as the direct limit of a
sequence $(X_{(n)})_{n \in \mathbb{N}}$ of subspaces for which there exists, for every  $n \in \mathbb{N}$,
a set $I_n$, a family $(H_i)_{i \in I_n}$ of closed subgroups of $G$ and a push-out square
$$\begin{CD}
\underset{i \in I_n}{\coprod} (G/H_i) \times S^{n-1} @>>> \underset{i \in I_n}{\coprod} (G/H_i) \times D^n \\
@VVV @VVV \\
X_{(n-1)} @>>> X_{(n)}
\end{CD}$$ in the category of $G$-spaces
(where we have a trivial action of $G$ on both the $(n-1)$-sphere $S^{n-1}$ and the closed $n$-disk $D^n$), with the convention that
$X_{-1}=\emptyset$. The spaces $(G/H_i) \times \overset{\circ}{D^n}$
are called the equivariant cells (or $G$-cells) of $X$.
A $G$-CW-complex is \defterm{proper} when all its isotropy subgroups are compact,
i.e. all the groups $H_i$ in the preceding description are compact.

Relative $G$-CW-complexes are defined accordingly.
A \defterm{pointed proper $G$-CW-complex} is a relative $G$-CW-complex $(X,*)$, with $*$ a point, such that the $G$-space $X \setminus *$ is proper.
Notice that whenever $(X,A)$ is a relative $G$-CW-complex such that $X \setminus A$ is proper, the $G$-space $X/A$
inherits a natural structure of pointed proper $G$-CW-complex.

\subsubsection{The category of compactly-generated $G$-spaces}

Let $G$ be a topological group.
A \defterm{$G$-pointed k-space} consists of a $G$-space which is a k-space
together with a point in it which is fixed by the action of $G$.

The category $CG_G^{\bullet}$ is the one whose objects are the
$G$-pointed k-spaces and whose morphisms are the pointed $G$-maps.
The category $CG_G^{h\bullet}$ is the category with the same objects as
$CG_G^{\bullet}$, and whose morphisms are the equivariant pointed homotopy classes
of $G$-maps between objects (i.e. $CG_G^{h\bullet}$ it is the homotopy category of $CG_G^{\bullet})$.
Given two $G$-spaces (resp. two pointed $G$-spaces) $X$ and $Y$, we let $[X,Y]_G$ (resp. $[X,Y]_G^\bullet$)
denote the set of equivariant homotopy classes of $G$-maps (resp. pointed $G$-maps) from $X$ to $Y$.

A morphism $f: X \rightarrow Y$ is called a \defterm{$G$-weak equivalence}
when the restriction $f^H:X^H \rightarrow Y^H$
is a weak equivalence for every compact subgroup $H$ of $G$
(here, $X^H$ denotes the subspace $\bigl\{x \in X : \; \forall h \in H, \; h.x=x\bigr\}$ of $X$).
We finally define $W_G$ as the class of morphisms in $CG_G^{h\bullet}$
which have $G$-weak equivalences as representative maps.
We may then consider the category of fractions $CG_G^{h\bullet}[W_G^{-1}]$, with
its usual universal property. The crucial property in this paper is the following one:

\begin{prop}\label{pointedweakeq}
Let $Y \overset{f}{\rightarrow} Y'$ be a $G$-weak equivalence between
pointed $G$-spaces. \\
Then, for every proper pointed $G$-CW-complex $X$, the map $f$ induces a bijection
$$f_*: [X,Y]_G^\bullet \longrightarrow [X,Y']_{G.}^\bullet$$
\end{prop}

\subsubsection{$G$-fibre bundles}

Let $G$ be a topological group. Given
a $G$-space $X$, we call \defterm{pseudo-$G$-vector bundle}
(resp.\ $G$-\defterm{vector bundle}) over $X$ the data consisting of
a pseudo-vector bundle (resp.\ a vector bundle)
$p : E \rightarrow X$ over $X$ and of a (left) $G$-action on $E$, such that
$p$ is a $G$-map, and, for all $g \in G$ and $x \in X$, the map $E_x \rightarrow E_{g.x}$
induced by the $G$-action on $E$ is a linear isomorphism.

Given an integer $n \in \N$ and a $G$-space $X$,
$\mathbb{V}\text{ect}_G^{F,n}(X)$ will denote the set of isomorphism classes of
$n$-dimensional $G$-vector bundles over $X$. Accordingly,
$\mathbb{V}\text{ect}_G^{F}(X)$ will denote the abelian monoid
of isomorphism classes of finite-dimensional $G$-vector bundles over $X$. \label{Gvectorbundle}

Given another topological group $H$, a \defterm{$(G,H)$-principal bundle} is
an $H$-principal bundle $\pi : E \rightarrow X$ with structures of $G$-spaces on $E$ and $X$
for which $\pi$ is a $G$-map and $\forall (g,h,x)\in G \times H \times E, \; g.(x.h)=(g.x).h$.

\subsubsection{Topological categories attached to simi-Hilbert bundles}

\label{orthogroups}
We denote by $U_n(F)$ (resp.\ $\Sim_n(F)$) the group of orthogonal automorphisms
(resp. of similarities) of the vector space $F^n$.

\begin{Def}A \defterm{simi-Hilbert space} is a finite-dimensional vector space $V$ (with ground field $F$)
with a linear family $(\lambda \langle -,- \rangle)_{\lambda \in \mathbb{R}_+^*}$ of inner products on $V$.
\end{Def}

The relevant notion of isomorphisms between two simi-Hilbert spaces is that of similarities.
We do have a notion of orthogonality, but no notion of orthonormal families.
The relevant notion is that of \defterm{simi-orthonormal} families:
a family will be said to be simi-orthonormal when it is orthogonal and all its vectors share the same positive norm (for any inner product in the
linear family). Equivalently, a family of vectors is simi-orthonormal if and only if it is orthonormal
for some inner product in the linear family.

\begin{Def}
Let $G$ be a topological group. \\
For $n \in \mathbb{N}$, an $n$-dimensional simi-$G$-Hilbert bundle is a $G$-vector bundle with fiber $F^n$ and structural group $\Sim_n(F)$. \\
A disjoint union of $k$-dimensional $G$-simi-Hilbert bundles, for $k\in \mathbb{N}$, is called a $G$-\textbf{simi-Hilbert bundle.}
\end{Def}

We fix an integer $n \in \mathbb{N}$ for the rest of the paragraph.
Let $\varphi: E \rightarrow X$ be an $n$-dimensional simi-Hilbert over a locally-countable CW-complex, and $\tilde{\varphi}: \tilde{E} \rightarrow X$
the $\Sim_n(F)$-principal bundle canonically associated to it (by considering $\tilde{E}$ as a subspace of $E^{\oplus n}$). \\
\label{defsframe}\label{defsmod}\label{defsBdl}We define:
\begin{itemize}
\item $\varphi \sframe$ as the category $\mathcal{E}\tilde{E}$, with a natural right-action of $\Sim_n(F)$;
\item $\varphi\smod$ as $\varphi\sframe/\Sim_n(F)$:
an object of $\varphi\smod$ corresponds to a point of $X$, and
a morphism from $x$ to $y$ (with $(x,y) \in X^2$) corresponds to
a similarity $E_x \overset{\cong}{\rightarrow} E_y$;
\item $\varphi\sBdl$ as the category whose space of objects is $E$,
and whose space of morphisms is the (closed) subspace of
$E \ktimes E \ktimes \Mor(\varphi\smod)$
consisting of those triples $(e,e',f)$ such that
$\varphi(e) \underset{f}{\longrightarrow} \varphi(e')$ and $f(e)=e'$.
\end{itemize}

\subsubsection{The category $\Gamma \text{-Fib}_F$ and the $\smod$ functor}

\label{4.1.1}We define the category $\Gamma \text{-Fib}_F$ as follows:
\begin{itemize}
\item An object of $\Gamma \text{-Fib}_F$ consists of a finite
set $S$, a locally-countable CW-complex $X$,
and, for every $s \in S$, of a Hilbert bundle
$p_s:E_s \rightarrow X$ with ground field $F$.
Such an object is called an \textbf{$S$-object} over $X$.
If $S=\mathbf{n}$ for some $n \in \mathbb{N}$,
an $S$-object will be called an $n$-object.
\item A morphism $f:(S,X,(p_s)_{s\in S}) \longrightarrow
(T,Y,(q_t)_{t\in T})$ consists of a morphism $\gamma: T \rightarrow S$
in the category $\Gamma$, a continuous map
$\bar{f}:X \rightarrow Y$, and, for every $t \in T$,
a strong morphism of Hilbert bundles
$$\begin{CD}
\underset{s \in \gamma(t)}{\oplus}E_s @>{f_t}>> E'_t \\
@V{\underset{s \in \gamma(t)}{\oplus}p_s}VV @VV{q_t}V \\
X @>>{\bar{f}}> Y.
\end{CD}$$
\end{itemize}
If $f:(S,X,(p_s)_{s\in S}) \rightarrow (T,Y,(q_t)_{t\in T})$ is the morphism
in $\Gamma \text{-Fib}_F$ corresponding to $(\gamma,\bar{f},(f_t)_{t\in T})$, and
$g:(T,Y,(q_t)_{t\in T}) \rightarrow (U,Z,(r_u)_{u\in U})$ is the morphism
in $\Gamma \text{-Fib}_F$ corresponding to $(\gamma',\bar{g},(g_u)_{u\in U})$, then the composite
morphism $g \circ f: (S,X,(p_s)_{s\in S}) \rightarrow (U,Z,(r_u)_{u\in U})$
is the one which corresponds to the triple consisting of $\gamma \circ \gamma'$, $\bar{g} \circ \bar{f}$, and
the family
$\left(g_u \circ \left[\underset{t\in \gamma'(u)}{\oplus}f_t\right]\right)_{u \in U.}$

\label{4.2.2}
Let $(X,p:E \rightarrow X)$ be a $1$-object of $\Gamma \text{-Fib}_F$.
The natural map $\dim_p:\begin{cases}
X & \longrightarrow \mathbb{N} \\
x  & \longmapsto \dim(E_x)
\end{cases}$ is continuous because $X$ is a CW-complex.
Setting $X_{n}:=\dim_p^{-1}\{n\}$, $E_n:=p^{-1}(X_n)$, and
$p_n=p_{|E_n}: E_n \rightarrow X_{n}$, we then have $p=\underset{n \in \mathbb{N}}{\coprod}p_n.$ \\
Set $$p \smod
:=\underset{n \in \mathbb{N}}{\coprod}(p_n \smod).$$
We obtain a functor
$p \smod \rightarrow \mathcal{E}X \times \mathcal{B}\mathbb{R}_+^*$ by assigning $(x,y,\|\varphi\|)$ to every
morphism $\varphi:E_x \rightarrow E_y$ (here, $\|\varphi\|$ denotes the norm of the similarity
$\varphi$ with respect to the respective inner product structures on $E_x$ and $E_y$):
this is compatible with the composition of morphisms, since we are dealing
with similarities here. \\
For any $S$-object $\varphi=(S,X,(p_s)_{s \in S})$,
$\varphi \smod$ is defined as the fiber product of the categories
$p_s \smod$ over $\mathcal{E}X \times \mathcal{B}\mathbb{R}_+^*$ for all $s\in S$.
For any $S$-object $\varphi=(S,X,(p_s:E_s \rightarrow X)_{s \in S})$, an object of
$\varphi \smod$ simply corresponds to a point $x\in X$, while
a morphism $x \rightarrow y$ in $\varphi \smod$ is a family $(\varphi_s)_{s\in S}$
of similarities $\varphi_s: (E_s)_x \overset{\cong}{\rightarrow} (E_s)_y$ \emph{which share the same norm}.
It is then easy to extend this construction to obtain a functor
$$\smod: \Gamma \text{-Fib}_F \longrightarrow \text{kCat}.$$

\subsubsection{Hilbert $\Gamma$-bundles}
\label{4.4}
\begin{Def}
A \defterm{Hilbert $\Gamma$-bundle} is a contravariant functor
$\varphi:\Gamma \longrightarrow \Gamma \text{-Fib}_F$
which satisfies the following conditions:

\begin{enumerate}[(i)]
\item $\mathcal{O}^F_\Gamma \circ \varphi=\id_{\Gamma}$;
\item $\varphi(\mathbf{0})=(\mathbf{0},*,\emptyset)$;
\item For every $n \in \mathbb{N}^*$, there exists a morphism $f_n:n.\varphi
(\mathbf{1}) \rightarrow \varphi(\mathbf{n})$ in $\Gamma \text{-Fib}_F$ such that
$\mathcal{O}^F_\Gamma(f_n)=\id_{\mathbf{n}}$.
\end{enumerate}
\end{Def}

\label{gammaspaces}
\begin{Def}
Let $\varphi$ be an object of $\Gamma-\text{Fib}_F$ and $G$ be a Lie group. We define
$$s\vEc_G^\varphi:=\left|\Func(\mathcal{E}G,\varphi \smod)\right|.$$
\end{Def}

In what follows, we let $\varphi$ be a Hilbert $\Gamma$-bundle and
$G$ be a Lie group. We define a functor from $\Gamma$ to $\text{CG}_G$:
$$s\underline{\vEc}_G^\varphi: S \longmapsto s\vEc_G^{\varphi(S)}.$$
This is actually a $\Gamma-G$-space.

\begin{Def}
For every finite set $S$, we let $\Gamma (S)$ denote the set of maps $f:\mathcal{P}(S)
\rightarrow \mathcal{P}(\mathbb{N})$ which respect disjoint unions (and in particular $f(\emptyset)=\emptyset$),
and such that $f(S)$ is finite. We will write $S \overset{f}{\rightarrow}
\mathbb{N}$ when $f\in \Gamma(S)$.

For an inner product space $\mathcal{H}$ (with underlying field $F$) of finite dimension or
isomorphic to $F^{(\infty)}$, and for a finite subset $A$ of $\mathbb{N}$, we may consider the inner product space
$\mathcal{H}^A$ as embedded in the Hilbert space $\mathcal{H}^\infty$.
We then define $G_A(\mathcal{H})$ as the set of subspaces of dimension $\# A$ of $\mathcal{H}^A$,
with the limit topology for the inclusion of $G_{\#A}(E)$, where $E$ ranges over
the finite dimensional subspaces of $\mathcal{H}$.
When $A$ is empty, we set $G_\emptyset(\mathcal{H})=*$.

We let $p_A(\mathcal{H}): E_A(\mathcal{H}) \longrightarrow G_A(\mathcal{H})$
denote the canonical Hilbert bundle of dimension $\#A$ over $G_A(\mathcal{H})$.
The set $E_A(\mathcal{H})$ is constructed as a subspace of
the product of $\mathcal{H}^A$ (with the limit topology described above)
with $G_A(\mathcal{H})$.
\end{Def}

\label{4.5}For every finite set $S$, we define the following object of $\Gamma \text{-Fib}_F$:
$$\text{Fib}^{\mathcal{H}}(S):=
\left(S,X^\mathcal{H}(S),(p^\mathcal{H}(s))_{s \in S}\right),$$
where
$$X^\mathcal{H}(S):=\underset{f \in \Gamma(S)}{\coprod}\left[\underset{s \in S}{\prod}
G_{f(s)}(\mathcal{H})\right]$$
and, for every $s \in S$,
$$p^\mathcal{H}(s):=
\underset{f \in \Gamma(S)}{\coprod}\left[p_{f(s)}(\mathcal{H})
\times \underset{s' \in S\setminus\{s\}}{\prod}
\id_{G_{f(s')}(\mathcal{H})}\right].
$$
Let $\gamma: S \rightarrow T$ be a morphism in $\Gamma$.
We define a morphism
$\text{Fib}^{\mathcal{H}}(\gamma):\text{Fib}^{\mathcal{H}}(T) \rightarrow
\text{Fib}^{\mathcal{H}}(S)$ of $\Gamma \text{-Fib}_F$ for which
$\mathcal{O}^F_\Gamma(\text{Fib}^{\mathcal{H}}(\gamma))=\gamma$, in the following way:
for every $f \in \Gamma(S)$, we consider the map
$$\begin{cases}
\underset{t \in \gamma(s)}{\prod}
G_{f(t)}(\mathcal{H}) & \longrightarrow
G_{f \circ \gamma (s)}(\mathcal{H}) \\
(E_t)_{t \in \gamma(s)} & \longmapsto \underset{t \in \gamma(s)}{\overset{\bot}{\oplus}}
E_t,
\end{cases}$$
and, for every $s \in S$ and $f \in \Gamma(S)$, we have a commutative
square:
$$\begin{CD}
\left[\underset{t \in \gamma(s)}{\oplus}
E_{f(t)}(\mathcal{H})\right]
\times \underset{t \in T \setminus \gamma(s)}{\prod}
G_{f(t)}(\mathcal{H}) @>>>
E_{f\circ \gamma(s)}(\mathcal{H})
\times \underset{s' \in S \setminus \{s\}}{\prod}
G_{f(s')}(\mathcal{H}) \\
@VVV @VVV \\
\underset{t \in T}{\prod}
G_{f(t)}(\mathcal{H}) @>>>
\underset{s_1 \in S}{\prod}
G_{f\circ \gamma(s_1)}(\mathcal{H}),
\end{CD}
$$
where the upper morphism is given by the previous map and the following one:
$$\begin{cases}
\underset{t \in \gamma(s)}{\oplus}
E_{f(t)}(\mathcal{H}) & \longrightarrow
E_{f \circ \gamma (s)}(\mathcal{H}) \\
(x_t)_{t \in \gamma(s)} & \longmapsto
\underset{t \in \gamma(s)}{\sum} x_t.
\end{cases}$$

\label{defFoo}Denoting by $F^{(\infty)}$ the direct limit of the sequence $(F^k)_{k \geq 0}$
for the standard inclusion of inner product spaces, we have the following result:

\begin{prop}
Let $\mathcal{H}$ be an inner product space with ground field $F$. Assume that
$\mathcal{H}$ is finite-dimensional or isomorphic to $F^{(\infty)}$. Then
$\text{Fib}^{\mathcal{H}}$ is a Hilbert $\Gamma$-bundle.
\end{prop}

\label{deftypique}We may now set:
$$s\underline{\vEc}_G^{F,\infty}:=
s\underline{\vEc}_G^{\text{Fib}^{F^{(\infty)}}} \quad; \quad
s\vEc_G^{F,\infty}:=
s\vEc_G^{\text{Fib}^{F^{(\infty)}}(\mathbf{1})}
\quad; \quad
Es\vEc_G^{F,\infty}:=
Es\vEc_G^{\text{Fib}^{F^{(\infty)}}(\mathbf{1})}.
$$
The $\Gamma$-space structure of $\underline{\vEc}_G^{F,\infty}$
induces a structure of equivariant H-space on
$\vEc_G^{F,\infty}=\underline{\vEc}_G^{F,\infty}(\mathbf{1})$. The following result was established in \cite{Ktheo1}:

\begin{prop}
Let $G$ be a second-countable Lie group
Then $Es\vEc_G^{F,\infty} \rightarrow s\vEc_G^{F,\infty}$ is universal
for finite-dimensional $G$-simi-Hilbert bundles
over $G$-CW-complexes, and, for every $G$-CW-complex
$X$, the induced bijection
$$\Phi: [X,\vEc_G^{F,\infty}]_G \overset{\cong}{\longrightarrow}
\VEct_G^F(X)$$ is a homomorphism of abelian monoids.
\end{prop}

\subsubsection{A definition of equivariant K-theory}

\label{6}
For any Lie group $G$, we set:
$$sKF_G^{[\infty]}:=\Omega Bs\vEc_G^{F,\infty}.$$

\begin{Def}
Let $G$ be a Lie group, $F=\mathbb{R}$ or $\mathbb{C}$,
$(X,A)$ a proper $G$-CW-pair and $n \in \mathbb{N}$.
We set:
$$KF_G^{-n}(X,A):=\bigl[\Sigma ^n(X/A),sKF_G^{[\infty]}\bigr]_G^\bullet,$$
and
$$KF_G^{-n}(X):=KF_G^{-n}(X\cup \{*\},\{*\}).$$
In particular, for every proper $G$-CW-complex $X$,
$$KF_G(X):=KF_G^0(X)=\bigl[X,sKF_G^{[\infty]}\bigr]_{G.}$$
\end{Def}

It was shown in \cite{Ktheo1} that $KF_G(-)$ may be extended to positive degrees in order to recover
a good equivariant cohomology theory.
Given a $G$-CW-complex $X$,
the canonical map $i:s\vEc_G^{F,\infty} = \underline{s\vEc}_G^{F,\infty}(\mathbf{1}) \rightarrow sKF_G^{[\infty]}$
induces a natural homomorphism $[X,s\vEc_G^{F,\infty}]_G \overset{i_*}{\rightarrow} [X,sKF_G^{[\infty]}]_G$
of abelian monoids, which, pre-composed with the inverse of the natural isomorphism
$[X,s\vEc_G^{F,\infty}]_G \overset{\cong}{\rightarrow} \mathbb{V}\text{ect}_G^{F}(X)$, yields
a natural homomorphism of abelian monoids $$\mathbb{V}\text{ect}_G^{F}(X) \longrightarrow [X,sKF_G^{[\infty]}]_G.$$
By the universal property of the Grothendieck construction, this yields a natural homomorphism of abelian groups
$$\gamma _X: \mathbb{K}F_G(X) \longrightarrow KF_G(X).$$
This clearly defines a natural transformation
$\gamma: \mathbb{K}F_G(-) \longrightarrow KF_G(-)$
on the category of proper $G$-CW-complexes. It was shown in \cite{Ktheo1} how this natural transformation
may be extended in arbitrary degrees so as to be compatible with long exact sequences.
Finally, the following property of $\gamma$ was established in \cite{Ktheo1}:

\begin{prop}
For every compact subgroup $H$ of $G$, every integer $n$ and every finite CW-complex $Y$ on which $G$ acts
trivially, the map
$$\gamma _{(G/H) \times Y}^{-n}: \mathbb{K}F_G^{-n}((G/H) \times Y)
\overset{\cong}{\longrightarrow} KF_G^{-n}((G/H) \times Y)$$
is an isomorphism.
\end{prop}

\subsection{Main ideas and structure of the paper}
The basic idea is to start from the ``quasi''-classifying space that naturally arises in the study of Phillips' equivariant K-theory,
i.e.\ the space $\Fred(L^2(G)^\infty)$ of Fredholm operators on the Hilbert $G$-module $L^2(G)^\infty$, with the norm topology.
Indeed, any $G$-map $f : X \longrightarrow \Fred(L^2(G)^\infty)$, where $X$ is a proper locally-compact Hausdorff $G$-space, yields
a morphism $$\begin{cases}
X \times L^2(G)^\infty & \longrightarrow X \times L^2(G)^\infty \\
(x,y) & \longmapsto (x,f(x)[y])
\end{cases}$$
of $G$-Hilbert bundle over $X$, which in turns yields an element of $KF_G^{\text{Ph}}(X)$ (cf.\ example 3.5 p.41 of \cite{Phillips}).
This gives rise to a natural transformation $[-,\Fred(L^2(G)^\infty)]_G \longrightarrow KF_G^{\text{Ph}}(-)$ on the category of
proper locally compact Hausdorff $G$-spaces, and it is quite easy to check that, if we consider the structure of topological
monoid on $\Fred(L^2(G)^\infty)$ induced by the composition of Fredholm morphisms, this is
actually a natural transformation between monoid-valued functors.

However, $\Fred(L^2(G)^\infty)$ is not always a classifying space for $KF_G^{\text{Ph}}(-)$: unless
$G$ is a discrete group, it may be shown indeed that
$\Fred(L^2(G)^\infty)$ is not a $G$-space (since the usual $G$-action on this topological space is not continuous).
Classically, this is handled by replacing $\Fred(L^2(G)^\infty)$ with its subspace
$\underline{\Fred}(L^2(G)^\infty)$ consisting of those elements $x$ such that the action of $G$ on the orbit of $x$ is continuous.
Classically, $\underline{\Fred}(L^2(G)^\infty)$ is a $G$-space and every continuous equivariant map
$X \longrightarrow \Fred(L^2(G)^\infty)$ has its range in $\underline{\Fred}(L^2(G)^\infty)$ when $X$ is a $G$-space.
Hence we really have a natural transformation
$$[-,\underline{\Fred}(L^2(G)^\infty)]_G \longrightarrow KF_G^{\text{Ph}}(-)$$
which is not known in general to be an isomorphism (the compact case however has been solved, see \cite{Twisted}).

\begin{figure}[h]
\begin{center}
\psfrag{C}{$s\vEc_G^{F,\infty}$}
\psfrag{C1}{$s\underline{\vEc}_G^{F,\infty}(\mathbf{n})$}
\psfrag{C2}{$\Omega Bs\underline{\vEc}_G^{F,\infty}$}
\psfrag{B}{$s\vEc_{G,L^2}^{F,\infty}$}
\psfrag{B1}{$s\underline{\vEc}_G^{F,\infty}(\mathbf{n})$}
\psfrag{B2}{$\Omega Bs\underline{\vEc}_{G,L^2}^{F,\infty}$}
\psfrag{A}{$s\vEc_{G,L^2}^{F,\infty*}$}
\psfrag{A1}{$s\vEc_{G,L^2}^{F,\infty*}[n]$}
\psfrag{A2}{$\Omega Bs\vEc_{G,L^2}^{F,\infty*}$}
\psfrag{D}{$\underset{k \in \mathbb{N}}{\bigcup}\sub_k(L^2(G,\mathcal{H}))$}
\psfrag{E}{$\underline{\Fred}(L^2(G,\mathcal{H}^\infty))$}
\psfrag{E1}{$\Fred(G,\mathcal{H}^\infty)[n]$}
\psfrag{E2}{$\Omega B\Fred(G,\mathcal{H}^\infty))$}
\includegraphics{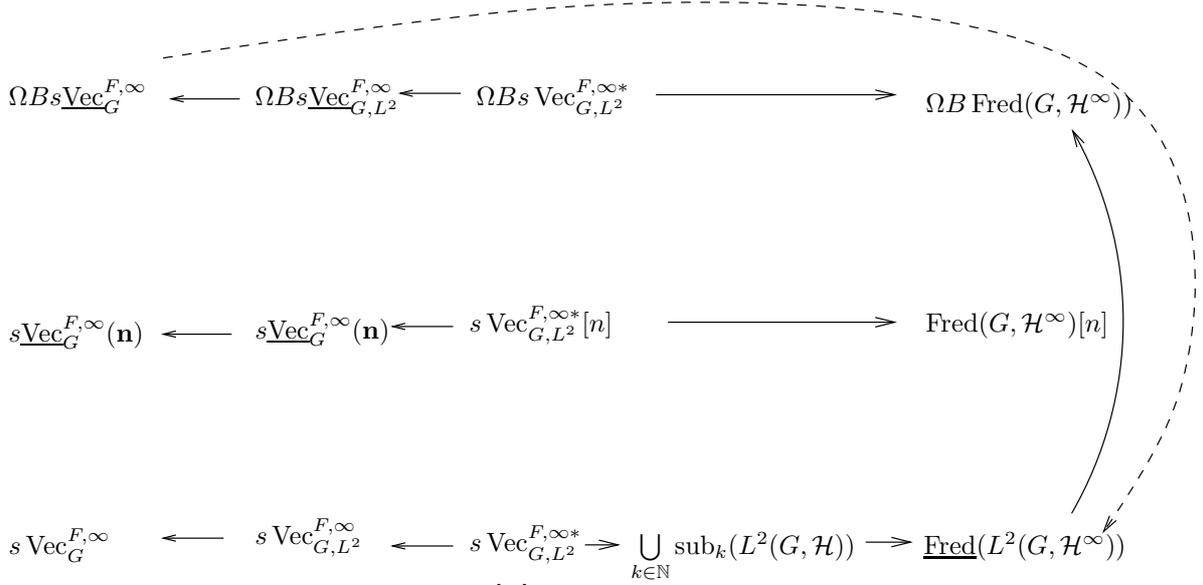}
\caption{Going from $sKF_G^{[\infty]}$ to $\underline{\Fred}(L^2(G,\mathcal{H})^\infty)$}
\end{center}
\end{figure}
\paragraph{}
Our basic idea is to construct a reasonable morphism $sKF_G^{[\infty]} \longrightarrow \underline{\Fred}(L^2(G)^\infty)$ in the category $CG_G^{h\bullet}[W_G^{-1}]$, and then compose the natural transformation $KF_G(-) \longrightarrow [-,\underline{\Fred}(L^2(G)^\infty)]_G$
derived from it with the one discussed earlier. The construction of this morphism is now briefly explained:
\begin{itemize}
\item In Section \ref{10.1}, we will define $L^2$ functors from $\mathcal{E}G$ to $\varphi \smod$ (where $\varphi$ is a Hilbert bundle), we define the space $s\vEc_{G,L^2}^{\varphi}$ from those functors very much like
$s\vEc_G^{\varphi}$, and will construct a canonical $G$-weak equivalence
$s\vEc_{G,L^2}^{\varphi} \rightarrow s\vEc_G^{\varphi}$. This construction will be extended to a
$\Gamma-G$-space $s\underline{\vEc}_{G,L^2}^{F,\infty}$, and a canonical $G$-weak equivalence
$\Omega B s \vEc_{G,L^2}^{F,\infty} \longrightarrow \Omega B s \vEc_{G}^{F,\infty}$ will be produced.
\item In Section \ref{10.2}, we will define the simplicial $G$-space $s\underline{\vEc}_{G,L^2}^{F,\infty *}$
by only considering functors that are $L^2$ and ``contiguous'' label-wise, and natural transformations that are
non-decreasing label-wise. There will then be a canonical morphism
$s\underline{\vEc}_{G,L^2}^{F,\infty*} \longrightarrow s\underline{\vEc}_{G,L^2}^{F,\infty}$
which will induce a $G$-weak equivalence $\Omega B s\vEc_{G,L^2}^{F,\infty*} \longrightarrow \Omega B s \vEc_{G,L^2}^{F,\infty}$.
\item In Section \ref{10.3}, we will briefly introduce the $G$-space $\underline{\Fred}(L^2(G,\mathcal{H}))$
together with the simplicial $G$-space $\Fred(G,\mathcal{H})$ associated to this monoid, and prove that the canonical map
$\underline{\Fred}(L^2(G,\mathcal{H})) \longrightarrow \Omega B \underline{\Fred} (L^2(G,\mathcal{H}))$
is a $G$-weak equivalence. We will also briefly recall the definition of the shift map
$\Shift : \underset{n\in \mathbb{N}}{\bigcup}\sub_n(L^2(G,\mathcal{H})) \longrightarrow \underline{\Fred}(L^2(G,\mathcal{H})^\infty)$
and its basic properties.

\item Section \ref{10.4} is devoted to the construction of a morphism of simplicial $G$-spaces
$s\vEc_{G,L^2}^{F,\infty} \longrightarrow \underline{\Fred}(G,\mathcal{H}^\infty)$: this part is very
technical so we will start with a lengthy explanation of the main ideas before going into more details.
Note that this construction uses an unresolved - yet very reasonable - conjecture on
relative triangulations of smooth compact manifolds, see Section \ref{conjecturesection}.

\item Finally, we wrap up the construction in Section \ref{10.5}, where we obtain a natural transformation
$KF_G(-) \longrightarrow KF_G^{\text{Ph}}(-)$ (that is independent of some choices made during the construction),
and we finish by proving that it is an isomorphism on the category of finite proper $G$-CW-complexes.
\end{itemize}

\subsection{Additional notations and definitions}

Given an integer $n \in \mathbb{N}$, and $G$-space $X$,
$\mathbb{V}\text{ect}_G^{F,n}(X)$ will denote the set of isomorphism classes of
$n$-dimensional $G$-vector bundles over $X$. Accordingly,
$\mathbb{V}\text{ect}_G^{F}(X)$ will denote the abelian monoid
of isomorphism classes of finite-dimensional $G$-vector bundles over $X$.

\vskip 2mm
We denote by $\Delta^*$ the hemi-simplicial category, i.e.\ the subcategory of $\Delta$ with the same objects, and
with the monomorphisms of $\Delta$ as morphisms.
A \defterm{hemi-simplicial} $G$-space is a contravariant functor from $\Delta^*$ to the category of $G$-spaces.

\vskip 2mm
\label{hilbertspace}When $\mathcal{H}$ is an inner product space, and $n \in \mathbb{N}^*$, $B_n(\mathcal{H}) \subset \mathcal{H}^n$
will denote the space of linearly independent
$n$-tuples of elements of $\mathcal{H}$ (with the convention that $B_0(\mathcal{H})=*$), while
$V_n(\mathcal{H}) \subset \mathcal{H}^n$ will denote the space of orthonormal $n$-tuples of elements of $\mathcal{H}$ (with the convention that $V_0(\mathcal{H})=*$).
We also let $B_{\mathcal{H}}$ denote the unit ball of $\mathcal{H}$, $\sub(\mathcal{H})$ denote the
set of closed linear subspaces of $\mathcal{H}$, and, for $n \in \mathbb{N}, \sub_n(\mathcal{H})$
the set of $n$-dimensional linear subspaces of $\mathcal{H}$.
Finally, $\mathcal{B}(H)$ will denote the space of bounded linear operators on $H$.

\vskip 2mm
For every smooth manifold $N$, and every set $x$, we will let $T_x N$ denote the
tangent space of $N$ at $x$ when $x \in N$, and we set $T_x N=\emptyset$  when $x \not\in N$.

\vskip 2mm \label{undercat}Given a category $\mathcal{C}$ and an object $X$ of $\mathcal{C}$, we denote by
$\mathcal{C} \downarrow X$ will denote the category whose objects are the morphisms
$Y \rightarrow X$ of $\mathcal{C}$, and the morphisms from $Y \overset{f}{\rightarrow} X$
to $Y' \overset{g}{\rightarrow} X$ are the morphisms $Y \overset{\varphi}{\rightarrow} Y'$ of $\mathcal{C}$
such that $g \circ \varphi=f$.

\paragraph{List of important notation} ${}$
\vskip 1mm
\def\8[#1;#2]{\hbox to \textwidth{#1 \hfill p. #2}\vskip2truemm}
\8[$\Ob(\mathcal{C}), \Mor(\mathcal{C}), \In_{\mathcal{C}}, \Fin_{\mathcal{C}}, \Id_{\mathcal{C}}, \Comp_\mathcal{C},
\mathcal{N}(\mathcal{C})$
\quad ($\mathcal{C}$ a small category);\pageref{struct}]
\8[$U_n(F), \Sim_n(F)$;\pageref{orthogroups}]
\8[$\Gamma\text{-Fib}_F$;\pageref{4.1.1}]
\8[$\smod, \sBdl : \Gamma \text{-Fib}_F\rightarrow \text{kCat}$;\pageref{4.2.2}]
\8[Hilbert $\Gamma$-bundles $\phi:\Gamma\rightarrow \Gamma \text{-Fib}_F$;\pageref{4.4}]
\8[$s\underline{\vEc}_G^\phi$;\pageref{gammaspaces}]
\8[$\text{Fib}^{\mathcal{H}}(S)$;\pageref{4.5}]
\8[$F^{(\infty)}$;\pageref{defFoo}]
\8[$s\underline{\vEc}_G^{F,\infty}$;\pageref{deftypique}]
\8[$sKF_G^{[\infty]}$;\pageref{6}]
\8[$B_n(\mathcal{H}), V_n(\mathcal{H}), B_{\mathcal{H}}, \sub(\mathcal{H}), \sub_n(\mathcal{H}),
\mathcal{B}(\mathcal{H})$ \quad  ($\mathcal{H}$ an inner product space);\pageref{hilbertspace}]
\8[$\mathcal{C} \downarrow X$;\pageref{undercat}]
\8[$s\underline{\vEc}_{G,L^2}^\phi$ ($\phi$ a Hilbert $\Gamma$-bundle);\pageref{10.1.2}]
\8[$sKF_{G,L^2}^\phi$;\pageref{10.1.3}]
\8[$Es\vEc_{G,L^2}^\phi, s\vEc_{G,L^2}^\phi, s\widetilde{\vEc}_{G,L^2}^\phi$;\pageref{10.1.4}]
\8[$\Gamma^*$;\pageref{gammastar}]
\8[$\Func_{\uparrow L^2}(\mathcal{E}G,\phi \Mod)$;\pageref{increasingfunc}]
\8[$s\vEc_{G,L^2}^{F,m*}$;\pageref{ultimatevec}]
\8[$Es\vEc_{G,L^2}^{F,\infty*}$;\pageref{ultimatebdl}]
\8[$s\widetilde{\vEc}_{G,L^2}^{n,F,\infty*}, Es\vEc_{G,L^2}^{n,F,\infty*}$ ;\pageref{ultimateframe}]
\8[$(\Fred(G,\mathcal{H})[n])_{n \in \mathbb{N}}$;\pageref{10.3.2}]
\8[$\Shift$;\pageref{10.3.3}]
\8[$\Hom_{\uparrow}([n],\mathbb{N})$;\pageref{10.4.1}]
\8[$f+N$;\pageref{10.4.1}]
\8[$S(n),S_{u,v},\Lat(u)$;\pageref{10.4.1}]
\8[$\Delta_i^n, C(\Delta^n)$;\pageref{core}]
\8[$\Delta_u$;\pageref{deltau}]
\8[$\sigma\natural \delta, \delta \natural \sigma, \sigma \natural i$;\pageref{10.4.2}]
\8[$G_\sigma(\delta)$;\pageref{gsigma}]
\8[$\nu_\sigma$;\pageref{nusigma}]
\8[$\mathcal{H}_f,\mathcal{H}_n,\mathcal{H}_f^{(l)}, \mathcal{H}_n^(l)$;\pageref{10.4.3}]
\8[$\mathcal{H}_{n,u,f}, \mathcal{H}_{n,u,f}^{(l)}, \phi_{n,u,f}$;\pageref{hnuf}]
\8[$V_n(f), V_n^{(l)}(f)$;\pageref{10.4.4}]
\8[$\alpha_G^\psi$;\pageref{10.4.7}]

\section{The equivariant $\Gamma$-space $s\vEc_{G,L^2}^\varphi$}\label{10.1}

In this section, $L^2_*(G)$ will denote $L^2(G,\mathbb{R}_+^*) \coprod \{0\}$
(with the $L^2$-topology on $L^2(G,\mathbb{R}_+^*)$), and
$C_*(G)$ will denote $C(G,\mathbb{R}_+^*)\coprod \{0\}$ (with the compact-open topology on
$C(G,\mathbb{R}_+^*)$).
Finally, $C_*(G)\cap L^2_*(G)$
is equipped with the initial topology for the diagonal map
$C_*(G)\cap L^2_*(G) \longrightarrow C_*(G)\times L^2_*(G)$, i.e.
the topology with the fewest open sets so that the diagonal map is continuous.

\subsection{The topological category $\Func_{L^2}(\mathcal{E}G,\varphi \smod)$}\label{10.1.1}

\subsubsection{The topological category induced by a functor}
Let $\mathcal{C}$ be a small category, $\mathcal{D}$ a topological category, and $F: \mathcal{C} \rightarrow \mathcal{D}$ a functor.
We define $\Ob(\mathcal{C}_F)$ as the topological space with underlying set $\Ob(\mathcal{C})$ and
the initial topology with respect to $F_{|\Ob} : \Ob(\mathcal{C}) \rightarrow \Ob(\mathcal{D})$ (i.e. the
topology with the fewest open sets such that $F_{|\Ob}$ is continuous).
We define $\Mor(\mathcal{C}_F)$ as the topological space with underlying set $\Mor(\mathcal{C})$ and
the initial topology with respect to $F_{|\Mor} : \Mor(\mathcal{C}) \rightarrow \Mor(\mathcal{D})$.
Since the squares
$$\begin{CD}
\Ob (\mathcal{C}) @>{F_{|\Ob}}>> \Ob (\mathcal{D}) \\
@V{\Id_\mathcal{C}}VV @V{\Id_\mathcal{D}}VV \\
\Mor (\mathcal{C}) @>{F_{|\Mor}}>> \Mor (\mathcal{D})
\end{CD} \quad , \quad
\begin{CD}
\Mor (\mathcal{C}) @>{F_{|\Mor}}>> \Mor (\mathcal{D}) \\
@V{\Fin_\mathcal{C}}VV @V{\Fin_\mathcal{D}}VV \\
\Ob (\mathcal{C}) @>{F_{|\Ob}}>> \Ob (\mathcal{D})
\end{CD} \quad , \quad
\begin{CD}
\Mor (\mathcal{C}) @>{F_{|\Mor}}>> \Mor (\mathcal{D}) \\
@V{\In_\mathcal{C}}VV @V{\In_\mathcal{D}}VV \\
\Ob (\mathcal{C}) @>{F_{|\Ob}}>> \Ob (\mathcal{D})
\end{CD}
$$
and
$$\begin{CD}
\Mor (\mathcal{C}) \underset{\vartriangle}{\times} \Mor(\mathcal{C}) @>{F_{|\Mor} \times F_{|\Mor}}>> \Mor (\mathcal{D}) \underset{\vartriangle}{\times} \Mor(\mathcal{D}) \\
@V{\Comp_\mathcal{C}}VV @V{\Comp_\mathcal{D}}VV \\
\Ob (\mathcal{C}) @>{F_{|\Ob}}>> \Ob (\mathcal{D})
\end{CD}$$
are all commutative, $\mathcal{C}_F$ is a topological category, and we call it the
category $\mathcal{C}$ \emph{with the topology induced by} $F$.

\begin{lemme}\label{L2nerve}
For every $n \in \mathbb{N}$, $\mathcal{N}(\mathcal{C}_F)_n$ has the initial topology for
$\mathcal{N}(F)_n: \mathcal{N}(\mathcal{C}_F)_n \rightarrow \mathcal{N}(\mathcal{D})_n$.
\end{lemme}

\begin{proof}
This is true by definition for $n\in \{0,1\}$. Let $n\in \mathbb{N} \setminus\{0,1\}$.
Then $\mathcal{N}(\mathcal{C}_F)_n \rightarrow \mathcal{N}(\mathcal{D})_n$ is continuous as a restriction of the continuous
map $\Mor(\mathcal{C}_F)^n \rightarrow \Mor(\mathcal{D})^n$. Let $Z$ be a topological space, and
$\alpha : Z \rightarrow \mathcal{N}(\mathcal{C}_F)_n$ be a map such that
the composite map $\beta : Z \rightarrow \mathcal{N}(\mathcal{C}_F)_n \rightarrow \mathcal{N}(\mathcal{D})_n$ is continuous.
Then all the corresponding maps $\beta_1,\dots,\beta_n$ from $Z$ to $\Mor(\mathcal{D})$ are continuous, and it follows that
the maps $\alpha_1,\dots,\alpha_n$ from $Z$ to $\Mor(\mathcal{C}_F)$ which correspond to $\alpha$ are continuous. Hence $\alpha$ is continuous.
This proves that $\mathcal{N}(\mathcal{C}_F)_n$ has the initial topology for $\mathcal{N}(F)_n$.
\end{proof}

\subsubsection{The topological category $\Func_{L^2}(\mathcal{E}G,\varphi \smod)$}

Let $S$ be a finite set, and $\varphi$ be an $S$-object of $\Gamma \text{-Fib}_F$.
\begin{Def}
We define
$$\alpha_\varphi :
\begin{cases}
\Ob(\Func(\mathcal{E}G,\varphi \smod)) & \longrightarrow C_*(G) \times \mathbb{N}^S \\
\mathbf{F}: \mathcal{E}G \rightarrow \varphi \smod & \longmapsto \begin{cases}
\left([g \mapsto \|\underset{s \in S}{\oplus}\mathbf{F}(1,g)_s\|],(\dim(\mathbf{F}(1)_s))_{s \in S}\right)
& \text{if} \quad
\exists s \in S: \dim(\mathbf{F}(1)_s)\neq 0 \\
(0,(0,\dots,0)) & \text{otherwise}.
\end{cases}
\end{cases}$$
The map $\alpha_\varphi$ is obviously continuous.
We then define $\Func_{L^2}(\mathcal{E}G,\varphi \smod)$ as the full subcategory of
$\Func(\mathcal{E}G,\varphi \smod)$ whose set of objects is $\alpha_\varphi^{-1}((L^2_*(G)\cap C_*(G)) \times \mathbb{N}^S)$.

We will now let
$\alpha_\varphi : \Ob(\Func_{L^2}(\mathcal{E}G,\varphi \smod))
\longrightarrow (C_*(G)\cap L^2_*(G)) \times \mathbb{N}^S$
denote the restriction of $\alpha_\varphi$.
Finally, we equip $\Func_{L^2}(\mathcal{E}G,\varphi \smod)$
with the topology induced by the functor
$$\Func_{L^2}(\mathcal{E}G,\varphi \smod)
\longrightarrow \Func(\mathcal{E}G,\varphi \smod) \times \mathcal{E}\bigl(
(C_*(G)\cap L^2_*(G)) \times \mathbb{N}^S
\bigr).$$
The topology on $\Func_{L^2}(\mathcal{E}G,\varphi \smod)$ is also induced by the functor
$$\Func_{L^2}(\mathcal{E}G,\varphi \smod)
\longrightarrow \Func(\mathcal{E}G,\varphi \smod) \times \mathcal{E}\bigl(
C_*(G) \times L^2_*(G)
\bigr).$$
We finally set:
$$s\vEc_{G,L^2}^\varphi:=\Big|\Func_{L^2}(\mathcal{E}G,\varphi \smod) \Big|.$$
\end{Def}

\subsubsection{The action of $G$ on $\Func_{L^2}(\mathcal{E}G,\varphi \smod)$}

Here, we prove that the action of $G$ on $\Func(\mathcal{E}G,\varphi \smod)$ induces a continuous action of $G$ on
$\Func_{L^2}(\mathcal{E}G,\varphi \smod)$.

\vskip 2mm
\noindent We first need to check that $\Func_{L^2}(\mathcal{E}G,\varphi \smod)$ is a stable subcategory for the action of
$G$ on $\Func(\mathcal{E}G,\varphi \smod)$.  The left-action of $G$ on $\Ob(\Func(\mathcal{E}G,\varphi \smod))$
gives rise to a commutative square
$$\begin{CD}
G \times \Ob(\Func(\mathcal{E}G,\varphi \smod)) @>{\id_G \times \alpha_\varphi}>>
G \times C_*(G) \\
@VVV @VVV \\
\Ob(\Func(\mathcal{E}G,\varphi \smod))  @>{\alpha_\varphi}>> C_*(G),
\end{CD}$$
in which the right-hand vertical map is
$$\begin{cases}
G \times C_*(G) & \longrightarrow C_*(G) \\
(g,f) & \longmapsto
\begin{cases}
0 & \quad  \text{if} \quad f=0 \\
\left[h \longmapsto \cfrac{f(hg)}{f(g)}\right] & \quad \text{otherwise}.
\end{cases}
\end{cases}$$
Since we are working with a right-invariant Haar measure on $G$, the image of
$G \times (C_*(G) \cap L^2_*(G))$ by this map is included in $C_*(G) \cap L^2_*(G)$.
It follows that $\Func_{L^2}(\mathcal{E}G,\varphi \smod)$ is a stable subcategory for the action of
$G$ on $\Func(\mathcal{E}G,\varphi \smod)$.

\vskip 2mm
\noindent
Moreover, the induced map $G \times (C_*(G) \cap L^2_*(G)) \longrightarrow C_*(G) \cap L^2_*(G)$
is continuous since the action of $G$ on $L^2(G)$ is continuous, and it follows that the
induced action of $G$ on
$\Func_{L^2}(\mathcal{E}G,\varphi \smod)$ is continuous.

\subsubsection{The functor $\Func_{L^2}(\mathcal{E}G,- \smod)$}

Given a morphism $f: \varphi \rightarrow \psi$ in
$\Gamma \text{-Fib}_F$, we want to check that the functor
$f^*:\Func(\mathcal{E}G,\varphi \smod) \longrightarrow \Func(\mathcal{E}G,\psi \smod)$
induces a continuous functor
$f_{L^2}^*: \Func_{L^2}(\mathcal{E}G,\varphi \smod) \longrightarrow
\Func_{L^2}(\mathcal{E}G,\psi \smod)$.

Let $\gamma=\mathcal{O}_\Gamma^F(f): S \rightarrow T$. Then $\gamma$ induces a map
$$\gamma^{\mathbb{N}}_*: \begin{cases}
\mathbb{N}^T & \rightarrow \mathbb{N}^S \\
(n_t)_{t \in T} & \mapsto \left(\underset{i \in \gamma(s)}{\sum}n_i\right)_{s \in S,}
\end{cases}$$
and finally a map $$\gamma_* : \begin{cases}
C_*(G) \times \mathbb{N}^T & \longrightarrow C_*(G) \times \mathbb{N}^S \\
(f,(n_t)_{t\in T}) & \longmapsto \begin{cases}
(0,0) & \text{if} \quad \gamma^{\mathbb{N}}_*((n_t)_{t\in T})=0 \\
(f,\gamma^{\mathbb{N}}_*((n_t)_{t\in T})) & \text{otherwise}.
\end{cases}
\end{cases}$$
Clearly $\gamma_*$ maps $(L^2_*(G) \cap C_*(G)) \times \mathbb{N}^T$ into
$(L^2_*(G) \cap C_*(G)) \times \mathbb{N}^S$, and its restriction
$(L^2_*(G) \cap C_*(G)) \times \mathbb{N}^T \longrightarrow (L^2_*(G) \cap C_*(G)) \times \mathbb{N}^S$
is obviously continuous.

\vskip 2mm
\noindent
Since the square
$$\begin{CD}
\Ob(\Func(\mathcal{E}G,\varphi \smod)) @>{\alpha_\varphi}>> C_*(G) \times \mathbb{N}^T \\
@V{f^*}VV @VV{\gamma_*}V \\
\Ob(\Func(\mathcal{E}G,\psi \smod)) @>{\alpha_\psi}>> C_*(G) \times \mathbb{N}^S
\end{CD}$$
is commutative, it follows that $f^*$ restricts to a continuous functor
$$f^*_{L^2} : \Func_{L^2}(\mathcal{E}G,\varphi \smod) \longrightarrow \Func_{L^2}(\mathcal{E}G,\psi \smod).$$
We have just constructed a functor $\Func_{L^2}(\mathcal{E}G, - \smod)$ from
$\Gamma \text{-Fib}_F$ to the category of topological categories. By composing it with the canonical functor
$\text{TopCat} \rightarrow \text{kCat}$, we recover a functor from $\Gamma \text{-Fib}_F$ to the category of k-categories, and
we still write it $\Func_{L^2}(\mathcal{E}G, - \smod)$.

\noindent \textbf{Remark :} If $G$ is compact, then all continuous functions on $G$ are square integrable and $C(G) \hookrightarrow
L^2(G)$ is continuous, therefore $\Func_{L^2}(\mathcal{E}G, - \smod)=\Func(\mathcal{E}G, - \smod)$.

\subsection{The $\Gamma-G$-space $s\vEc_{G,L^2}^\varphi$}\label{10.1.2}

\begin{Def}
Let $\varphi$ denote a Hilbert $\Gamma$-bundle. We define a contravariant functor from $\Gamma$ to $CG_G$ by:
$$s\underline{\vEc}_{G,L^2}^\varphi:
S \longmapsto |\Func_{L^2}(\mathcal{E}G,\varphi(S) \smod)|=s\vEc_{G,L^2}^{\varphi(S)}.$$
\end{Def}

We wish to prove that $s\underline{\vEc}_{G,L^2}^\varphi$ is a $\Gamma-G$-space.
In order to do this, we need the following lemma.

\begin{lemme}\label{L2transformations}
Let $S$ be a finite set, $\varphi$ and $\psi$ be two $S$-objects in $\Gamma \text{-Fib}_F$
together with  two morphisms $f: \varphi \rightarrow \psi$ and $g:\psi \rightarrow \varphi$
such that $\mathcal{O}_\Gamma^F(f)=\mathcal{O}_\Gamma^F(g)=\id_S$. \\
Then the natural transformations
$$\eta : \id_{\varphi \smod} \longrightarrow (g \smod)\circ (f \smod) \quad \text{and} \quad
\epsilon : \id_{\psi \smod} \longrightarrow (f \smod)\circ (g \smod) $$
from Corollary 3.3 of \cite{Ktheo1} induce, for every Lie group $G$, continuous
natural transformations
$$\eta^*: \id_{\Func_{L^2}(\mathcal{E}G,\varphi \smod)} \longrightarrow
g^*_{L^2} \circ f^*_{L^2}$$
and
$$\epsilon^*: \id_{\Func_{L^2}(\mathcal{E}G,\psi \smod)} \longrightarrow
f^*_{L^2} \circ g^*_{L^2}.$$
\end{lemme}

\begin{proof}
Since $\Func_{L^2}(\mathcal{E}G,\varphi \smod)$ is a full subcategory
of $\Func(\mathcal{E}G,\varphi \smod)$, it suffices to check that the natural transformations $\eta^*$ and
$\epsilon^*$ are continuous.
In the case of $\eta^*$, we only need to prove that the composite map
$$\Ob(\Func_{L^2}(\mathcal{E}G,\varphi \smod)) \overset{\eta^*}{\longrightarrow}
\Mor(\Func_{L^2}(\mathcal{E}G,\varphi \smod)) \overset{(\Fin,\In)}{\longrightarrow}
\Ob(\Func_{L^2}(\mathcal{E}G,\varphi \smod))^2$$
is continuous. This is a simple consequence of the fact that $g_{L^2}^* \circ f_{L^2}^*$ is continuous. The case
of $\epsilon^*$ may be treated similarly.
\end{proof}

\begin{prop}
For any Hilbert $\Gamma$-bundle $\varphi$, $s\underline{\vEc}_{G,L^2}^\varphi$ is a $\Gamma-G$-space.
\end{prop}

\begin{proof}
It is clear that $s\underline{\vEc}_{G,L^2}^\varphi(\mathbf{0})=*$. Let $n$ be a positive integer, and
$\varphi$ be a Hilbert $\Gamma$-bundle. We set $p:=\varphi(\mathbf{1})$.
The same strategy as in the proof of Proposition 3.1 of \cite{Ktheo1} yields three functors
$F^p: n.p \smod \longrightarrow (p \smod)^n$, $G^p: (p \smod)^n \longrightarrow n.p \smod$
and $H^p: (p \smod)^n \times I \longrightarrow (p \smod)^n$, which, in turn, induce three functors
$$F^p_{L^2}: \Func_{L^2}(\mathcal{E}G,n.p \smod) \longrightarrow (\Func(\mathcal{E}G,p \smod))^n,$$
$$G^p_{L^2}: (\Func_{L^2}(\mathcal{E}G,p \smod))^n \longrightarrow \Func(\mathcal{E}G,n.p \smod),$$
and
$$H^p_{L^2} : (\Func_{L^2}(\mathcal{E}G,p \smod))^n  \times I \longrightarrow
(\Func(\mathcal{E}G,p \smod))^n. $$
We need to prove that $F^p_{L^2}$ and $H^p_{L^2}$ (resp.\ $G^p_{L^2}$)
take values in $(\Func_{L^2}(\mathcal{E}G,p \smod))^n$
(resp.\ in $\Func_{L^2}(\mathcal{E}G,n.p \smod)$) and that they induce continuous functors
$$F^p_{L^2}: \Func_{L^2}(\mathcal{E}G,n.p \smod)
\longrightarrow (\Func_{L^2}(\mathcal{E}G,p \smod))^n,$$
$$G^p_{L^2}: (\Func_{L^2}(\mathcal{E}G,p \smod))^n \longrightarrow
\Func_{L^2}(\mathcal{E}G,n.p \smod),$$
and
$$H^p_{L^2} : (\Func_{L^2}(\mathcal{E}G,p \smod))^n  \times I \longrightarrow
(\Func_{L^2}(\mathcal{E}G,p \smod))^n.$$
This is however quite easy: for example, in the case of $F^p_{L^2}$, it suffices to consider the commutative square:
$$\begin{CD}
\Ob(\Func_{L^2}(\mathcal{E}G,n.p \smod)) @>{\alpha_{n.p}}>> C_*(G) \\
@V{F^p_{L^2}}VV @VVV \\
\Ob(\Func_{L^2}(\mathcal{E}G,p \smod))^n @>{(\alpha_p)^n}>> C_*(G)^n,
\end{CD}$$
where the right-hand vertical map is the diagonal map.

We then deduce that $|F^p_{L^2}|$ and $|G^p_{L^2}|$
are inverse one to the other up to equivariant homotopy, and it follows that the canonical map
$$|\Func_{L^2}(\mathcal{E}G,n.p \smod)| \longrightarrow
|\Func_{L^2}(\mathcal{E}G,p \smod)|^n$$ is an equivariant homotopy equivalence.

Using Lemma \ref{L2transformations} and the definition of a Hilbert $\Gamma$-bundle,
it is then easy to check that the canonical map
$$|\Func_{L^2}(\mathcal{E}G,\varphi(\mathbf{n}) \smod)| \longrightarrow
|\Func_{L^2}(\mathcal{E}G,n.\varphi(\mathbf{1}) \smod)|
$$
is an equivariant homotopy equivalence.

We conclude that the canonical map
$$|\Func_{L^2}(\mathcal{E}G,\varphi(\mathbf{n}) \smod)| \longrightarrow
|\Func_{L^2}(\mathcal{E}G,\varphi(\mathbf{1}) \smod)|^n
$$ is an equivariant homotopy equivalence.
\end{proof}

\begin{Def}Given a Hilbert $\Gamma$-bundle $\varphi$ and a second countable Lie group $G$, we define
$$\boxed{sKF_{G,L^2}^\varphi:=\Omega Bs\underline{\vEc}_{G,L^2}^\varphi.}$$
\end{Def}

\subsection{The morphism $s\vEc_{G,L^2}^\varphi \rightarrow
s\vEc_G^\varphi$ and its properties}\label{10.1.3}

\subsubsection{Main statements}

For every object $\psi$ in the category $\Gamma \text{-Fib}_F$, the inclusion of categories
defines a canonical continuous functor
$\Func_{L^2}(\mathcal{E}G,\psi \smod) \longrightarrow
\Func(\mathcal{E}G,\psi \smod)$. Thus, for every Hilbert-$\Gamma$-bundle
$\varphi$, we recover a morphism of equivariant $\Gamma$-spaces
$$s\underline{\vEc}_{G,L^2}^\varphi \longrightarrow s\underline{\vEc}_{G.}^\varphi$$

When $G$ is compact, this morphism is simply the identity.

\begin{prop}\label{L2=indif}
Let $p$ be a $1$-object of $\Gamma \text{-Fib}_F$.
Then, for any compact subgroup $H$ of $G$,
$$(s\vEc_{G,L^2}^p)^H \longrightarrow
(s\vEc_G^p)^H$$
is a homotopy equivalence.
\end{prop}

\begin{cor}\label{L2=indiff}
For every Hilbert $\Gamma$-bundle $\varphi$,
the morphism $s\underline{\vEc}_{G,L^2}^\varphi
\longrightarrow s\underline{\vEc}_G^\varphi$ induces a $G$-weak equivalence: $$sKF_{G,L^2}^\varphi \longrightarrow sKF_G^\varphi.$$
\end{cor}

\begin{proof}[Proof of Corollary \ref{L2=indiff}, assuming Proposition \ref{L2=indif} holds:]
Let $H$ be a compact subgroup of $G$ and $n$ be any non-negative integer.
The canonical commutative diagram
$$\begin{CD}
s\underline{\vEc}_{G,L^2}^\varphi(\mathbf{n}) @>>> s\underline{\vEc}_G^\varphi(\mathbf{n}) \\
@VVV @VVV \\
\Bigl(s\underline{\vEc}_{G,L^2}^\varphi(\mathbf{1})\Bigr)^n @>>> \Bigl(s\underline{\vEc}_G^\varphi(\mathbf{1})\Bigr)^n
\end{CD}$$
gives rise, by restriction to sets of points fixed by $H$, to a commutative diagram
$$\begin{CD}
\bigl(s\underline{\vEc}_{G,L^2}^\varphi(\mathbf{n})\bigr)^H @>>> \bigl(s\underline{\vEc}_G^\varphi(\mathbf{n})\bigr)^H \\
@VVV @VVV \\
\Bigl(\bigl(s\underline{\vEc}_{G,L^2}^\varphi(\mathbf{1})\bigr)^H\Bigr)^n @>>>
\Bigl(\bigl(s\underline{\vEc}_G^\varphi(\mathbf{1})\bigr)^H\Bigr)^n.
\end{CD}$$
By Proposition \ref{L2=indif} applied to $p=\varphi(\mathbf{1})$,
the canonical map $\bigl(s\vEc_{G,L^2}^{\varphi(\mathbf{1})}\bigr)^H
\longrightarrow \bigl(s\vEc_G^{\varphi(\mathbf{1})}\bigr)^H$
is a homotopy equivalence.
Hence $\Bigl(\bigl(s\underline{\vEc}_{G,L^2}^\varphi(\mathbf{1})\bigr)^H\Bigr)^n \longrightarrow
\Bigl(\bigl(s\underline{\vEc}_G^\varphi(\mathbf{1})\bigr)^H\Bigr)^n$ is a
homotopy equivalence.
Since $s\underline{\vEc}_{G,L^2}^\varphi$ and $s\underline{\vEc}_G^\varphi$
are both  $\Gamma-G$-spaces, we use the previous commutative diagram
to deduce that the canonical map
$\bigl(s\underline{\vEc}_{G,L^2}^\varphi(\mathbf{n})\bigr)^H \longrightarrow
\bigl(s\underline{\vEc}_G^\varphi(\mathbf{n})\bigr)^H$
is a homotopy equivalence.

Since this is true for any non-negative integer $n$, it follows (by
Proposition A.1 of \cite{Segal-cat}), that the canonical map
$$\bigl(Bs\underline{\vEc}_{G,L^2}^\varphi\bigr)^H \longrightarrow
\bigl(Bs\underline{\vEc}_G^\varphi\bigr)^H$$
is a homotopy equivalence. Finally, taking loop spaces shows that the map
$$\bigl(sKF_{G,L^2}^\varphi\bigr)^H \longrightarrow \bigl(sKF_G^\varphi\bigr)^H$$
is a homotopy equivalence.
\end{proof}

\subsubsection{The proof of Proposition \ref{L2=indif}}

We may assume that $G$ is non-compact. Let $H$ denote a compact subgroup of $G$.
In order to prove that $(s\vEc_{G,L^2}^p)^H \longrightarrow
(s\vEc_G^p)^H$ is a homotopy equivalence, we will construct a particular continuous functor
$$\Func(\mathcal{E}G,p \smod)
\longrightarrow \Func_{L^2}(\mathcal{E}G,p \smod)$$
that maps $H$-invariant functors to $H$-invariant functors.

\vskip 2mm
\noindent Here is the basic idea: given a continuous functor $F: \mathcal{E}G \rightarrow p \smod$, we want
to build a continuous functor $F': \mathcal{E}G \rightarrow p \smod$ such that
$g \mapsto \|F'(1_G,g)\|$ is square integrable. Since only the norm of $F'(1_G,g)$ matters, we only have to
multiply the map $g \mapsto F(1_G,g)$ by some ``pit map'' $\psi$ from $G$ to $\mathbb{R}_+^*$
so that $g \mapsto \psi(g)\|F(1_G,g)\|$ is square integrable. Of course, $\psi$ should depend continuously on
$F$, and should be $H$-invariant.

\vskip 2mm
\noindent We start by defining the functor on the sets of objects.
When $F:\mathcal{E}G \rightarrow p \smod$ is a functor, we define the \emph{dimension} of $F$ as $\dim F:=\dim(F(1_G))$.
We define $\Func^0(\mathcal{E}G,p \smod)$ (resp.\ $\Func^{>0}(\mathcal{E}G,p \smod)$)
as the full subcategory of $\Func(\mathcal{E}G,p \smod)$
whose space of objects consists of those functors of dimension $0$ (resp.\ of positive dimension).
Similarly, we define $\Func^0_{L^2}(\mathcal{E}G,p \smod)$ (resp.\ $\Func^{>0}_{L^2}(\mathcal{E}G,p \smod)$)
as the full subcategory of $\Func_{L^2}(\mathcal{E}G,p \smod)$
whose space of objects consists of those functors of dimension $0$ (resp.\ of positive dimension).

We easily see that
$$\Func^0_{L^2}(\mathcal{E}G,p \smod)=\Func^0(\mathcal{E}G,p \smod),$$
$$\Func(\mathcal{E}G,p \smod) =\Func^0(\mathcal{E}G,p \smod) \coprod \Func^{>0}(\mathcal{E}G,p \smod)$$
and
$$\Func_{L^2}(\mathcal{E}G,p \smod) =\Func^0_{L^2}(\mathcal{E}G,p \smod) \coprod \Func^{>0}_{L^2}(\mathcal{E}G,p \smod).$$
We then set $\Phi(\mathbf{F}):=\mathbf{F}$ for every $\mathbf{F} \in \Func^0(\mathcal{E}G,p \smod)$.

\vskip 2mm
\indent The Haar measure $\mu$ induces a measure on $G/H$
which is finite on compact subsets (i.e. $\sigma$-compact). Since $G/H$ is locally compact, there
exists a map $\lambda : G/H \rightarrow \mathbb{R}_+^*$ which is continuous and square integrable.
We choose such a map. The following lemma is then straightforward:

\begin{lemme}
The map
$$\beta : \begin{cases}
\Ob(\Func_{L^2}^{>0}(\mathcal{E}G,p \smod)) & \longrightarrow C(G,\mathbb{R}_+^*) \\
\mathbf{F} & \longmapsto \left[g \longmapsto \cfrac{\lambda(gH)}{\underset{h \in H}{\max} \|F(1_G,gh)\|}\right]
\end{cases}$$
is well defined and continuous.
\end{lemme}

Finally, we define a continuous map: $$\gamma : \begin{cases}
C(G,\mathbb{R}_+^*) \times
\Ob\left(\Func^{>0}(\mathcal{E}G,p \smod)\right) & \longrightarrow \Ob\left(\Func^{>0}(\mathcal{E}G,p \smod)\right) \\
(\psi,\mathbf{F}) & \longmapsto
\begin{cases}
g & \longmapsto \mathbf{F}(g) \\
(g,g') & \longmapsto \frac{\psi(g')}{\psi(g)}\cdot \mathbf{F}(g,g')
\end{cases}
\end{cases}$$
and consider the composite map:
\begin{multline*}
\Phi : \Ob\left(\Func^{>0}(\mathcal{E}G,p \smod)\right) \Right5{(\beta,\id)}
C(G,\mathbb{R}_+^*) \times \Ob\left(\Func^{>0}(\mathcal{E}G,p \smod)\right) \\
\overset{\gamma}{\longrightarrow} \Ob\left(\Func(\mathcal{E}G,p \smod)\right),
\end{multline*}
which is clearly continuous.

\vskip 2mm
Given a functor
$\mathbf{F}:\mathcal{E}G \rightarrow p \smod$ of positive dimension,
$\Phi(\mathbf{F})$ is an object of $\Func_{L^2}(\mathcal{E}G,p \smod)$.
Indeed, for every $g \in G$,
$$\|\Phi(\mathbf{F})(1_G,g)\|=
\frac{1}{\beta(\mathbf{F})(1_G)}\cfrac{\|\mathbf{F}(1_G,g)\|}{\underset{h \in H}{\max}\,\|\mathbf{F}(1_G,gh)\|}\,\lambda(gH)
\leq \frac{1}{\beta(\mathbf{F})(1_G)}\, \lambda(gH)$$
Since $\lambda$ is square integrable, $g \mapsto \|\Phi(\mathbf{F})(1_G,g)\|$ is square integrable.
Hence $\Phi(\mathbf{F}) \in \Ob(\Func_{L^2}(\mathcal{E}G,p \smod))$.

\vskip 2mm
\noindent
We have thus constructed a map
$$\Phi : \Ob\bigl(\Func(\mathcal{E}G,p \smod)\bigr)
\longrightarrow \Ob\bigl(\Func_{L^2}(\mathcal{E}G,p \smod)\bigr).$$

\begin{lemme}
The map $\Phi$ is continuous.
\end{lemme}

\begin{proof}
It suffices to prove that the composite map
$$\Ob\left(\Func^{>0}(\mathcal{E}G,p \smod)\right)
\overset{\Phi}{\longrightarrow} \Ob\left(\Func_{L^2}^{>0}(\mathcal{E}G,p \smod)\right) \overset{\alpha_p}{\longrightarrow} L^2(G)$$
is continuous. It thus suffices to establish the continuity of
$$\Phi_1 :\begin{cases}
C(G,\mathbb{R}_+^*) & \longrightarrow L^2(G,\mathbb{R}_+^*) \\
f & \longmapsto \left[g \mapsto \cfrac{f(g)}{\underset{h \in H}{\max}\,f(gh)}\,\lambda(gH)\right].
\end{cases}$$
Let $f \in C(G,\mathbb{R}_+^*)$ and $\epsilon >0$. Since $G/H$ is $\sigma$-compact, we may choose
a compact subset $K$ of $G$ such that $\int_{G \setminus K} \lambda(gH)\diff g \leq \epsilon$.
We already know that the map
$$\begin{cases}
C(G,\mathbb{R}_+^*) & \longrightarrow C(G,\mathbb{R}_+^*) \\
f & \longmapsto \left[g \mapsto \frac{f(g)}{\underset{h \in H}{\max}f(gh)} \lambda(gH)\right]
\end{cases}$$
Hence $$O_\epsilon :=\left\{f_1 \in C(G,\mathbb{R}_+^*):
\|\Phi_1(f_1)-\Phi_1(f)\|_\infty^K <\sqrt{\frac{\epsilon}{\mu(K)}} \right\}$$
is an open subset of $C(G,\mathbb{R}_+^*)$ which contains $f$. \\
For every $f_1\in O_\epsilon$,
$$\int_G \big|\Phi_1(f_1)-\Phi_1(f)\big|^2\diff \mu
\leq 2\left(\int_{G \setminus K}(\big|\Phi_1(f_1)\big|^2 +\big|\Phi_1(f)\big|^2)\diff \mu \right)
+ \int_K \big|\Phi_1(f_1)-\Phi_1(f)\big|^2 \diff \mu
\leq 5\epsilon.$$
It follows that $\Phi_1$ is continuous, hence $\Phi$ also is.
\end{proof}

We may now define $\Phi$ on morphisms by:
$$\Phi: \begin{cases}
\Mor(\Func(\mathcal{E}G,p \smod)) & \longrightarrow \Mor(\Func_{L^2}(\mathcal{E}G,p \smod)) \\
\mathbf{F} \overset{\eta}{\longrightarrow} \mathbf{F}' & \longmapsto
\begin{cases}
\left[g \longmapsto \cfrac{\beta(\mathbf{F}')[g]}{\beta(\mathbf{F})[g]}\cdot \eta_g \right] &
\text{if} \quad \dim \mathbf{F}>0 \\
[g \mapsto \eta_g] & \text{otherwise},
\end{cases}
\end{cases}$$
which yields a continuous functor
$$\Phi : \Func(\mathcal{E}G,p \smod) \longrightarrow
\Func_{L^2}(\mathcal{E}G,p \smod).$$

Since $\tilde{\beta}(x)$ is $H$-invariant for every $x\in \left]0,+\infty\right[^\mathbb{N}$, $\Phi$ induces a functor
$$\Phi^H: \Func(\mathcal{E}G,p \smod)^H
\longrightarrow \Func_{L^2}(\mathcal{E}G,p \smod)^H.$$
We define $i^H: \Func_{L^2}(\mathcal{E}G,p \smod)^H
\longrightarrow \Func(\mathcal{E}G,p \smod)^H$ as the functor induced by inclusion of categories.

\vskip 2mm
Finally, we define a continuous natural transformation
$\eta : \id_{\Func_{L^2}(\mathcal{E}G,p \smod)^H} \longrightarrow
\Phi^H \circ i^H$ by:
$$\eta: \mathbf{F} \longmapsto
\begin{cases}
 \left[g \mapsto \beta(\mathbf{F})[g].\id_{\mathbf{F}(g)} \right] & \text{if} \quad \dim(\mathbf{F})>0 \\
\left[g \mapsto \id_{\mathbf{F}(g)}\right] & \text{otherwise},
\end{cases}$$
and a continuous natural transformation
$\epsilon : \id_{\Func(\mathcal{E}G,p \smod)^H} \longrightarrow
i^H \circ \Phi^H$ by
$$\epsilon: \mathbf{F} \longmapsto
\begin{cases}
\left[g \mapsto \beta(\mathbf{F})[g].\id_{\mathbf{F}(g)}\right] & \text{if} \quad \dim(\mathbf{F})>0 \\
\left[g \mapsto \id_{\mathbf{F}(g)}\right] & \text{otherwise}.
\end{cases}
$$
We conclude that $|\Phi^H|$ is a homotopy-inverse of $|i^H|$. This finishes the proof of Proposition \ref{L2=indif}.

\subsection{A universal bundle over $|\Func_{L^2}(\mathcal{E}G,\varphi \smod)|$}\label{10.1.4}

\begin{Def}
Let $\varphi$ be a 1-object of $\Gamma \text{-Fib}_F$. We define
$\Func_{L^2}(\mathcal{E}G,\varphi \sBdl)$ as the full-subcategory
of $\Func(\mathcal{E}G,\varphi \sBdl)$ whose set of objects
is the inverse image of
$\Ob(\Func_{L^2}(\mathcal{E}G,\varphi \smod))$ by the canonical functor
$\Func(\mathcal{E}G,\varphi \sBdl) \longrightarrow
\Func(\mathcal{E}G,\varphi \smod).$ \\
We then equip $\Func_{L^2}(\mathcal{E}G,\varphi \sBdl)$ with the topological
category structure induced by the canonical functor
$$\Func_{L^2}(\mathcal{E}G,\varphi \sBdl)
\longrightarrow \Func_{L^2}(\mathcal{E}G,\varphi \smod)
\times \Func(\mathcal{E}G,\varphi \sBdl).$$

Similarly, if $n$ is a non-negative integer and $\varphi$ an $n$-dimensional $1$-object of $\Gamma \text{-Fib}_F$,
we define $\Func_{L^2}(\mathcal{E}G,\varphi \sframe)$ as the full subcategory
of $\Func(\mathcal{E}G,\varphi \sframe)$ whose set of objects
is the inverse image of $\Ob(\Func_{L^2}(\mathcal{E}G,\varphi \smod))$ by the canonical functor
$\Func(\mathcal{E}G,\varphi \sframe) \longrightarrow \Func(\mathcal{E}G,\varphi \smod)$.

We then equip $\Func_{L^2}(\mathcal{E}G,\varphi \sframe)$ with the topology induced by the canonical functor
$$\Func_{L^2}(\mathcal{E}G,\varphi \sframe)
\longrightarrow \Func_{L^2}(\mathcal{E}G,\varphi \smod)
\times \Func(\mathcal{E}G,\varphi \sframe).$$
For any $1$-object $\varphi$ of $\Gamma \text{-Fib}_F$, set:
$$Es\vEc_{G,L^2}^\varphi:=|\Func_{L^2}(\mathcal{E}G,\varphi \sBdl)|.$$
If $\varphi$ has dimension $n\in \mathbb{N}$, set:
$$s\widetilde{\vEc}_{G,L^2}^\varphi:=|\Func_{L^2}(\mathcal{E}G,\varphi \sframe)|.$$
\end{Def}
With those definitions, we find (compare with Theorem 2.6 of \cite{Ktheo1}):

\begin{prop}\label{theoL2}
Let $n\in \mathbb{N}$ and $\varphi$ be an $n$-dimensional $1$-object of $\Gamma \text{-Fib}_F$. Then:
\begin{itemize}
\item $s\widetilde{\vEc}_{G,L^2}^\varphi \rightarrow
s\vEc_{G,L^2}^\varphi$ is a $(G,\Sim_n(F))$-principal bundle;
\item $Es\vEc_{G,L^2}^\varphi \rightarrow s\vEc_{G,L^2}^\varphi$
is an $n$-dimensional $G$-simi-Hilbert bundle;
\item The morphism
$s\widetilde{\vEc}_{G,L^2}^\varphi \times _{\Sim_n(F)}F^n \rightarrow
Es\vEc_{G,L^2}^\varphi$ is an isomorphism of $G$-simi-Hilbert bundles.
\end{itemize}
\end{prop}

\begin{proof}
One needs to go back to the details of the proof of Theorem 2.6 of \cite{Ktheo1} adapted to the case of simi-Hilbert bundles.
The careful reader will then easily check that the corresponding maps
$\nu_m$, $\epsilon_m$, $\chi_m$, et $\psi_{g,m}$ are continuous in the present situation (the proof of this claim
relies upon Lemma \ref{L2nerve}).
The rest of the proof is identical to that of Theorem 2.6 of \cite{Ktheo1}.
\end{proof}

For $\varphi=\text{Fib}^{F^{(\infty)}}$,
set now
$$s\vEc_{G,L^2}^{F,\infty}:=s\vEc_{G,L^2}^{\varphi(\mathbf{1})} \quad \text{and} \quad
Es\vEc_{G,L^2}^{F,\infty}:=Es\vEc_{G,L^2.}^{\varphi(\mathbf{1})}$$
The canonical injection $\Func_{L^2}(\mathcal{E}G,\varphi(\mathbf{1}) \sBdl)
\rightarrow \Func(\mathcal{E}G,\varphi(\mathbf{1}) \sBdl)$
then induces a strong morphism of $G$-simi-Hilbert bundles
$$\begin{CD}
Es\vEc_{G,L^2}^{F,\infty} @>>> Es\vEc_G^{F,\infty} \\
@VVV @VVV \\
s\vEc_{G,L^2}^{F,\infty} @>>> s\vEc_G^{F,\infty}.
\end{CD}$$

It follows that, for every proper $G$-CW-complex $X$, the composite map

$$[X,s\vEc_{G,L^2}^{F,\infty}]_G \overset{\cong}{\longrightarrow}
[X,s\vEc_G^{F,\infty}]_G
\overset{\cong}{\longrightarrow} s\VEct_G^F(X)$$
is the one obtained by pulling-back the $G$-simi-Hilbert bundle
$Es\vEc_{G,L^2}^{F,\infty} \rightarrow s\vEc_{G,L^2}^{F,\infty}$. This last morphism is thus
systematically an isomorphism of monoids, which yields:

\begin{prop}
The $G$-simi-Hilbert bundle $Es\vEc_{G,L^2}^{F,\infty} \rightarrow s\vEc_{G,L^2}^{F,\infty}$ is
universal amongst $G$-simi-Hilbert bundles over proper $G$-CW-complexes.
\end{prop}

\section{The simplicial $G$-space $s\vEc_{G,L^2}^{F,\infty *}$}\label{10.2}

\subsection{Categorical background}\label{10.2.1}

\begin{Def}
Let $S$ be a totally ordered set.
A subset $I$ of $S$ is said to be an \textbf{interval} of $S$ when
$$\forall (x,y,z)\in I^2 \times S,\quad  x <z<y \Rightarrow z \in I.$$
For $(S_1,S_2)\in \mathcal{P}(S)^2$, we will write $S_1 <S_2$ when $\forall (x,y)\in S_1 \times S_2,\; x<y$, and
$S_1 \leq S_2$ when $\forall y \in S_2, \exists x \in S_1 : \, x \leq y$.
\end{Def}

\begin{Def}
Let $S$ and $T$ be two totally ordered sets, and $f:\mathcal{P}(S) \rightarrow \mathcal{P}(T)$ be a map.  \\
We say that $f$ is \textbf{increasing} when
$$\forall (S_1,S_2) \in \mathcal{P}(S)^2, \; S_1<S_2 \Rightarrow f(S_1) < f(S_2).$$
We say that $f$ is \textbf{hole-free} when $f(I)$ is an interval of $T$ for every interval $I$ of $S$.
\end{Def}

\textbf{Remarks :}
\begin{enumerate}[(i)]
\item If $S_1, S_2, S_3$ are totally ordered sets, and $f:\mathcal{P}(S_1) \rightarrow \mathcal{P}(S_2)$
and $g:\mathcal{P}(S_2) \rightarrow \mathcal{P}(S_3)$ are two maps, we easily see that
$g \circ f$ is increasing if $f$ and $g$ are both increasing, and
$g \circ f$ is hole-free if $f$ and $g$ are both hole-free.
\item If $S$ is a totally ordered finite set, $T$ is a totally ordered set, and $f: \mathcal{P}(S) \rightarrow \mathcal{P}(T)$
is a map which respects disjoint unions and such that $f(S)$ is finite, then
\begin{itemize}
\item $f$ is increasing if and only if $\forall (k,k')\in S^2, \; k<k' \Rightarrow f(\{k\}) < f(\{k'\})$;
\item in the case $f$ is increasing, it is hole-free if and only if
$f(S)=\emptyset$ or $f(S)=\{k\in T : \inf f(S) \leq k \leq \sup f(S)\}$.
\end{itemize}
\end{enumerate}

\begin{Def}\label{gammastar}
We define $\Gamma^*$ as the category whose
objects are the totally ordered finite sets, and whose morphisms are the morphisms of $\Gamma$
that are both increasing and hole-free. Forgetting the order structures gives rise to a ``forgetful'' functor
$\Gamma^* \rightarrow \Gamma$.
\end{Def}

The previous remarks make it then easy to check that the canonical
functor: $\Delta \rightarrow \Gamma$ induces a functor $\Delta \rightarrow \Gamma^*$.

\subsection{The simplicial $G$-space $s\vEc_{G,L^2}^{F,m*}$}\label{10.2.2}

\begin{Def}
For every totally ordered finite set $S$, we define
$\Gamma^*(S)$ as the subset of $\Gamma(S)$
consisting of those maps $\mathcal{P}(S) \rightarrow
\mathcal{P}(\mathbb{N})$ which are increasing and hole-free.
We define a structure of poset on $\Gamma^*(S)$ by
$$\forall (f,g)\in \Gamma^*(S)^2, \; f \leq g \underset{\text{def}}{\Longleftrightarrow} \bigl(\forall X \in \mathcal{P}(S),\; f(X) \leq g(X)\bigr).$$
\end{Def}

Every morphism $f:S \rightarrow T$ in $\Gamma^*$ induces a non-decreasing map
$f^*:\Gamma^*(T) \rightarrow \Gamma^*(S)$ by precomposition.

\vskip 2mm
Let $\mathcal{H}$ be an inner product space that is either finite-dimensional
or isomorphic to $F^{(\infty)}$.
For every object $S$ of $\Gamma^*$, we set
$$\text{Fib}^{\mathcal{H}*}(S):=
\left(S,\underset{f \in \Gamma^*(S)}{\coprod}\underset{s \in S}{\prod}
G_{f(s)}(\mathcal{H}),
\left(\underset{f \in \Gamma^*(S)}{\coprod}\left[p_{f(s)}
(\mathcal{H}) \times \underset{s' \in S\setminus\{s\}}{\prod}
\id_{G_{f(s')}(\mathcal{H})}\right]\right)_{s \in S}\right).$$
Since $\Gamma^*(S) \subset \Gamma(S)$, we may view $\text{Fib}^{\mathcal{H}*}(S)$ as a sub-object
of $\text{Fib}^{\mathcal{H}}(S)$ in the category $\Gamma \text{-Fib}_F$.
The morphism $\text{Fib}^{\mathcal{H}}(T) \longrightarrow \text{Fib}^{\mathcal{H}}(S)$
induced by a morphism $f : S \rightarrow T$ in $\Gamma^*$ then yields a morphism
$\text{Fib}^{\mathcal{H}*}(T) \longrightarrow \text{Fib}^{\mathcal{H}*}(S)$. This defines a contravariant functor
$\text{Fib}^{\mathcal{H}*}:\Gamma^* \longrightarrow \Gamma \text{-Fib}_F$. The canonical morphisms
$\text{Fib}^{\mathcal{H}*}(S) \longrightarrow \text{Fib}^{\mathcal{H}}(S)$ then induce a
natural transformation from $\text{Fib}^{\mathcal{H}*}$ to the composite of $\text{Fib}^{\mathcal{H}}:\Gamma \rightarrow \Gamma \text{-Fib}_F$
with the forgetful functor $\Gamma^* \rightarrow \Gamma$.

\vskip 2mm
\noindent
For every object $S$ of $\Gamma^*$, we equip the set of objects of $\text{Fib}^{\mathcal{H}*}(S) \smod$
with the following preorder: $$\forall (f,g)\in \Gamma^*(S)^2, \;\forall (x,y)\in \left(\underset{s \in S}{\prod}
G_{f(s)}(\mathcal{H})\right) \times \left(\underset{s \in S}{\prod} G_{g(s)}(\mathcal{H})\right), \quad x \leq y
\underset{\text{def}}{\Longleftrightarrow} f\leq g.$$
\begin{Def}
A morphism $x \overset{\varphi}{\rightarrow} y$ in $\text{Fib}^{\mathcal{H}*}(S) \smod$ is called \textbf{non-decreasing} when
$x \leq y$.
\end{Def}
It is then obvious that, for every morphism $\gamma : S \rightarrow T$ in $\Gamma^*$,
$\text{Fib}^{\mathcal{H}*}(\gamma) \smod$ maps the set of non-decreasing morphisms of $\text{Fib}^{\mathcal{H}*}(T) \smod$
into the set of non-decreasing morphisms of $\text{Fib}^{\mathcal{H}*}(S) \smod$.
\label{increasingfunc}
Finally, we let $\Func_{\uparrow L^2}\left(\mathcal{E}G,\text{Fib}^{\mathcal{H}*}(S) \smod \right)$
denote the subcategory of $\Func_{L^2}\left(\mathcal{E}G,\text{Fib}^{\mathcal{H}*}(S) \smod \right)$
which has the same space of objects and whose morphisms are the natural transformations that map $G$ into the set
of non-decreasing morphisms of $\text{Fib}^{\mathcal{H}*}(S)\smod$.

Composing $\text{Fib}^{\mathcal{H}*}$ with the canonical functor $\Delta \rightarrow \Gamma^*$, we
obtain a functor  $\Delta \rightarrow \Gamma \text{-Fib}_F$
which we will still denote by $\text{Fib}^{\mathcal{H}*}$.

Finally, we recover a simplicial $G$-space:
$$s\vEc_{G,L^2}^{\text{Fib}^{\mathcal{H}*}}: [n] \longmapsto
\left|\Func_{\uparrow L^2}\left(\mathcal{E}G,\text{Fib}^{\mathcal{H}*}([n]) \smod
\right)\right|.$$
With a slight abuse of notation, we will also denote by $s\vEc_{G,L^2}^{\text{Fib}^\mathcal{H}}$
the simplicial $G$-space deduced from the equivariant $\Gamma$-space $s\underline{\vEc}_{G,L^2}^{\text{Fib}^\mathcal{H}}$.
The previous natural transformation yields a morphism of simplicial $G$-spaces:
$s\vEc_{G,L^2}^{\text{Fib}^{\mathcal{H}*}} \longrightarrow s\vEc_{G,L^2}^{\text{Fib}^\mathcal{H}}$.
Our main result follows:

\begin{prop}\label{increasindif}
For every $n \in \mathbb{N}$,
$s\vEc_{G,L^2}^{\text{Fib}^{\mathcal{H}*}}[n]
\rightarrow s\vEc_{G,L^2}^{\mathcal{H}}[n]$
is an equivariant homotopy equivalence.
\end{prop}

\begin{proof} Let $n\in \mathbb{N}$ and $f\in \Gamma(\mathbf{n})$. If $f(\mathbf{n})=\emptyset$, we set $r_n(f):=f$. Otherwise, we define
$r_n(f)$ as the unique element of $\Gamma^*(\mathbf{n})$ such that
$0 \in r_n(f)(\mathbf{n})$ and  $\# r_n(f)(k)=\# f(k)$ for every $k\in \mathbf{n}$.
For every $k \in \mathbf{n}$, we then let $i_k^f$ denote the unique increasing map
from $f(k)$ to $r_n(f)(k)$. For every $k\in \mathbb{N}$, $i_k^f$ yields an isometry $\mathcal{H}^{f(k)} \longrightarrow \mathcal{H}^{r_n(f)(k)}$,
and these isometries give rise to a strong morphism of Hilbert bundles
$$\begin{CD}
E_{f(i)}(\mathcal{H}) \times \underset{k \in \mathbf{n}\setminus \{i\}}{\prod} G_{f(k)}(\mathcal{H})
@>>> E_{r_n(f)(i)}(\mathcal{H}) \times \underset{k \in \mathbf{n}\setminus \{i\}}{\prod} G_{r_n(f)(k)}(\mathcal{H}) \\
@VVV @VVV \\
\underset{k \in \mathbf{n}}{\prod} G_{f(k)}(\mathcal{H})  @>>> \underset{k \in \mathbf{n}}{\prod} G_{r_n(f)(k)}(\mathcal{H}) \\
\end{CD}$$
for every $i\in \mathbf{n}$.
This defines a morphism $r_n^{\mathcal{H}} : \text{Fib}^{\mathcal{H}}(\mathbf{n})
\rightarrow \text{Fib}^{\mathcal{H}*}(\mathbf{n})$
such that $\mathcal{O}_\Gamma^F(r_n^{\mathcal{H}})=\id_{\mathbf{n}}$ and
$r_n^{\mathcal{H}} \smod$ maps any morphism of $\text{Fib}^{\mathcal{H}}(\mathbf{n}) \smod$
to a non-decreasing morphism of $\text{Fib}^{\mathcal{H}*}(\mathbf{n}) \smod$.
Hence, $r_n^{\mathcal{H}}$ induces a $G$-map $s\vEc_{G,L^2}^{\text{Fib}^\mathcal{H}}[n]
\rightarrow s\vEc_{G,L^2}^{\text{Fib}^{\mathcal{H}*}}[n]$, and we now prove that it is an equivariant homotopy inverse of
the canonical map $s\vEc_{G,L^2}^{\text{Fib}^{\mathcal{H}*}}[n] \rightarrow s\vEc_{G,L^2}^{\text{Fib}^\mathcal{H}}[n]$.

\vskip 2mm
\noindent
Since the canonical map is induced by the inclusion of
$\Gamma^*(\mathbf{n})$ into $\Gamma(\mathbf{n})$, it is induced by a morphism
$\text{Fib}^{\mathcal{H}*}(\mathbf{n}) \rightarrow \text{Fib}^{\mathcal{H}}(\mathbf{n})$
over $\id_\mathbf{n}$, and we then deduce from Lemma \ref{L2transformations} that
the composite map $s\vEc_{G,L^2}^{\text{Fib}^\mathcal{H}}[n]
\rightarrow s\vEc_{G,L^2}^{\text{Fib}^{\mathcal{H}*}}[n] \rightarrow  s\vEc_{G,L^2}^{\text{Fib}^\mathcal{H}}[n]$ is $G$-homotopic
to the identity.

By construction, one has $r_n(f) \leq f$ for every $f \in \Gamma^*([n])$. We deduce that the inverse of the natural transformation
from $\id_{\Func_{L^2}\left(\mathcal{E}G,\text{Fib}^{\mathcal{H}*}([n]) \smod
\right)}$ to the composite functor $$\Func_{L^2}\left(\mathcal{E}G,\text{Fib}^{\mathcal{H}*}([n]) \smod \right)
\longrightarrow \Func_{L^2}\left(\mathcal{E}G,\text{Fib}^{\mathcal{H}}([n]) \smod\right)
\longrightarrow \Func_{L^2}\left(\mathcal{E}G,\text{Fib}^{\mathcal{H}}([n]) \smod\right)$$
obtained in Lemma \ref{L2transformations} and the proof of Corollary 3.3 in \cite{Ktheo1}
is really a natural transformation from
$$\Func_{\uparrow L^2}\left(\mathcal{E}G,\text{Fib}^{\mathcal{H}*}([n]) \smod \right)
\longrightarrow \Func_{L^2}\left(\mathcal{E}G,\text{Fib}^{\mathcal{H}}([n]) \smod\right)
\longrightarrow \Func_{\uparrow L^2}\left(\mathcal{E}G,\text{Fib}^{\mathcal{H}}([n]) \smod\right)$$
to $\id_{\Func_{\uparrow L^2}\left(\mathcal{E}G,\text{Fib}^{\mathcal{H}*}([n]) \smod
\right)}$. Hence the composite map
$s\vEc_{G,L^2}^{\text{Fib}^{\mathcal{H}*}}[n] \rightarrow s\vEc_{G,L^2}^{\text{Fib}^\mathcal{H}}[n] \rightarrow
s\vEc_{G,L^2}^{\text{Fib}^{\mathcal{H}*}}[n]$ is $G$-homotopic
to the identity map.
\end{proof}

\begin{Def}\label{ultimatevec}
For $m\in \mathbb{N}^* \cup \{\infty\}$,
we now set $s\vEc_{G,L^2}^{F,m *}:=s\vEc_{G,L^2}^{\text{Fib}^{F^{m}*}}$.
\end{Def}

\begin{cor}
The canonical map $\Omega Bs\vEc_{G,L^2}^{F,m*} \longrightarrow
\Omega Bs\underline{\vEc}_{G,L^2}^{F,m}$ is a $G$-weak equivalence.
\end{cor}

\begin{proof}
We deduce from Proposition \ref{increasindif} that, for every compact subgroup $H$ of $G$, \\
$\bigl(s\vEc_{G,L^2}^{F,m *}[n]\bigr)^H \longrightarrow \bigl(s\vEc_{G,L^2}^{F,m}[n]\bigr)^H$
is a homotopy equivalence for every non negative integer $n$. It follows then from proposition A.1 of \cite{Segal-cat}
that $\bigl(\Omega Bs\vEc_{G,L^2}^{F,m*}\bigr)^H \longrightarrow \bigl(\Omega Bs\underline{\vEc}_{G,L^2}^{F,m}\bigr)^H$ is a homotopy equivalence
for every compact subgroup $H$ of $G$. \end{proof}

\begin{cor}
The $G$-space $\Omega Bs\vEc_{G,L^2}^{F,\infty*}$ is a classifying space for
$KF^*_G(-)$ on the category of proper $G$-CW-complexes.
\end{cor}

\subsection{The universal bundle $Es\vEc_{G,L^2}^{F,\infty*} \longrightarrow s\vEc_{G,L^2}^{F,\infty*}[1]$}\label{10.2.3}

We denote by $\Func_{\uparrow L^2}\left(\mathcal{E}G,\text{Fib}^{\mathcal{H}*}([1]) \sBdl \right)$
the subcategory of $\Func_{L^2}\left(\mathcal{E}G,\text{Fib}^{\mathcal{H}*}([1]) \sBdl \right)$
which has the same space of objects, and whose morphisms are the natural transformations which, after composition with the canonical
functor $\text{Fib}^{\mathcal{H}*}([1]) \sBdl \longrightarrow \text{Fib}^{\mathcal{H}*}([1]) \smod$, map $G$ into the set
of non-decreasing morphisms of $\text{Fib}^{\mathcal{H}*}(\mathbf{1})\smod$.

\label{ultimatebdl}
We define $$Es\vEc_{G,L^2}^{F,\infty*}:=
\left|\Func_{\uparrow L^2}\left(\mathcal{E}G,\text{Fib}^{\mathcal{H}*}([1]) \sBdl \right)\right|.$$
Then $Es\vEc_{G,L^2}^{F,\infty*}$ is simply the inverse image of
$s\vEc_{G,L^2}^{F,\infty*}[1]$ by the canonical map $Es\vEc_{G,L^2}^{F,\infty} \longrightarrow s\vEc_{G,L^2}^{F,\infty}[1]$.
The canonical morphism $\text{Fib}^{F^{(\infty)}*}([1])
\rightarrow \text{Fib}^{F^{(\infty)}}([1])$ then induces a strong
morphism of $G$-simi-Hilbert bundles:
$$\begin{CD}
Es\vEc_{G,L^2}^{F,\infty*} @>>> Es\vEc_{G,L^2}^{F,\infty} \\
@VVV @VVV \\
s\vEc_{G,L^2}^{F,\infty*}[1] @>>> s\vEc_{G,L^2}^{F,\infty}.
\end{CD}$$
The following result follows readily:

\begin{prop}
For every proper $G$-CW-complex $X$, the map
$[X,s\vEc_{G,L^2}^{F,\infty*}[1]]_G \rightarrow s\VEct_G^F(X)$ induced by pulling back
the $G$-simi-Hilbert bundle $Es\vEc_{G,L^2}^{F,\infty*} \rightarrow s\vEc_{G,L^2}^{F,\infty*}[1]$
is the composite map
$$\bigl[X,s\vEc_{G,L^2}^{F,\infty*}[1]\bigr]_G \overset{\cong}{\longrightarrow}
\bigl[X,s\vEc_{G,L^2}^{F,\infty}\bigr]_G \overset{\cong}{\longrightarrow} s\VEct_G^F(X).$$
\end{prop}

Finally, denote by $\varphi_n$ the $n$-dimensional part of the Hilbert bundle $\text{Fib}^{F^{(\infty)}*}([1])$,
and define
$\Func_{\uparrow L^2}\left(\mathcal{E}G,\varphi_n \sframe \right)$ as the subcategory of
$\Func_{L^2}\left(\mathcal{E}G,\varphi_n \sframe \right)$
which has the same space of objects and whose morphisms are the natural transformations which, after composition with the canonical
functor $\varphi_n \sframe \longrightarrow \text{Fib}^{\mathcal{H}*}([1]) \smod$, map $G$ into the set
of non-decreasing morphisms of $\text{Fib}^{\mathcal{H}*}([1])\smod$.

\label{ultimateframe}
Set $s\widetilde{\vEc}_{G,L^2}^{n,F,\infty*}:=|\Func_{\uparrow L^2}\left(\mathcal{E}G,\varphi_n \sframe \right)|$,
and notice that $s\widetilde{\vEc}_{G,L^2}^{n,F,\infty*}$ is simply the inverse image of
$\left|\Func_{\uparrow L^2}\left(\mathcal{E}G,\text{Fib}^{\mathcal{H}*}([1])\smod \right)\right|
\cap s\vEc_{G,L^2}^{\varphi_n}$ under the identification
map $s\widetilde{\vEc}_{G,L^2}^{\varphi_n} \longrightarrow  s\vEc_{G,L^2}^{\varphi_n}$. We conclude that
$\underset{n=0}{\overset{\infty}{\coprod}}s\widetilde{\vEc}_{G,L^2}^{n,F,\infty*} \longrightarrow
s\vEc_{G,L^2}^{F,\infty*}[1]$ is an identification map.

\vskip 2mm
\noindent Denote finally by $Es\vEc_{G,L^2}^{n,F,\infty*}$ the inverse image of
$\big|\Func_{\uparrow L^2}\left(\mathcal{E}G,\text{Fib}^{\mathcal{H}*}([1])\smod
\right)\big| \cap s\vEc_{G,L^2}^{\varphi_n}$ under the $G$-vector bundle
map $E\vEc_{G,L^2}^{\varphi_n} \longrightarrow s\vEc_{G,L^2}^{\varphi_n}$.
Proposition \ref{theoL2} then yields that the canonical morphism
$s\widetilde{\vEc}_{G,L^2}^{n,F,\infty*} \times _{\Sim_n(F)}F^n \rightarrow Es\vEc_{G,L^2}^{n,F,\infty*}$
is an isomorphism of $G$-simi-Hilbert bundles.

\section{The simplicial $G$-space $(\Fred(G,\mathcal{H})[n])_{n \in \mathbb{N}}$}\label{10.3}

\subsection{The $G$-space $\Fred(L^2(G,\mathcal{H}))$}\label{10.3.1}

Let $\mathcal{H}$ be a separable Hilbert space.
The map
$$(g,f) \longmapsto \left[x \mapsto f(xg)\right]$$
classically defines a structure of Hilbert $G$-module on $L^2(G,\mathcal{H})$.
It follows that
$$\begin{cases}
G \times \mathcal{B}(L^2(G,\mathcal{H})) & \longrightarrow \mathcal{B}(L^2(G,\mathcal{H})) \\
(g,f) & \longmapsto \left[x \mapsto g.f(g^{-1}.x) \right]
\end{cases}$$
defines an action of the group $G$ on the space $\mathcal{B}(L^2(G,\mathcal{H}))$ of
bounded operators on $L^2(G,\mathcal{H})$, but this action is not continuous \emph{a priori}
(if $G$ is non-discrete and $\mathcal{H} \neq \{0\}$, it may actually be shown that this action is non-continuous).

\vskip 2mm
\noindent
For every $g \in G$ and $f \in \mathcal{B}(L^2(G,\mathcal{H}))$, notice that
$$\Ker(g.f)=g.\Ker(f) \quad \text{and} \quad \im(g.f)=g.\im(f).$$
The previous action thus induces an action of the group $G$ on the space $\Fred(L^2(G,\mathcal{H}))$ of Fredholm
operators on $L^2(G,\mathcal{H})$ (again, a non-continuous action \emph{a priori}).
Denote by $\underline{\Fred}(L^2(G,\mathcal{H}))$ the subset of $\Fred(L^2(G,\mathcal{H}))$ consisting
of the operators $F$ such that $g \mapsto g.F$ is a continuous map.
Since $F \mapsto g.F$ is an isometry of $\Fred(L^2(G,\mathcal{H}))$ for every $g \in G$,
it may easily be shown that $\underline{\Fred}(L^2(G,\mathcal{H}))$ is a $G$-space.

\subsection{The simplicial $G$-space
$(\Fred(G,\mathcal{H})[n])_{n \in \mathbb{N}}$}\label{10.3.2}

Let $g\in G$ and $(f,f')\in \mathcal{B}(L^2(G,\mathcal{H}))$. Then  $g.(f\circ f')=(g.f)\circ (g.f')$.
It follows that the composition of operators in $\mathcal{B}(L^2(G,\mathcal{H}))$
induces a (continuous) composition law in $\underline{\Fred}(L^2(G,\mathcal{H}))$.
Therefore, $\underline{\Fred}(L^2(G,\mathcal{H}))$ has a structure of
(non-abelian) topological monoid for this composition law, and $G$ acts by morphisms on $\underline{\Fred}(L^2(G,\mathcal{H}))$ .

We then set
$$(\Fred(G,\mathcal{H})[n])_{n \in \mathbb{N}}:=\mathcal{N}\left(\mathcal{B}
\underline{\Fred}(L^2(G,\mathcal{H}))\right),$$
where the nerve is constructed in the category of k-spaces.

Thus, for every $n \in \mathbb{N}$, we have
$$\Fred(G,\mathcal{H})[n]=\underline{\Fred}(L^2(G,\mathcal{H}))^n,$$ with
the usual face and degeneracy maps: this is obviously a k-space.

\begin{prop}\label{fredweak}
The canonical map
$$\underline{\Fred}(L^2(G,\mathcal{H})) \longrightarrow
\Omega B\Fred(G,\mathcal{H})$$
is a $G$-weak equivalence.
\end{prop}

\begin{proof}
Let $H$ denote a compact subgroup of $G$.
For every $g \in G$ and $f \in \mathcal{B}(L^2(G,\mathcal{H}))$, one has
$(g.f)^*=g.f^*$. It follows that the adjunction map induces
a continuous self isometry on $\underline{\Fred}(L^2(G,\mathcal{H}))$. The maps
$$h_1: (t,f) \longmapsto t.\id_{L^2(G,\mathcal{H})}+(1-t)f\circ f^* \quad \text{and} \quad
h_2: (t,f) \longmapsto t.\id_{L^2(G,\mathcal{H})}+(1-t)f^*\circ f.$$
then yield equivariant homotopies respectively from
$\id_{\underline{\Fred}(L^2(G,\mathcal{H}))}$ to $f\circ f^*$,
and from $\id_{\underline{\Fred}(L^2(G,\mathcal{H}))}$ to $f^*\circ f$.
We deduce that the topological monoid $\underline{\Fred}(L^2(G,\mathcal{H}))^H$ has an inverse up to homotopy.

Since $\Fred(G,\mathcal{H})^H$ is a simplicial k-space such that
$$\forall n \in \mathbb{N}, \; \Fred(G,\mathcal{H})[n]^H
\overset{\cong}{\longrightarrow} \left(\Fred(G,\mathcal{H})[1]^H\right)^n, $$
and $\Fred(G,\mathcal{H})[1]^H$ is an H-space
with an inverse up to homotopy, we deduce from proposition 1.5 of \cite{Segal-cat}
that the canonical map $\Fred(G,\mathcal{H})[1]^H \longrightarrow \Omega B\Fred(G,\mathcal{H})^H$
is a homotopy equivalence. Hence $\underline{\Fred}(L^2(G,\mathcal{H})) \longrightarrow \Omega B\Fred(G,\mathcal{H})$
is a $G$-weak equivalence. \end{proof}

\subsection{From finite dimensional subspaces of $L^2(G,\mathcal{H})$ to Fredholm operators on
$L^2(G,\mathcal{H}^\infty)$: the shift map}\label{10.3.3}

Recall that when $\mathcal{H}$ is a Hilbert space,
$\mathcal{H}^\infty$ is defined as the Hilbert space completion of the inner product space $\mathcal{H}^{(\infty)}$.
Recall also the canonical isomorphism $L^2(G,\mathcal{H}^\infty) \cong L^2(G,\mathcal{H})^\infty$.

\subparagraph{}We will now let $H$ denote a Hilbert $G$-module. The case of interest is
$H=L^2(G,\mathcal{H})$.

We write $H^\infty=\overline{\underset{i\in \mathbb{N}}{\oplus} (H\times \{i\})}$.
We define the shift operator on $H^\infty$ as
$$S_H : (x_i,i)_{i\in \mathbb{N}} \longmapsto (x_{i+1},i)_{i\in \mathbb{N}.}$$
Clearly, $S_H$ is a bounded linear operator of norm $1$, it is onto and its kernel is $H \times \{0\}$. \\
Let $V$ be a closed linear subspace of $H$. Then $H^\infty=V^\infty \overset{\bot}{\oplus} (V^\bot)^\infty$, and
we can define $\Shift(V)$ as the linear operator on $H^\infty$
such that
$$\begin{cases}
\forall x \in V^\infty, & \Shift(V)[x]=S_V(x) \\
\forall x \in (V^\bot)^\infty, & \Shift(V)[x]=x
\end{cases}$$
i.e.\ $\Shift(V)$ should be thought of as the shift alongside $V^\infty$ in $H^\infty$.
Obviously, $\Shift(V)$ is a bounded operator of norm $1$. Also, $\Shift(V)$ is onto and $\Ker(\Shift(V))=V \times \{0\}$.
We have thus defined a map
$$\Shift: \begin{cases}
\sub(H) & \longrightarrow \mathcal{B}(H^\infty) \\
V & \longmapsto \Shift(V).
\end{cases}$$
The following properties are then classical and easily checked:

\begin{prop}\label{propertiesofshift} ${}$ \\
\begin{enumerate}[(i)]
\item The map $\Shift$ is a $G$-map.
\item For every pair $(V,V')\in \sub(H)^2$ such that $V \bot V'$,
$$\Shift(V \oplus V')=\Shift(V) \circ \Shift(V')=\Shift(V') \circ \Shift(V).$$
\end{enumerate}
\end{prop}

\begin{cor}
The map $\Shift$ induces a $G$-map
$$\Shift: \underset{n\in \mathbb{N}}{\bigcup}\sub_n(H) \longrightarrow \underline{\Fred}(H^\infty).$$
\end{cor}

\begin{proof}  Let $V \in \underset{n\in \mathbb{N}}{\bigcup}\sub_n(H)$. Then $\Shift(V)$ is surjective
and its kernel is $V \times \{0\}$, which is a finite-dimensional space. It follows that $\Shift(V)$ is a Fredholm operator on
$H^\infty$. Hence $\Shift$ induces a continuous equivariant map
$\underset{n\in \mathbb{N}}{\bigcup}\sub_n(H) \longrightarrow \Fred(H^\infty)$. However
$\underset{n\in \mathbb{N}}{\bigcup}\sub_n(H)$ is a $G$-space,
hence, for every $V \in \underset{n\in \mathbb{N}}{\bigcup}\sub_n(H)$, one recovers that $g \mapsto g.\Shift(V)$
is continuous on $G$, which shows that $\Shift$ maps $\underset{n\in \mathbb{N}}{\bigcup}\sub_n(H)$ into $\underline{\Fred}(H^\infty)$.
\end{proof}

In particular, for every separable Hilbert space $\mathcal{H}$, we have constructed a $G$-map
$$\Shift : \underset{n\in \mathbb{N}}{\bigcup}\sub_n(L^2(G,\mathcal{H})) \longrightarrow \underline{\Fred}(L^2(G,\mathcal{H}^\infty)).$$

\section{Construction of a morphism $sK_{G,L^2}^{F,\infty} \rightarrow \underline{\Fred}(L^2(G,\mathcal{H})^\infty)$}

\subsection{Main ideas}
\label{10.4}

The purpose of this short section is to help motivate the very technical constructions that
the reader will have to face in our construction of a ``good" morphism
$sK_{G,L^2}^{F,\infty} \rightarrow \underline{\Fred}(L^2(G,\mathcal{H}^\infty))$ in the category
$CG_G^{h\bullet}[W_G^{-1}]$ (where $\mathcal{H}$ is a separable Hilbert space which contains $F^{(\infty)}$, to be defined
later on).

First of all, we know that the canonical maps $\Omega Bs\vEc_{G,L^2}^{F,\infty*} \rightarrow sKF_G^{[\infty]}$
and $\underline{\Fred}(G,\mathcal{H}^\infty) \longrightarrow
\Omega B\Fred(G,\mathcal{H}^\infty)$ are $G$-weak equivalences, so all we need is 
 ``good" $G$-map from $\Omega Bs\vEc_{G,L^2}^{F,\infty*}$ to $\Omega B\Fred(G,\mathcal{H}^\infty)$.
Such a map will of course be obtained by constructing a morphism of hemi-simplicial $G$-spaces
from $s\underline{\vEc}_{G,L^2}^{F,\infty*}$ to $(\Fred(G,\mathcal{H}^\infty)[n])_{n \in \mathbb{N}}$, i.e.
a collection of $G$-maps
$s\vEc_{G,L^2}^{F,\infty*}(\mathbf{n}) \longrightarrow \underline{\Fred}(L^2(G,\mathcal{H}^\infty))^n$
which are compatible with the face maps\footnote{Note that only the hemi-simplicial space structure is taken into
account since the thick realization only uses face maps.}.
Assume that a $G$-map
$$\alpha_G : s\vEc_{G,L^2}^{F,\infty*}(\mathbf{1}) \longrightarrow \underline{\Fred}(L^2(G,\mathcal{H}^\infty))$$
has been obtained at the level $1$.
For every $n \in \mathbb{N}$, we have a canonical $G$-map $s\underline{\vEc}_{G,L^2}^{F,\infty}(\mathbf{n})
\longrightarrow (s\underline{\vEc}_{G,L^2}^{F,\infty}(\mathbf{1}))^n$. Composing it with
$\alpha_G^n : (s\underline{\vEc}_{G,L^2}^{F,\infty}(\mathbf{1}))^n \longrightarrow
(\underline{\Fred}(L^2(G,\mathcal{H}^\infty)))^n$ yields a $G$-map
$$s\underline{\vEc}_{G,L^2}^{F,\infty}(\mathbf{n}) \longrightarrow
(\underline{\Fred}(L^2(G,\mathcal{H}^\infty)))^n.$$
Elementary considerations on simplicial sets show that those maps yield a morphism of hemi-simplicial $G$-spaces
if and only if the diagram
\begin{equation}\label{comdiagcond}
\xymatrix{
s\underline{\vEc}_{G,L^2}^{F,\infty}(\mathbf{2}) \ar[d]_{d_1^2} \ar[r] &
(s\vEc_{G,L^2}^{F,\infty})^2 \ar[r]^(0.45){(\alpha_G)^2} & (\underline{\Fred}(L^2(G,\mathcal{H})^\infty))^2
\ar[d]^{\circ} \\
s\vEc_{G,L^2}^{F,\infty} \ar[rr]^{\alpha_G} & & \underline{\Fred}(L^2(G,\mathcal{H})^\infty)
}
\end{equation}
is commutative.

\vskip 2mm
Let us now see how $\alpha_G$ should be constructed.
Let $\varphi:=\text{Fib}^{F^{(\infty)}*}(\mathbf{1})$.
Recall that the objects of the category $\varphi \smod$ are the pairs consisting of
a finite \emph{interval} $A$ of $\mathbb{N}$ (the ``label''), and a $\#A$-dimensional subspace of $(F^{(\infty)})^A$, which may
of course be seen as a subspace of $\mathcal{H}:=(F^{(\infty)})^{(\mathbb{N})}$.
Let $\mathbf{F} : \mathcal{E}G \longrightarrow \text{Fib}^{F^{(\infty)}}(\mathbf{1}) \smod$
be a square integrable continuous functor, i.e.\ an object of $\Func_{\uparrow L^2}(\mathcal{E}G,\text{Fib}^{F^{(\infty)}}(\mathbf{1})\smod)$.
If we choose a basis $(e_1,\dots,e_n)$ of $\mathbf{F}(1_G)$, then the maps
$\begin{cases}
G & \longrightarrow (F^{(\infty)})^{(\mathbb{N})} \\
g & \longmapsto \mathbf{F}(1_G,g)[e_i]
\end{cases}$, for $i \in \{1,\dots,n\}$,
define an  $n$-tuple of linearly independent elements of
$L^2(G,\mathcal{H})$. In turn, this $n$-tuple defines an
$n$-dimensional subspace of $L^2(G,\mathcal{H})$ which does not
depend on the choice of $(e_1,\dots,e_n)$, and this subspace yields a Fredholm operator on $L^2(G,\mathcal{H})^\infty$ by
the usual shift construction (cf.\ Section \ref{10.3.3}). This
procedure defines a map from the set of $0$-simplices in
$s\vEc_{G,L^2}^{F,\infty*}=\big|\Func(\mathcal{E}G,\text{Fib}^{F^{(\infty)}}(\mathbf{1})\smod)\big|$
to the $G$-space $\underline{\Fred}(L^2(G,\mathcal{H})^\infty)$.

\vskip 2mm
So far, we have associated to every $0$-simplex in the geometric realization
$s\vEc_{G,L^2}^{F,\infty*}$ a finite-dimensional subspace of
$L^2(G,\mathcal{H})$, which itself defines a Fredholm operator on
$L^2(G,\mathcal{H})^\infty$. We now want to build a
map: $$\Mor\left(\Func_{\uparrow L^2}(\mathcal{E}G,\text{Fib}^{F^{(\infty)}}(\mathbf{1}) \smod)\right)
\times [0,1] \longrightarrow \underset{n=0}{\overset{\infty}{\bigcup}}\sub_n(L^2(G,\mathcal{H}))$$
which is compatible with the preceding one. Assume first that $G$ is trivial. Let $x$ be a morphism of $\Func_{\uparrow L^2}(\mathcal{E}\{1\},\text{Fib}^{F^{(\infty)}}(\mathbf{1}) \smod)=\text{Fib}^{F^{(\infty)}}(\mathbf{1}) \smod$. Then $x$ is a $5$-tuple consisting
of two finite intervals $A$ and $B$ of $\mathbb{N}$, a $\#A$-dimensional subspace $x_0$ of $(F^{(\infty)})^A$, a $\#B$-dimensional
subspace $x_1$ of $(F^{(\infty)})^B$, and a similarity $\varphi : x_0 \overset{\cong}{\rightarrow} x_1$. We set $n:=\#A=\#B$.

Let $\tilde{x}$ be a morphism
of $\varphi_n \sframe$ in the fiber of $x$. Then $\tilde{x}=(\mathbf{B}_0,\mathbf{B}_1)$, where
$\mathbf{B}_0$ is a simi-orthonormal basis of $x_0$, and $\mathbf{B}_1$ is a simi-orthonormal basis of
$x_1$. We may write $\mathbf{B}_0=\lambda_0.\mathbf{B}'_0$ and $\mathbf{B}_1=\lambda_1.\mathbf{B}'_1$, where
$(\lambda_0,\lambda_1) \in (\mathbb{R}_+^*)^2$, $\mathbf{B}'_0$ is an orthonormal basis of $x_0$,
and $\mathbf{B}'_1$ is an orthonormal basis of $x_1$.
The basic idea is to create a path from $\mathbf{B}'_0$ to $\mathbf{B}'_1$ in the space of
orthonormal $n$-tuples of elements of $\mathcal{H}$.

\begin{figure}[h]
\begin{center}
\psfrag{A}{$(F^{(\infty)})^A$}
\psfrag{B}{$(F^{(\infty)})^B$}
\psfrag{A1}{$\mathbf{B}'_0$}
\psfrag{B1}{$\mathbf{B}'_1$}
\includegraphics{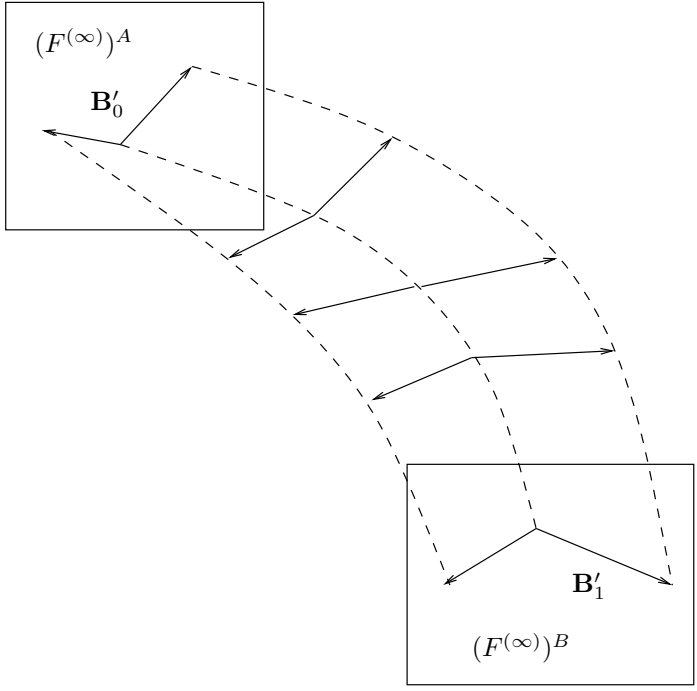}
\caption{A path from $\mathbf{B}'_0$ to $\mathbf{B}'_1$}
\end{center}
\end{figure}
Once we have a such a path, we may multiply it with the affine path from $\lambda_0$ to $\lambda_1$ in $\R_+^*$
in order to recover a path from $\mathbf{B}_0$ to $\mathbf{B}_1$ in the space of
simi-orthonormal $n$-tuples of elements of $\mathcal{H}$. This path would yield a continuous map
$\{x\} \times [0,1] \longrightarrow \sub_n(\mathcal{H})$. Of course, this map
should not depend on the choice of $\tilde{x}$ in the fiber of $x$. It should also be constant in the case $A=B$ and
$\mathbf{B}_0'=\mathbf{B}_1'$ (since we need compatibility with the degeneracy maps),
it should remain in $(F^{(\infty)})^A$ when $A=B$, and it should be continuous with respect to $x$.

\vskip 2mm
The construction is no more complicated in the case $G$ is non trivial.
In this case, a morphism of $\Func_{\uparrow L^2}(\mathcal{E}G,\varphi_n \sframe)$ is a family of morphisms
of $\varphi_n \sframe$ indexed over $G$: each of these morphisms will yield
a path in the space of simi-orthonormal $n$-tuples of elements of $\mathcal{H}$: we recover a family indexed over $G$ of
paths in the space of simi-orthonormal $n$-tuples of elements of $\mathcal{H}$; this can be seen as a path
in the space of maps from $G$ to the space of simi-orthonormal $n$-tuples of elements of $\mathcal{H}$. Finally, a map
from $G$ to the space of simi-orthonormal $n$-tuples of elements of $\mathcal{H}$ defines, under certain conditions, an $n$-dimensional
subspace of $L^2(G,\mathcal{H})$.

With similar ideas, it is quite easy to see how one can expect to construct a reasonable map
$$\mathcal{N}\left(
\Func_{\uparrow L^2}(\mathcal{E}G,\text{Fib}^{F^{(\infty)}}(\mathbf{1}) \smod)\right)_m
\times \Delta^m \longrightarrow \underset{n=0}{\overset{\infty}{\bigcup}}\sub_n(L^2(G,\mathcal{H}))$$
for every $m \in \mathbb{N}$.

\vskip 2mm
Assume now that we have a $G$-map $s\vEc_{G,L^2}^{F,\infty *} \longrightarrow
\underset{n\in \mathbb{N}}{\bigcup}\sub_n(L^2(G,\mathcal{H}))$ which extends the previous construction
already done on the $0$-simplices, and sketched on the $1$-skeleton of $s\vEc_{G,L^2}^{F,\infty *}$.
By composition with the shift map
$\Shift : \underset{n\in \mathbb{N}}{\bigcup}\sub_n(L^2(G,\mathcal{H})) \longrightarrow
\underline{\Fred}(L^2(G,\mathcal{H})^\infty)$ (cf.\ Section \ref{10.3}), we recover the $G$-map
$\alpha_G : s\vEc_{G,L^2}^{F,\infty *} \longrightarrow \underline{\Fred}(L^2(G,\mathcal{H})^\infty)$
we were looking for.

\vskip 2mm
Let us now look at the condition on $\alpha_G$ imposed by the commutativity of diagram
\eqref{comdiagcond}.
Assume $G$ is trivial for sake of simplicity.
Let $(n,m) \in \mathbb{N}^2$ and
$(x_1,x_2,y_1,y_2)\in (F^{(\infty)})^{\{n\}} \times (F^{(\infty)})^{\{n+1\}} \times
(F^{(\infty)})^{\{m\}} \times (F^{(\infty)})^{\{m+1\}}$ be a $4$-tuple of unit vectors. By the previous construction, we have a path from $(n,x_1)$ to $(m,y_1)$ and from $(n+1,x_2)$ to $(m+1,y_2)$. For the previous square to be commutative, it is necessary
that the path from $(\{n,n+1\},(x_1,x_2))$ to $(\{m,m+1\},(y_1,y_2))$ in the previous construction
be precisely obtained by juxtaposing the two previous paths.
In particular, this means that the two paths should be ``orthogonal'' at every step.
Even taking those conditions as granted, the preceding requirement will not be fulfilled by
every choice of paths. The key idea now is that,
when $m>n$, we create a new inner product space, labeled $\mathcal{H}_{n,m}$,
which is isomorphic to $(F^{(\infty)})^{\{n\}} \oplus (F^{(\infty)})^{\{m\}}$,
and we choose an isomorphism $\varphi : (F^{(\infty)})^{\{n\}} \oplus (F^{(\infty)})^{\{m\}} \overset{\cong}{\longrightarrow}
\mathcal{H}_{n,m}$.
Instead of creating a path from $x_1$ to $y_1$ in $(F^{(\infty)})^{\{n,m\}}$, we
create a path in the orthogonal direct sum $(F^{(\infty)})^{\{n,m\}} \overset{\bot}{\oplus} \mathcal{H}_{n,m}$ as follows: \begin{itemize}
\item First of all, we use a rotation to go from $x_1$ to $\varphi(x_1)$ in the orthogonal direct sum
$(F^{(\infty)})^{\{n\}} \overset{\bot}{\oplus} \mathcal{H}_{n,m}$.
\item Then we go from $\varphi(x_1)$ to $\varphi(y_1)$ in $\mathcal{H}_{n,m}$.
\item Finally, we go from $\varphi(y_1)$ to $y_1$ by a rotation
in the orthogonal direct sum $(F^{(\infty)})^{\{m\}} \overset{\bot}{\oplus} \mathcal{H}_{n,m}$.
\end{itemize}

\begin{figure}[h]
\begin{center}
\psfrag{A1}{$(F^{(\infty)})^{\{n\}}$}
\psfrag{A}{$x_1$}
\psfrag{C}{$(F^{(\infty)})^{\{m\}}$}
\psfrag{C1}{$y_1$}
\psfrag{B}{$\mathcal{H}_{n,m}$}
\psfrag{B1}{$\varphi(x_1)$}
\psfrag{B2}{$\varphi(y_1)$}
\includegraphics{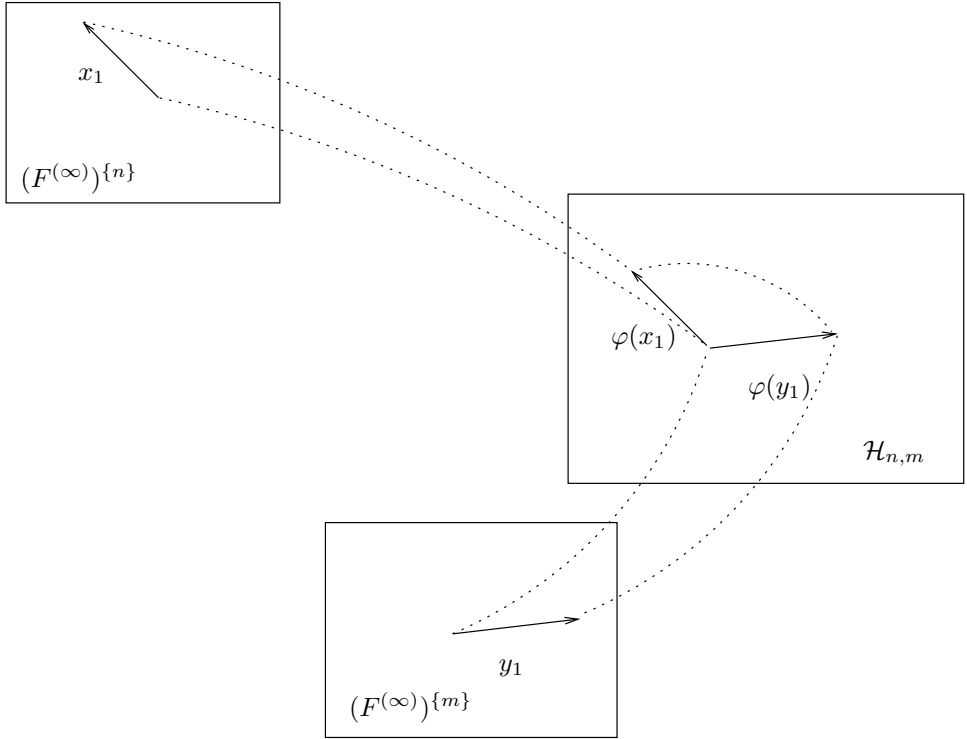}
\caption{The three-leg path construction}
\end{center}
\end{figure}

This will solve the problem of orthogonality. Of course, by doing so, we need
to define $\mathcal{H}$ as much larger than $(F^{(\infty)})^{(\N)}$: this will be done in Section \ref{10.4.3}.

\vskip 2mm
Finally let $n<m$ be non-negative integers and $x \in (F^{(\infty)})^{\{n\}}$, $y \in (F^{(\infty)})^{\{m\}}$
and $z \in (F^{(\infty)})^{\{m\}}$ be unit vectors.
Assume we already have constructed a path $f$ from $y$ to $z$
in $(F^{(\infty)})^{\{m\}}$, and, for every
$y \in (F^{(\infty)})^{\{m\}}$, a path from $x$ to $y'$. Then for every $t \in [0,1]$, we have a path from
$x$ to $f(t)$ in  the orthogonal direct sum
$(F^{(\infty)})^{\{n\}} \overset{\bot}{\oplus} \mathcal{H}_{n,m}$, and those
paths may be used directly to create a map from $\Delta^2$ to the orthogonal direct sum
$(F^{(\infty)})^{\{n\}} \overset{\bot}{\oplus} \mathcal{H}_{n,m}$,
using the ``associativity of barycenters'' in the $2$-simplex (cf.\ next figure).
\begin{center}
\mbox{
\psfrag{A}{$x$}
\psfrag{C}{$y$}
\psfrag{B}{$z$}
\psfrag{D}{$f(t_1)$}
\psfrag{E}{$f(t_2)$}
\psfrag{F}{$f(t_3)$}
\includegraphics{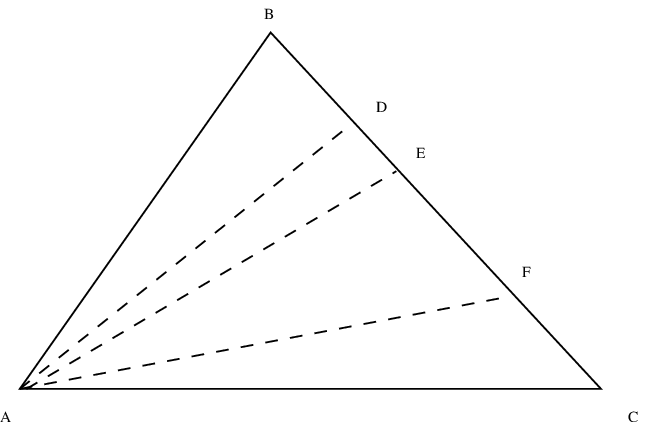}
}
\end{center}

\vskip 2mm
\noindent \textbf{Structure of the rest of the section:} \\
In Section \ref{10.4.1}, we develop an elementary construction of subdivisions of the $n$-simplex that generalizes the subdivision
of the interval into three parts which is the basis of the path construction that we have just discussed. Section \ref{10.4.2} is devoted to the
combinatorial material that we need in order to generalize the ``associativity of barycenters'' argument which was
mentioned earlier. In Section \ref{10.4.3}, we turn to the definition of the inner-product space $\mathcal{H}$
we will work with. The detailed construction of ``paths'' is the topic of Section \ref{10.4.4}, in which we formulate the
conjecture that a certain construction of paths exists. It is then proved in Section \ref{10.4.5}
that this conjecture is a consequence of a simple conjecture on the triangulation of smooth manifolds (see Section \ref{conjecturesection}).
Since the construction involves choices of extensions, we also use the conjecture on smooth manifolds to establish that
two possible constructions will necessarily be ``homotopic''. In Section \ref{10.4.6}, we use the previous construction to obtain
a $G$-map $s\vEc_{G,L^2}^{F,\infty*}[1] \longrightarrow \underline{\Fred}(L^2(G,\mathcal{H})^\infty)$, and
we finally recover a morphism $s\vEc_{G,L^2}^{F,\infty*} \longrightarrow \underline{\Fred}(G,\mathcal{H}^\infty)$
in Section \ref{10.4.7}, as explained earlier.

\subsection{Additional definitions}\label{10.4.0}

Recall that when $(A_n)_{n \in \mathbb{N}}$ is a simplicial set, an element $x \in A_n$ is called
\textbf{degenerate} when $x \in \sigma_*(A_k)$ for some $\sigma : [n] \twoheadrightarrow [k]$ with $k<n$.

For every $n\in \mathbb{N}$ and every $x \in A_n$, there is a unique pair
$(\sigma,y)$ such that $\sigma : [n] \twoheadrightarrow [k]$ is an epimorphism of $\Delta$,
the element $y$ is a non-degenerate element of $A_k$, and $x=\sigma_*(y)$. We will call
$y$ the \textbf{root} of $x$, and $\sigma$ its \textbf{reduction}.

\vskip 2mm
For every $n \in \mathbb{N}$, we denote by $\Hom_\uparrow([n],\mathbb{N})$ the set of non-decreasing
maps from $[n]$ to $\mathbb{N}$. The structure of cosimplicial set
on $\Delta$ thus induces a simplicial set $\Hom_\uparrow(\Delta,\mathbb{N})$ so that, for every morphism $\tau
: [k] \rightarrow [n]$ in $\Delta$,
$$\tau_*: \begin{cases}
\Hom_\uparrow([n],\mathbb{N}) & \longrightarrow \Hom_\uparrow([k],\mathbb{N}) \\
f & \longmapsto f \circ \tau.
\end{cases}$$ Of course, $f\in \Hom_\uparrow([n],\mathbb{N})$ is non-degenerate if and
only if it is an increasing map. Finally, for every $N \in \mathbb{N}$ and $f\in \Hom_\uparrow([n],\mathbb{N})$, we set
$f+N : \begin{cases}
[n] & \longrightarrow \mathbb{N} \\
i & \longmapsto f(i)+N.
\end{cases}$

\subsection{Subdivisions of the $n$-simplex}\label{10.4.1}

\subsubsection{The relation $<$}
Set
$$E:=\underset{(k,n)\in \mathbb{N}^2}{\coprod}\Hom_{\Delta^*}([k],[n]).$$
We define a binary relation $<$ on $E$ as follows:
for every $\delta: [k] \hookrightarrow [n]$ and $\delta' : [k'] \hookrightarrow [n']$,
$$\delta< \delta' \; \underset{\text{def}}{\Longleftrightarrow} \;
\begin{cases}
k<n  \\
n=k'
\end{cases}$$
(i.e. $\delta<\delta'$ if and only if $\delta' \circ \delta$ exists
and $\delta$ is not bijective). \\
A finite sequence $(\delta_1,\delta_2,\dots,\delta_m)$ in $E$ is called \textbf{increasing}
when \begin{itemize}
\item $\forall i \in \{1,\dots,m-1\},\; \delta_i <\delta_{i+1}$;
\item $\delta_m=\id_{[n]}$ for some $n \in \mathbb{N}$.
\end{itemize}
If $u$ is an increasing sequence, we simply write
$u=\delta_1<\delta_2<\dots <\delta_{m-1}<[n]$ if $\delta_m=\id_{[n]}$,
we will say that $u$ ends at $[n]$, and we will call $m$ the \textbf{length} of $u$.
We say that $u$ is \textbf{trivial} when its length is $1$.
Finally, we define $$\Lat(u):=
\bigl\{\delta_m \circ \delta_{m-1} \circ \dots \circ \delta_i \mid  i \in \{1,\dots,m\} \bigr\}.$$
Let $u=\delta_1<\dots<\delta_{m-1}<[n]$ be a non trivial increasing sequence.
We define the \textbf{differential} of $u$ as
$$u':=\delta_1<\dots<\delta_{m-2}<\id_{[k]},$$
where $[k]$ is the domain of $\delta_{m-1}$.
More generally, we define, for every $i \in \{1,\dots,m-1\}$, the sequence
$$u^{(i)}:=\delta_1<\dots<\delta_{m-i-1}<\id_{[k_i]},$$
where $[k_i]$ is the domain of $\delta_{m-i}$.
Obviously, $\Lat(u)=\{\id_{[n]}\} \cup \bigl\{\delta_{m-1} \circ \delta\; |\; \delta \in \Lat(u')\bigr\}$. \\
For every $n \in \mathbb{N}$, we denote by $S(n)$ the set of increasing sequences
in $E$ which end at $[n]$, and we define an order relation $\subset$ on $S(n)$ by
$$\forall (u,v)\in S(n)^2, \quad u \subset v \; \underset{\text{def}}{\Longleftrightarrow} \; \Lat(u) \subset \Lat(v).$$
Finally, for every $n \in \mathbb{N}$ and every pair $(u,v)\in S(n)^2$, we define
$$S_{u,v}:=\bigl\{w \in S(n) : \quad w \subset u \quad \text{and} \quad w \subset v\bigr\}.$$

\subsubsection{Subdivisions of the $n$-simplex}\label{core}

For $n \in \mathbb{N}$, we consider the $n$-simplex
$$\Delta^n=\bigl\{(t_0,\dots,t_n) \in (\mathbb{R}_+)^{n+1}:t_0+t_1+\dots+t_n=1\bigr\} \subset
\mathbb{R}^{n+1}.$$

For any $i \in \{0,\dots,n\}$, we set
$$\Delta_i^n:=\left\{(t_0,\dots,t_n) \in \Delta^n: \quad
t_i\leq \frac{1}{2(n+1)} \quad \text{and} \quad \forall
j \in \{0,\dots,n\}, \; t_i \leq t_j\right\}$$
and
$$C(\Delta^n):=\left\{(t_0,\dots,t_n) \in \Delta^n: \quad
\forall i \in \{0,\dots,n\}, \; t_i \geq \frac{1}{2(n+1)}\right\}.$$
\begin{figure}[h]
\begin{center}
\psfrag{0}{$0$}
\psfrag{1}{$1$}
\psfrag{2}{$2$}
\psfrag{a}{$\Delta_1^2$}
\psfrag{b}{$\Delta_0^2$}
\psfrag{c}{$\Delta_2^2$}
\psfrag{d}{$C(\Delta^2)$}
\includegraphics{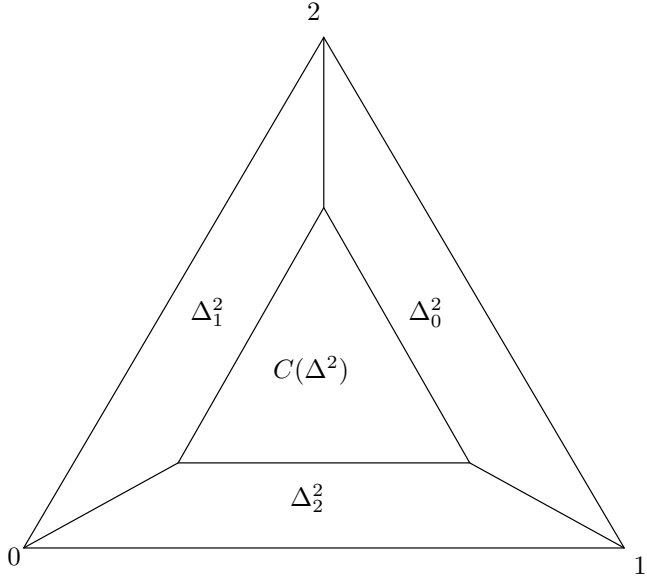}
\caption{A subdivision of $\Delta^2$}
\end{center}
\end{figure}
Notice the homeomorphism
$$\alpha^n : \begin{cases}
C(\Delta^n) & \overset{\cong}{\longrightarrow} \Delta^n \\
(t_0,\dots,t_n) & \longmapsto \left(2(t_0-\frac{1}{2(n+1)}),
2(t_1-\frac{1}{2(n+1)}),\dots, 2(t_n-\frac{1}{2(n+1)})\right).
\end{cases}$$
Furthermore,  for every $i\in \{0,\dots,n\}$, we define a homeomorphism
$$\alpha_i^n : \begin{cases}
\Delta_i^n & \overset{\cong}{\longrightarrow} \Delta^{n-1} \times I \\
(t_0,\dots,t_n) & \longmapsto \left[(\frac{t_0}{1-t_i},\dots,\frac{t_{i-1}}{1-t_i}
,\frac{t_{i+1}}{1-t_i},\dots\frac{t_n}{1-t_i}), 2(n+1)t_i \right].
\end{cases}$$
Gluing the maps $(\alpha_0^n)^{-1},(\alpha_1^n)^{-1},\dots,(\alpha_n^n)^{-1}$ together yields a homeomorphism
$$r_n : \partial \Delta^n \times [0,1] \overset{\cong}{\longrightarrow}
\underset{0\leq i\leq n}{\bigcup}\Delta_i^n.$$

\label{deltau}
\vskip 2mm
The following recursive definition yields, for every $n \in \mathbb{N}$ and every increasing sequence $u=\delta_1<\dots<\delta_{m-1}<n$ in $S(n)$,
a compact subset $\Delta_u$ of $\Delta^n$:
\begin{itemize}
\item If $u$ is trivial, then $\Delta_u:=C(\Delta^n)$.
\item Otherwise, $\Delta_u:=r_n(\delta_{m-1}^*(\Delta_{u'}) \times [0,1])$.
\end{itemize}
For every $n \in \mathbb{N}$, we obtain a subdivision of the $n$-simplex in this manner.

\vskip 2mm
\textbf{Example :} There are thirteen increasing sequences that end at $[2]$. They are: \begin{align*}
[2], & \quad u_1:=(0,1) <[2], & \quad   u_2:=(0,2) <[2],  \\
u_3:=(1,2) <[2] & \quad u_4:= 0<(0,1) <[2], & \quad  u_5:=1<(0,1) <[2],  \\
u_6:= 0<(0,2) <[2], & \quad u_7:= 1<(0,2) <[2], & \quad  u_8:=0<(1,2) <[2], \\
u_9:=1<(1,2) <[2], & \quad u_{10}:= 0<[2], & \quad u_{11}:=1 <[2], \quad u_{12}:=2<[2].
\end{align*}
The corresponding subdivision of $\Delta^2$ is detailed in the following figure:
\begin{figure}[h]
\begin{center}
\psfrag{a}{$0$}
\psfrag{b}{$1$}
\psfrag{c}{$2$}
\psfrag{u0}{$\Delta_{[2]}$}
\psfrag{u1}{$\Delta_{u_1}$}
\psfrag{u2}{$\Delta_{u_2}$}
\psfrag{u3}{$\Delta_{u_3}$}
\psfrag{u4}{$\Delta_{u_4}$}
\psfrag{u5}{$\Delta_{u_5}$}
\psfrag{u6}{$\Delta_{u_6}$}
\psfrag{u7}{$\Delta_{u_7}$}
\psfrag{u8}{$\Delta_{u_8}$}
\psfrag{u9}{$\Delta_{u_9}$}
\psfrag{v1}{$\Delta_{u_{10}}$}
\psfrag{v2}{$\Delta_{u_{11}}$}
\psfrag{v3}{$\Delta_{u_{12}}$}
\includegraphics{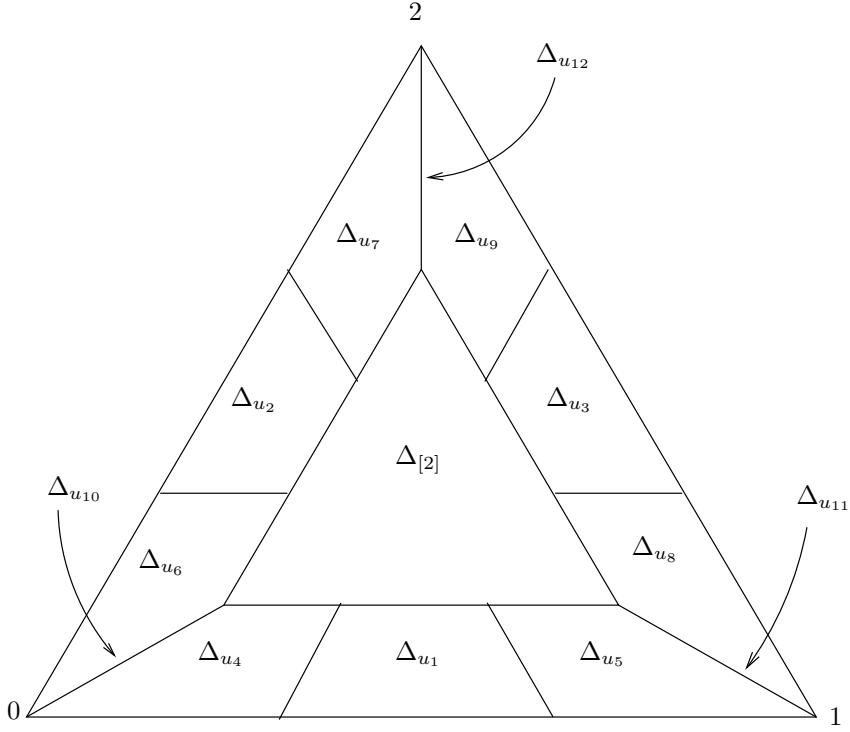}
\caption{The full subdivision of $\Delta^2$}
\end{center}
\end{figure}

\subsubsection{Some properties of the subdivisions}

In the rest of the construction, we will need the following two results:
\begin{prop}\label{glue}
Let $n \in \mathbb{N}$. Then:
\begin{itemize}
\item[\emph{(i)}] $\underset{u \in S(n)}{\bigcup}\Delta_u=\Delta^n$;
\item[\emph{(ii)}] $\underset{u \in S(n)\setminus\{[n]\}}{\bigcup}\Delta_u=\underset{0\leq i \leq N}{\bigcup}\Delta_i^n$;
\item[\emph{(iii)}] the space $\underset{0\leq i \leq N}{\bigcup}\Delta_i^n$ has the final topology for the inclusions $\Delta_u \subset
\underset{0\leq i \leq N}{\cup}\Delta_i^n$, with non-trivial $u\in S(n)$.
\end{itemize}
\end{prop}

\begin{prop}\label{intersection}
Let $n \in \mathbb{N}$, and $(u,v)\in S(n)^2$ be a pair of non-trivial sequences. Then
$$\Delta_u \cap \Delta_v \subset \underset {w \in S_{u,v}\setminus \{[n]\}}{\bigcup}\Delta_w.$$
\end{prop}

\begin{proof}[Proof of proposition \ref{glue} :]
We prove (i) and (ii) by induction on $n$. When $n=0$, $\Delta^0=C(\Delta^0)=\Delta_{[0]}$.
Let $n \in \mathbb{N}^*$, and assume the result is true for $n-1$.
It suffices to prove that $\Delta^n \subset \underset{u \in S(n)}{\bigcup}\Delta_u$
and $\underset{0\leq i \leq N}{\bigcup}\Delta_i^n \subset \underset{u \in S(n)\setminus\{[n]\}}{\bigcup}\Delta_u$.\\
Let $x \in \Delta^n$. If $x \in C(\Delta^n)$, then $x \in \Delta_{[n]}$. Otherwise,
we choose $i \in [n]$ such that $x \in \Delta_i^n$, and we set $(y,t):=r_n^{-1}(x)$, and
$z \in \Delta^{n-1}$ such that $y=(\delta_i^n)^*(z)$. Then, by the induction hypothesis, there is some
$u \in S(n-1)$ such that $z \in \Delta^{n-1}$, and we write
$u=\delta_1<\dots<\delta_{m-1}<[n-1]$. We then set $v:=\delta_1<\dots<\delta_{m-1}<\delta_i^n<[n]$. Then
$v$ is non trivial, $v \in S(n)$, $v'=u$, and we deduce that $\Delta_v=r_n((\delta_i^n)^*(\Delta_u)\times [0,1])$.
Since $x=r_n((\delta_i^m)^*(z),t))$ and $z \in \Delta_u$, it follows that $x \in \Delta_v$. \\
We conclude that $\Delta^n \subset \underset{u \in S(n)}{\bigcup} \Delta_u$ and
$\underset{0\leq i \leq N}{\bigcup}\Delta_i^n \subset \underset{u \in S(n)\setminus \{[n]\}}{\bigcup} \Delta_u$,
and this proves (i) and (ii) by induction on $n$.

Finally, let $n \in \mathbb{N}$. Since $S(n)$ is finite, $\underset{u \in S(n)\setminus\{[n]\}}{\coprod}\Delta_u$ is a compact space.
Also, $\underset{0\leq i \leq n}{\bigcup}\Delta_i^n$ is compact. Since the canonical map
$\underset{u \in S(n)\setminus\{[n]\}}{\coprod}\Delta_u \longrightarrow \underset{0\leq i \leq n}{\cup}\Delta_i^n$ is continuous and onto,
it follows that it is an identification map. \end{proof}

In order to prove Proposition \ref{intersection}, we will need
the following technical lemma:
\begin{lemme}\label{faceintersectionlemma}
Let $n \in \mathbb{N}$, $u \in S(f)$, and $\delta: [k] \hookrightarrow [n]$. \\
If $\delta^*(\Delta^k) \cap \Delta_u \not\subset \delta^*(\partial \Delta^k)$, then $\delta \in \Lat(u)$ and
there exists an $i \in \mathbb{N}$ such that $$\delta^*(\Delta^k) \cap \Delta_u=\delta^*(\Delta_{u^{(i)}}).$$
\end{lemme}

\begin{proof}
We will use the following simple fact of simplicial geometry:
\begin{center}
for every $n \in \mathbb{N}$, $\delta : [k] \hookrightarrow [n]$ and $\delta' : [k'] \hookrightarrow [n]$, one has
$\delta^*(\Delta^k) \cap (\delta')^*(\Delta^{k'}) \not \subset (\delta')^*(\partial \Delta^{k'})$ if and only
if there is some $\delta'': [k'] \hookrightarrow [k]$ such that $\delta' =\delta \circ \delta''$.
\end{center}
Now, we prove the claimed result by induction on $n$. It is obvious when $n=0$. \\
Let $n\in \mathbb{N}^*$, and assume the result holds for every $n'<n$, every $u \in S(n')$ and
$\delta : [k'] \hookrightarrow [n']$. \\
Let $u \in S(n)$ and $\delta : [k] \hookrightarrow [n]$. We assume that
$\delta^*(\Delta^k) \cap \Delta_u \not\subset \delta^*(\partial \Delta^k)$.
If $\delta=\id_{[n]}$, then $\delta \in \Lat(u)$ and $\delta^*(\Delta^k) \cap \Delta_u=\Delta^n \cap \Delta_u=\Delta_u$.

Assume now that $k<n$. If $u$ is trivial, then $\partial \Delta^n \cap \Delta_u =\emptyset$, and
this contradicts $\delta^*(\Delta^k) \cap \Delta_u \neq \emptyset$. Hence $u$ is non-trivial.
We now write $u=\delta_1<\dots<\delta_{m-1}<n$, with $\delta_{m-1} : [n'] \hookrightarrow [n]$.
By construction, $\Delta_u \cap \partial \Delta^n =\delta_{m-1}^*(\Delta_{u'})$.
It follows that $\delta_{m-1}^*(\Delta^{n'}) \cap \delta^*(\Delta^k) \not\subset \delta^*(\partial\Delta^k)$, and we deduce that
$\delta=\delta_{m-1} \circ \delta''$ for some $\delta'' : [k'] \hookrightarrow [n']$.

It follows that
\begin{align*}
\delta^*(\Delta^k) \cap \Delta_u & =\delta^*(\Delta^k) \cap (\partial \Delta^n \cap \Delta_u) \\
 & =\delta^*(\Delta^k) \cap \delta_{m-1}^*(\Delta_{u'}) \\
\delta^*(\Delta^k) \cap \Delta_u & =\delta_{m-1}^*((\delta'')^*(\Delta^{k}) \cap \Delta_{u'}).
\end{align*}
Moreover $\delta^*(\Delta^k)=(\delta_{m-1})^*((\delta'')^*(\Delta^{k}))$, hence
$(\delta'')^*(\Delta^k) \cap \Delta_{u'} \not\subset (\delta'')^*(\partial \Delta^k)$.

It follows from the induction hypothesis that $\delta''\in \Lat(u')$ and
$(\delta'')^*(\Delta^k) \cap \Delta_{u'}=(\delta'')^*(\Delta_{(u')^{(i)}})$ for some integer $i$.
Hence $\delta=\delta_{m-1} \circ \delta''\in \Lat(u)$, and
$$\delta^*(\Delta^k) \cap \Delta_u=\delta_{m-1}^*((\delta'')^*(\Delta_{u^{(i+1)}}))=\delta^*(\Delta_{u^{(i+1)}}).$$
This proves lemma \ref{faceintersectionlemma}.
\end{proof}

\begin{proof}[Proof of Proposition \ref{intersection} :]
We prove the result by induction on $n$. For $n=0$, the result is obvious.
Let $n\in \mathbb{N}^*$, and assume the result holds for every $n'<n$ and every pair $(u,v)\in S(n')^2$ of non-trivial sequences.

Let $(u,v)\in S(n)^2$ be a pair of non-trivial sequences. Let $x \in \Delta_u \cap \Delta_v$.
Then $x \in \underset{i \in [n]}{\bigcup} \Delta_i^n$ and we may therefore set
$(y,t):=(r_n)^{-1}(x)$. By construction of $\Delta_u$ and $\Delta_v$, we deduce that $y \in \Delta_u \cap \Delta_v$.
Then $y \in \partial \Delta^n$ and there is a unique lattice
$\delta : [k] \hookrightarrow [n]$, with $k<n$, such that $y \in \delta^*(\Delta^k \setminus \partial \Delta^k)$.
Since $y \in \Delta_u \cap \Delta_v$, we deduce that $\delta^*(\Delta^k) \cap \Delta_u \not\subset \delta^*(\partial \Delta^k)$
and $\delta^*(\Delta^k) \cap \Delta_v \not\subset \delta^*(\partial \Delta^k)$.
It then follows from Lemma \ref{faceintersectionlemma} that $\delta \in \Lat(u) \cap \Lat(v)$ and
that there exists a pair $(i,j)$ of integers such that
$\delta^*(\Delta^k) \cap \Delta_u=\delta^*(\Delta_{u^{(i)}})$ and
$\delta^*(\Delta^k) \cap \Delta_v=\delta^*(\Delta_{v^{(j)}})$.
Then $y \in \delta^*(\Delta_{u^{(i)}} \cap \Delta_{v^{(j)}})$.

Let now $z\in \Delta^k$ be such that $y=\delta^*(z)$. Hence
$z \in \Delta_{u^{(i)}} \cap \Delta_{v^{(j)}}$. If
$u^{(i)}$ or $v^{(j)}$ is trivial, then $w \in \Delta_{[k]}$.
Otherwise, we deduce from the induction hypothesis
that there exists $w \in S_{u^{(i)},v^{(j)}} \setminus\{[k]\}$ such that
$z \in \Delta_w$. In any case, we have $w \in S_{u^{(i)},v^{(j)}}$ such that $z \in \Delta_w$.
We write $w=\delta_1<\dots<\delta_{m-1} < [k]$, and set
$w_1:=\delta_1<\dots<\delta_{m-1} < \delta < [n]$. Then
\begin{align*}
\Lat(w_1) & =\{[n]\} \cup \{\delta \circ \delta', \delta' \in \Lat(w)\} \\
& \subset \{[n]\} \cup \{\delta \circ \delta', \delta' \in \Lat(u^{(i)})\} \\
\Lat(w_1)& \subset \Lat(u),
\end{align*}
and similarly $\Lat(w_1) \subset \Lat(v)$. We deduce that $w_1 \in S_{u,v}\setminus \{[n]\}$. \\
Finally, $w_1'=w$, and it follows that $x=r_n(\delta^*(z),t)) \in \Delta_w$. \\
We conclude that $\Delta_u \cap \Delta_v \subset \underset{w \in S_{u,v}\setminus\{[n]\}}{\bigcup}\Delta_w$.
\end{proof}

\subsection{Associativity of barycenters in the $n$-simplex}\label{10.4.2}

\subsubsection{Cartesian squares in the simplicial category}
In the category $\Delta$, every pair of morphisms $(\sigma,\delta)$, with
$\sigma : [n] \twoheadrightarrow [k]$, and $\delta : [k'] \hookrightarrow [k]$,
gives rise to a cartesian square: $$\xymatrix{
[m] \ar@{^{(}->}[d]^{\sigma \natural \delta} \ar@{>>}[r]^{\delta \natural \sigma} & [k'] \ar@{^{(}->}[d]^\delta \\
[n] \ar@{>>}[r]^{\sigma} & [k]
}$$
With $\delta : [k'] \hookrightarrow [k]$ fixed, every commutative triangle
$\xymatrix{
[n] \ar@{>>}[dr]^{\sigma} \ar[rr]^{\tau} & & [n'] \ar@{>>}[dl]^{\sigma'} \\
& [k]
}$ yields a commutative diagram
$$\xymatrix{
[m]
\ar@{>>}[dr]^{\delta \natural \sigma} \ar[rr]^{\tau'}
\ar[drrrr]^(0.7){\sigma \natural \delta}
 & & [m'] \ar@{>>}[dl]^{\delta \natural \sigma'}
\ar[drrrr]^{\sigma' \natural \delta}  \\
& [k'] \ar@{^{(}->}[drrrr]_\delta & & & [n] \ar[rr]^\tau \ar[dr]_\sigma & & [n'] \ar[dl]^{\sigma'} \\
& & & & & [k]
}.
$$
For a fixed epimorphism $\sigma : [n] \twoheadrightarrow [k]$,
we obtain a functor $\sigma\natural :\; \Delta^* \downarrow [k] \longrightarrow \Delta^* \downarrow [n]$.
For every $i \in [k]$, we also set $\sigma \natural i:=\sigma \natural \delta_i$, where
$\delta_i : [0] \hookrightarrow [k]$ maps $0$ to $i$.

\subsubsection{The maps defined by the associativity of barycenters}

Recall that $(\mathbb{R}^{[n]})_{n\in \mathbb{R}}$ is equipped with both a canonical structure of cosimplicial space
and a canonical structure of simplicial set.

Let $\sigma : [n] \twoheadrightarrow [k]$ be an epimorphism in $\Delta$, and set
$n_i:=\# \sigma^{-1}\{i\}$ for every $i\in [k]$.
The canonical ring structure on $\mathbb{R}^{[n]}$ yields a
continuous map
$$\lambda_\sigma : \begin{cases}
\mathbb{R}^{[n]} \times \mathbb{R}^{[k]} & \longrightarrow \mathbb{R}^{[n]} \\
(x,t) & \longmapsto \sigma_*(t) \times x
\end{cases}$$
where $\times$ is the standard $(n+1)$-fold product on $\R^{[n]}$ (the set of functions from $[n]$ to $\R$).
For every $i \in [k]$, we also set
$$\lambda^{(i)}_\sigma : \begin{cases}
\mathbb{R}^{\sigma^{-1}\{i\}} \times \mathbb{R}^{\{i\}} & \longrightarrow \mathbb{R}^{\sigma^{-1}\{i\}} \\
(x,t) & \longmapsto t.x
\end{cases}$$ and notice that
$\lambda_\sigma=\underset{i=0}{\overset{k}{\prod}}\lambda^{(i)}_\sigma$.

For every non-empty finite subset $I$ of $\mathbb{N}$, we denote by $\Delta(\mathbb{R}^I)$
the subset of $\mathbb{R}^I$ consisting of those families $(t_i)_{i \in I}$ such that $\underset{i \in I}{\sum} t_i=1$.
The previous decomposition of $\lambda_\sigma$ helps us see that $\lambda_\sigma$ induces a continuous map
$$\left(\underset{i=0}{\overset{k}{\prod}}\Delta\bigl(\mathbb{R}^{\sigma^{-1}\{i\}}\bigr)\right) \times \Delta^k
\longrightarrow \Delta^n,$$
which, composed with
$$\left(\underset{i=0}{\overset{k}{\sum}}(\sigma \natural i)^*\right)\times \id :
\left(\underset{i=0}{\overset{k}{\prod}}\Delta^{n_i-1}\right) \times \Delta^k
\longrightarrow
\left(\underset{i=0}{\overset{k}{\prod}}\Delta(\mathbb{R}^{\sigma^{-1}\{i\}})\right) \times \Delta^k,
$$
yields a map
$$\left(\underset{i=0}{\overset{k}{\prod}}\Delta^{n_i-1}\right) \times \Delta^k
\longrightarrow \Delta^n$$
which we still write $\lambda_\sigma$.

Note that $\lambda_\sigma$ is onto. Indeed, let $t=(t_i)_{0 \leq i \leq n}$:
set $t':=\sigma^*(t)$ and define $(x_i)_{0 \leq i \leq k}$ as: $x_i=(1,0,0,\dots,0)$ if $t'_i= 0$, and $x_i:=\left(\frac{t_{(\sigma \natural i)(j)}}{t_i}\right)_{j \in [n_i-1]}$
otherwise. It is then easily checked that $\lambda_\sigma(x,t')=t$.

If $k=0$ or $k=n$, then $\lambda_\sigma$ is the identity map.
In any other case, it may easily be proven that $\lambda_\sigma$ is \emph{not} one-to-one. This lack of injectivity
is a problem: in the following paragraphs, we explain what needs to be done in order to construct
a homeomorphism from $\lambda_\sigma$.

\subsubsection{The functors $F_\sigma$ and $H_\sigma$}

By composing $\sigma \natural -$ with
the forgetful functor $\Delta^* \downarrow [n] \longrightarrow \Delta$ and the functor defined by the canonical structure
of simplicial space on $(\mathbb{R}^{[n]})_{n\in \mathbb{N}}$, we recover a contravariant
functor $F_\sigma : \Delta^* \downarrow [k] \longrightarrow \text{Top}$.

By composing the forgetful functor $\Delta^* \downarrow [k] \longrightarrow \Delta$
with the canonical structure of cosimplicial space on $(\Delta^n)_{n\in \mathbb{N}}$, we recover
a functor $H_\sigma : \Delta^* \downarrow [k] \longrightarrow \text{Top}$
which maps $\delta : [m] \hookrightarrow [n]$ to $\Delta^m$ and
$\xymatrix{
[m] \ar@{^{(}->}[dr]^{\delta} \ar@{^{(}->}[rr]^{\delta''} & & [m'] \ar@{^{(}->}[dl]^{\delta'} \\
& [k]
}$ to $(\delta'')^*: \Delta^m \longrightarrow \Delta^{m'}$.

For every pair of morphisms $(\sigma,\delta)$, with
$\sigma : [n] \twoheadrightarrow [k]$, $\delta : [k'] \hookrightarrow [k]$, and
$\sigma \natural \delta : [m] \hookrightarrow [k']$,
the square
$$\begin{CD}
\mathbb{R}^{[k']} @>{(\delta \natural \sigma)_*}>> \mathbb{R}^{[m]} \\
@VV{\delta^*}V @VV{(\sigma \natural \delta)^*}V \\
\mathbb{R}^{[k]} @>{\sigma_*}>> \mathbb{R}^{[n]}
\end{CD}$$
is easily shown to be commutative. \\
It follows that, for every morphism $\varphi=\xymatrix{
[m] \ar@{^{(}->}[dr]^{\delta} \ar@{^{(}->}[rr]^{\delta''} & & [m'] \ar@{^{(}->}[dl]^{\delta'} \\
& [k]
}$ in $\Delta^*\downarrow [k]$, the square
$$\begin{CD}
F_\sigma(\delta') \times H_\sigma(\delta) @>>{\id \times H_\sigma(\varphi)}> F_\sigma(\delta') \times H_\sigma(\delta') \\
@VV{F_\sigma(\varphi)\times \id}V @VV{(\sigma \natural \delta')^* \circ \lambda_{\delta' \natural \sigma}}V \\
F_\sigma(\delta) \times H_\sigma(\delta) @>>{(\sigma \natural \delta)^* \circ \lambda_{\delta \natural \sigma}}> \Delta^n
\end{CD}$$
is commutative.

\subsubsection{The functor $G_\sigma$}\label{gsigma}

Let $\sigma : [n] \twoheadrightarrow [k]$ be an epimorphism in $\Delta$.
For every monomorphism $\delta : [m] \hookrightarrow [k]$,
set $G_\sigma(\delta):= \underset{i=0}{\overset{m}{\prod}}\Delta^{n_{\delta(i)}-1}$ and
$$\alpha_\sigma(\delta):=\underset{i=0}{\overset{m}{\sum}}((\delta \natural \sigma)\natural i)^*:
 G_\sigma(\delta) \longrightarrow F_\sigma(\delta).$$
For every morphism
$\varphi=\xymatrix{
[m] \ar@{^{(}->}[dr]^{\delta} \ar@{^{(}->}[rr]^{\delta''} & & [m'] \ar@{_{(}->}[dl]^{\delta'} \\
& [k]}$ in $\Delta^*\downarrow [k]$, set
$$G_\sigma(\varphi): \begin{cases}
\underset{i=0}{\overset{m'}{\prod}}\Delta^{n_{\delta'(i)}-1}
& \longrightarrow \underset{j=0}{\overset{m}{\prod}}\Delta^{n_{\delta(j)}-1} \\
(x_i)_{0 \leq i \leq m'} & \longmapsto (x_{\delta''(j)})_{0 \leq j \leq m.}
\end{cases}$$
This yields a contravariant functor
$G_\sigma : \Delta^*\downarrow [k] \longrightarrow \text{Top}$.
It is then easy to check that
$\alpha_\sigma : G_\sigma \longrightarrow F_\sigma$ is a natural transformation. \\
We deduce that for every morphism $\varphi=\xymatrix{
[m] \ar@{^{(}->}[dr]^{\delta} \ar@{^{(}->}[rr]^{\delta''} & & [m'] \ar@{_{(}->}[dl]^{\delta'} \\
& [k]}$ in $\Delta^*\downarrow [k]$, the square
$$\begin{CD}
G_\sigma(\delta') \times H_\sigma(\delta) @>>{\id \times H_\sigma(\varphi)}> G_\sigma(\delta') \times H_\sigma(\delta') \\
@VV{G_\sigma(\varphi)\times \id}V @VV{(\sigma \natural \delta')^* \circ \lambda_{\delta' \natural \sigma}}V \\
G_\sigma(\delta) \times H_\sigma(\delta) @>>{(\sigma \natural \delta)^* \circ \lambda_{\delta \natural \sigma}}> \Delta^n
\end{CD}$$
is commutative.

\label{nusigma}
It follows that for every object $\delta : [m] \hookrightarrow [k]$ of $\Delta^* \downarrow [k]$,
the maps
$(\sigma \natural \delta)^* \circ \lambda_{\delta \natural \sigma} : G_\sigma(\delta)\times \Delta^m
\longrightarrow \Delta^n$ yield a continuous map
$$\nu_\sigma : \left(\underset{ \delta \in \Ob(\Delta^* \downarrow [k])}{\coprod} G_\sigma(\delta) \times H_\sigma(\delta)\right)/_\sim
\longrightarrow \Delta^n,$$
where $\sim$ is the equivalence relation on
$\underset{\delta \in \Ob(\Delta^* \downarrow [k])}{\coprod} G_\sigma(\delta) \times H_\sigma(\delta)$
generated by the collection of elementary relations
$$\forall \varphi : \delta \rightarrow \delta', \; \forall (x,t) \in
G_\sigma(\delta') \times H_\sigma(\delta),\quad (x,\varphi^*(t)) \sim (\varphi_*(x),t).$$

Of course, since
$G_\sigma(\id_{[k]}) \times H_\sigma(\id_{[k]})=\left(\underset{i=0}{\overset{k}{\prod}}
\Delta^{n_i-1}\right) \times \Delta^k$ and
$\sigma \natural \id_{[k]}=\id_{[n]}$, the map $\lambda_\sigma$ is
the composite of $\nu_\sigma$ and of the canonical map
$$\left(\underset{i=0}{\overset{k}{\prod}}
\Delta^{n_i-1}\right) \times \Delta^k \longrightarrow
\left(\underset{\delta \in \Ob(\Delta^* \downarrow [k])}{\coprod} G_\sigma(\delta) \times H_\sigma(\delta)\right)/_{\sim.}$$

\begin{prop}\label{superglue}
The map $$\nu_\sigma : \left(\underset{ \delta \in
\Ob(\Delta^* \downarrow [k])}{\coprod} G_\sigma(\delta) \times H_\sigma(\delta)\right)/_\sim
\overset{\cong}{\longrightarrow} \Delta^n$$
is a homeomorphism.
\end{prop}

\begin{proof}
Since $\lambda_\sigma$ is onto, $\nu_\sigma$ is also onto.
Since $\Ob(\Delta^* \downarrow [k])$ is finite,
$G_\sigma (\delta) \times H_\sigma(\delta)$ is compact for every $\delta \in \Ob(\Delta^* \downarrow [k]$,
and $\Delta^n$ is compact, we deduce that $\nu_\sigma$
is an identification map.

In order to prove that $\nu_\sigma$ is one-to-one, it suffices to construct
a retraction of $\nu_\sigma$. Let $t \in \Delta^n$, and denote by $\delta : [m] \hookrightarrow [k]$
the unique monomorphism such that $\sigma^*(t) \in \delta^*(\partial \Delta^i)$.
Let $t' \in \partial \Delta^i$ such that $\delta^*(t')=\sigma^*(t)$, and set
$x:=\underset{i=0}{\overset{m}{\sum}}((\delta \natural \sigma)\natural i)^*(\frac{t'_j}{t'_i})
_{j \in [n_{\delta(i)}-1]}$.
Finally, we consider the class of $(x,t')$.
This construction yields a map which is easily seen to be a retraction of $\nu_\sigma$, and we deduce that $\nu_\sigma$ is one-to-one,
which finishes the proof.
\end{proof}

\subsubsection{Compatibility with the structure of cosimplicial space on $(\Delta^n)_{n\in \mathbb{N}}$}

Let $\tau : [n'] \rightarrow [n]$ be a morphism in $\Delta$. We then have a unique decomposition of $\sigma \circ \tau$ into
$$\xymatrix{
[n'] \ar[rr]^{\sigma \circ \tau}
\ar@{>>}[dr]^{\sigma'} & & [k] \\
& [k'] \ar@{^{(}->}[ur]_{\delta'}
}.$$
The morphism $\delta'$ is the root of $\sigma \circ \tau$, whilst $\sigma'$ is its reduction. \\
Let $\delta' \natural \sigma : [m] \rightarrow [k']$. The universal property of cartesian squares yields a morphism
$\tau': [n'] \rightarrow [m]$ which renders commutative the following diagram:
$$\xymatrix{
[n']  \ar@/_/[ddr]_\tau \ar@/^/[drr]^{\sigma'} \ar@{.>}[dr]|-{\tau'} \\
& [m] \ar@{>>}[r]^{\delta' \natural \sigma} \ar@{^{(}->}[d]^{\sigma \natural \delta'} & [k'] \ar@{^{(}->}[d]^{\delta'} \\
& [n] \ar@{>>}[r]^{\sigma} & [k].
}$$
For every $i \in [k']$, set $n'_i:=\#(\sigma')^{-1}\{i\}$: then
the commutative triangle $\xymatrix{
[n'] \ar@{>>}[dr]^{\sigma'} \ar[rr]^{\tau'} & & [m] \ar@{>>}[dl]^{\delta' \natural \sigma} \\
& [k']
}$ yields a commutative diagram
$$\xymatrix{
[n'_i-1]
\ar@{>>}[dr] \ar[rr]^{\tau_i}
\ar[drrrr]^(0.7){\sigma' \natural i}
 & & [n_{\delta'(i)}-1] \ar@{>>}[dl]
\ar[drrrr]^{(\delta' \natural \sigma) \natural i}  \\
& [0] \ar@{^{(}->}[drrrr]_i & & & [n'] \ar[rr]^\tau \ar[dr]_{\sigma'} & & [m]. \ar[dl]^{\delta' \natural \sigma} \\
& & & & & [k']
}
$$
The family $(\tau_i)_{0 \leq i \leq k'}$ will be called the \textbf{decomposition of $\tau$ over} $\sigma$.

\begin{prop}\label{splitting}
The square
$$\begin{CD}
G_{\sigma'}(\id_{[n']}) \times \Delta^{k'} @>{\nu_{\sigma'}}>> \Delta^{n'} \\
@V{\left(\underset{i=0}{\overset{k'}{\prod}}\tau_i^*\right) \times \id}VV @V{\tau^*}VV \\
G_\sigma(\delta') \times \Delta^{k'} @>{\nu_\sigma}>> \Delta^{n}
\end{CD}$$ is commutative.
\end{prop}

\begin{proof}
For every $i \in [k']$, one has $\tau'((\sigma')^{-1}\{i\}) \subset (\delta' \natural \sigma)^{-1}\{i\}$. We deduce that the square
$$\begin{CD}
\mathbb{R}^{(\sigma')^{-1}\{i\}} \times \mathbb{R}^{[k']} @>{\lambda^{(i)}_{\sigma'}}>> \mathbb{R}^{(\sigma')^{-1}\{i\}} \\
@V{(\tau')^* \times \id}VV @V{(\tau')^*}VV \\
\mathbb{R}^{(\delta' \natural \sigma)^{-1}\{i\}} \times \mathbb{R}^{[k']} @>{\lambda^{(i)}_{\delta' \natural \sigma}}>>
\mathbb{R}^{(\delta' \natural \sigma)^{-1}\{i\}}
\end{CD}
$$ is well defined and commutative for every $i \in [k']$, and it follows that the square

$$\begin{CD}
\mathbb{R}^{[n']} \times \mathbb{R}^{[k']} @>{\lambda_{\sigma'}}>> \mathbb{R}^{[n']} \\
@V{(\tau')^* \times \id}VV @V{(\tau')^*}VV \\
\mathbb{R}^{[m]} \times \mathbb{R}^{[k']} @>{\lambda_{\delta' \natural \sigma}}>> \mathbb{R}^{[m]}
\end{CD}
$$
is commutative. \\
Also, the square
$$\begin{CD}
\mathbb{R}^{[n'_i-1]} @>>> \mathbb{R}^{(\sigma')^{-1}\{i\}} \\
@V{(\tau_i)^*}VV @V{(\tau')^*}VV \\
\mathbb{R}^{[n_{\delta'(i)}-1]} @>>> \mathbb{R}^{(\delta' \natural \sigma)^{-1}\{i\}}
\end{CD}$$ is commutative for every $i\in [k']$. We deduce that the square
$$\begin{CD}
\underset{i=0}{\overset{k'}{\prod}}\Delta^{n'_i-1} @>{\alpha_{\sigma'}(\id_{[n']})}>>
\mathbb{R}^{[n']} \\
@V{\underset{i=0}{\overset{k'}{\prod}}\tau_i^*}VV @V{(\tau')^*}VV \\
\underset{i=0}{\overset{k'}{\prod}}\Delta^{n_{\delta'(i)}-1} @>{\alpha_{\sigma}(\delta')}>>
\mathbb{R}^{[n]}
\end{CD}$$
is commutative, and it follows that
the square
$$\begin{CD}
\underset{i=0}{\overset{k'}{\prod}}\Delta^{n'_i-1} \times \Delta^{n'} @>{\lambda_{\sigma'}}>>
\Delta^{n'} \\
@V{\left(\underset{i=0}{\overset{k'}{\prod}}\tau_i^*\right) \times \id}VV @V{(\tau')^*}VV \\
\underset{i=0}{\overset{k'}{\prod}}\Delta^{n_{\delta'(i)}-1} \times \Delta^{n'}
@>{\lambda_{\delta' \natural \sigma}}>> \Delta^{n'}
\end{CD}$$
is commutative. \\
Since $\tau=(\sigma \natural \delta')\circ \tau'$,
we conclude that the square
$$\begin{CD}
G_{\sigma'}(\id_{[n']}) \times \Delta^{k'} @>{\nu_{\sigma'}}>> \Delta^{n'} \\
@V{\left(\underset{i=0}{\overset{k'}{\prod}}\tau_i^*\right) \times \id}VV @V{\tau^*}VV \\
G_\sigma(\delta') \times \Delta^{k'} @>{\nu_\sigma}>> \Delta^{n}
\end{CD}$$
is commutative. \end{proof}

\subsection{The definition of $\mathcal{H}$}\label{10.4.3}

We now set $B:=\underset{n \in \mathbb{N}}{\coprod} \Hom_\uparrow([n],\mathbb{N})$
and $C:=\{f \in B : f \quad \text{is non degenerate}\}$. \\
For every non-negative integer $n$, and every non-degenerate $f : [n] \rightarrow \mathbb{N}$,
we set
$$\partial(f):=\left\{\delta_*(f),
\delta \in \underset{k<n}{\coprod}\Hom_{\Delta^*}([k],[n])\right\}$$
and $$\mathcal{H}_f:=(F^{(\infty)})^{\{f\}} \subset (F^{(\infty)})^{(B)}.$$
Let $f: [n] \rightarrow \mathbb{N}$ be a non-degenerate map. Then $\mathcal{H}_f$ has a
countable dimension (as a real vector space), and
$\forall g \in C, \; f \neq g \Rightarrow \mathcal{H}_f \bot \mathcal{H}_g$.
If $n>0$, we may therefore choose a bijective isometry
$$\varphi_f: \underset{g \in \partial(f)}{\overset{\bot}{\oplus}}\mathcal{H}_g
\overset{\cong}{\longrightarrow} \mathcal{H}_{f.}$$
We finally set
$$\mathcal{H}:=\overline{(F^{(\infty)})^{(C)}}=
\overline{\underset{f \in C}{\overset{\bot}{\oplus}}\mathcal{H}_f}$$
which is a separable Hilbert space since $C$ is countable.

By identifying every $f: [0] \rightarrow \mathbb{N}$ with $f(0)$, we may view $(F^{(\infty)})^{(\mathbb{N})}$ as a subspace
of $\mathcal{H}$. For every $n \in \mathbb{N}$, we also set
$\mathcal{H}_n:=(F^{(\infty)})^{\{n\}}$, seen as a subspace of $\mathcal{H}$.

\vskip 2mm
The filtration of $F^{(\infty)}$ by the
sequence $$F^1\hookrightarrow F^2 \hookrightarrow \cdots \hookrightarrow F^l \hookrightarrow F^{l+1} \cdots $$
gives rise, for every $k\in \mathbb{N}$, to a filtration of $\mathcal{H}_k$ by an increasing sequence of finite dimensional subspaces
$$\mathcal{H}_k^{(1)} \subset \mathcal{H}_k^{(2)} \subset \cdots \subset \mathcal{H}_k^{(l)} \subset \mathcal{H}_k^{(l+1)} \cdots $$
By induction on $n$, we recover, for every non degenerate $f : [n] \rightarrow \mathbb{N}$, a filtration
of  $\mathcal{H}_f$ by an increasing sequence of finite dimensional subspaces
$$\mathcal{H}_f^{(1)} \subset \mathcal{H}_f^{(2)} \subset \cdots \subset \mathcal{H}_f^{(l)} \subset \mathcal{H}_f^{(l+1)} \cdots $$
which is identical to the preceding one when $n=0$, and such that
$\varphi_f\left(\underset{g \in \partial(f)}{\bigoplus}\mathcal{H}_g^{(l)}\right)=\mathcal{H}_f^{(l)}$
for every positive integer $l$, when $n>0$. \\
For every positive integer $l$, we finally set
$$\mathcal{H}^{(l)} :=\overline{\underset{f \in C}{\overset{\bot}{\oplus}}\mathcal{H}^{(l)}_f}.$$
This defines a filtration of $\mathcal{H}$ by an increasing sequence of subspaces
$$\mathcal{H}^{(1)} \subset \mathcal{H}^{(2)} \subset \cdots \subset \mathcal{H}^{(l)} \subset \mathcal{H}^{(l+1)} \cdots $$

\vskip 2mm\label{hnuf}
Let $m\in \mathbb{N}$, $f: [m] \rightarrow \mathbb{N}$ be an increasing map, $u\in S(m)$, and $n\in \mathbb{N}$. Set
$$\mathcal{H}_{n,u,f}:= \underset{(i,\delta)\in [n-1] \times \Lat(u)}{\bigoplus} \mathcal{H}_{\delta_*(f)+i} \subset  \mathcal{H},$$
and, for every $l \in \mathbb{N}^*$, $\mathcal{H}^{(l)}_{n,u,f}:= \mathcal{H}_{n,u,f} \cap \mathcal{H}^{(l)}$. \\
Assume $u$ is non-trivial, and write $u=\delta_1<\cdots<\delta_{k-1}<[m]$. Then
$$\mathcal{H}_{n,u',f} \subset \underset{i=0}{\overset{n-1}{\bigoplus}}\left[\underset{g \in \partial (f+i)}{\bigoplus} \mathcal{H}_g\right],$$
and it follows that the condition
$$\forall (i,g)\in [n-1] \times \Lat(u'), \; \forall x \in \mathcal{H}_{\delta_*((\delta_{k-1})_*(f))+i},
\quad \varphi_{n,u,f}(x)=\varphi_{f+i}(x)$$
defines an isometry
$$\varphi_{n,u,f} :  \mathcal{H}_{n,u',(\delta_{k-1})_*(f)} \hookrightarrow \underset{i=0}{\overset{n-1}{\oplus}}\mathcal{H}_{f+i}=\mathcal{H}_{n,[m],f}$$
which is compatible with the respective filtrations of $\mathcal{H}_{n,u',(\delta_{k-1})_*(f)}$ and $\mathcal{H}_{n,[m],f.}$ \\
Obviously, if $v$ is another non-trivial class in $S(m)$ such that $u \subset v$, where, $v=\delta'_1<\cdots<\delta'_{k'-1}<[m]$,
then $\mathcal{H}_{n,u',(\delta_{k-1})_*(f)} \subset \mathcal{H}_{n,v',(\delta'_{k'-1})_*(f)}$, and
$(\varphi_{n,v,f})_{|\mathcal{H}_{n,u',(\delta_{k-1})_*(f)}}=\varphi_{n,u,f}$. \\
Finally, if $n'$ is another non-negative integer, then
$$\mathcal{H}_{n,u,f} \overset{\bot}{\oplus} \mathcal{H}_{n',u,f+n}=\mathcal{H}_{n+n',u,f},$$
$$\mathcal{H}_{n,u',(\delta_{m-1})_*(f)} \overset{\bot}{\oplus} \mathcal{H}_{n',u',(\delta_{m-1})_*(f)+n}=\mathcal{H}_{n+n',u',(\delta_{m-1})_*(f),}$$
and
$$\forall (x,y) \in \mathcal{H}_{n,u',(\delta_{m-1})_*(f)} \times \mathcal{H}_{n',u',(\delta_{m-1})_*(f)+n},
\quad \varphi_{n+n',u,f}(x+y)=\varphi_{n,u,f}(x)+\varphi_{n',u,f+n}(y).$$

\subsection{Universal paths between orthonormal bases}\label{10.4.4}

\subsubsection{The simplicial space $(V_n(m))_{m\in \mathbb{N}}$}

For every $l \in \mathbb{N}^*$, every linear subspace $\mathcal{H'}$ of $\mathcal{H}^{(l)}$, and every non-negative integer $n$,
we denote by $V_n(\mathcal{H'})$ the subspace of $(\mathcal{H}')^n$ consisting of orthonormal $n$-tuples,
which we consider as a $U_n(F)$-space with the canonical right-action of $U_n(F)$.

For any linear subspace $\mathcal{H'}$ of $\mathcal{H}$ and any integer $n$, we denote by $V_n(\mathcal{H}')$
the subset of $(\mathcal{H}')^n$ consisting of orthonormal $n$-tuples, equipped with the
final topology for the canonical map
$$\underset{l \in \mathbb{N}^*}{\coprod}V_n(\mathcal{H}'\cap \mathcal{H}^{(l)}) \longrightarrow V_n(\mathcal{H}').$$
This definition is clearly compatible with the preceding one, and the canonical
right-action of $U_n(F)$ on $V_n(\mathcal{H'})$ is clearly continuous.
Moreover, given two subspaces $\mathcal{H}'$ and  $\mathcal{H}''$ of $\mathcal{H}$ such that
$\mathcal{H}' \subset \mathcal{H}''$, then the inclusion $V_n(\mathcal{H}') \subset V_n(\mathcal{H}'')$ is an immersion.

For every $n \in \mathbb{N}$, $k \in \mathbb{N}$ and $l \in \mathbb{N}$, we define $V_{n,k}^{(l)}$
as $V_n\left(\underset{j=k}{\overset{k+n-1}{\bigoplus}}\mathcal{H}^{(l)}_j \right)$.
For every $n \in \mathbb{N}$ and $k \in \mathbb{N}$, we define
$V_{n,k}$ as $V_n\left(\underset{j=k}{\overset{k+n-1}{\bigoplus}}\mathcal{H}_j \right)$, so that
$V_{n,k}=\underset{l \in \mathbb{N}^*}{\Indlim}\, V_{n,k}^{(l)}$. \\
For every $f:[m] \rightarrow \mathbb{N}$ and $n \in \mathbb{N}$ set
$$V_n(f):=\underset{0 \leq i \leq n}{\prod} V_{n,f(i)}$$
which is $U_n(F)$-space for the diagonal action. \\
For every $l \in \mathbb{N}^*$, set $V_n^{(l)}(f):=\underset{0 \leq i \leq n}{\prod}V^{(l)}_{n,f(i)}$, so that
$V_n(f)=\underset{l \in \mathbb{N}^*}{\Indlim}\, V_n^{(l)}(f)$. \\
If $n$ and $n'$ are non-negative integers, and $f: [m] \rightarrow \mathbb{N}$ is a map,
then $\underset{j=f(i)}{\overset{f(i)+n-1}{\bigoplus}}\mathcal{H}^{(l)}_j$ and
$\underset{j=f(i)+n}{\overset{f(i)+n+n'-1}{\bigoplus}}\mathcal{H}^{(l)}_j$ are orthogonal for every $i \in [m]$ in the case $n>0$; it follows that
the juxtaposition of families yields a canonical injection: $$V_n(f) \times V_{n'}(f+n) \hookrightarrow V_{n+n'}(f).$$

\vskip 2mm
For every $m\in \mathbb{N}$, we define
$$V_n(m):=\underset{f\in \Hom_\uparrow([m],\mathbb{N})}{\coprod} V_n(f)$$
as a $U_n(F)$-space, with the previous action on every component. \\
For every $n\in \mathbb{N}$, and every $f : [m] \rightarrow \mathbb{N}$, every morphism $\tau : [m'] \rightarrow [m]$ in
the category $\Delta$ induces a $U_n(F)$-map:
$$\tau_* : \begin{cases}
V_n(f) & \longrightarrow V_n(f \circ \tau) \\
(x_i)_{0 \leq i \leq m} & \longmapsto (x_{\tau(i)})_{0 \leq i \leq m'.}
\end{cases}$$
This defines a structure of simplicial $U_n(F)$-space on $(V_n(m))_{m\in \mathbb{N}}$, for every $n \in \mathbb{N}$.

\vskip 2mm
For every $n\in \mathbb{N}$ and every compact space $K$,
the structure of cosimplicial space of $(\Delta^m)_{m \in \mathbb{N}}$ induces a structure
of simplicial $U_n(F)$-space on $(\Hom(\Delta^m \times K,V_n(\mathcal{H})))_{m \in \mathbb{N}}$ where, for every $m \in \mathbb{N}$,
$\Hom(\Delta^m \times K,V_n(\mathcal{H}))$ denotes the space of continuous maps from $\Delta^m \times K$
to $V_n(\mathcal{H})$ with the compact-open topology.

\subsubsection{The main goal}

Our main goal in the rest of the section is to construct, for every $n \in \mathbb{N}$, a morphism of
simplicial $U_n(F)$-spaces
$$\psi_n: (V_n(m))_{m \in \mathbb{N}} \longrightarrow (\Hom(\Delta^m,V_n(\mathcal{H})))_{m \in \mathbb{N}}$$

which fullfills a certain list of requirements.

Such a morphism is simply the data of a family $(\psi_{n,m})_{m\in \mathbb{N}}$
of continuous maps, with $\psi_{n,m} : V_n(m) \times \Delta^m \longrightarrow V_n(\mathcal{H})$
for every $m \in \mathbb{N}$, which satisfies a set of compatibility conditions.

Assuming that we have built two such families of morphisms $\psi=(\psi_n)_{n\in \mathbb{N}}$ and
$\psi'=(\psi'_n)_{n\in \mathbb{N}}$, we also want to build a ``homotopy'' from $\psi$ to $\psi'$, i.e.
a family $\Psi=(\Psi_n)_{n\in \mathbb{N}}$ of morphisms, such that
$$\Psi_n: (V_n(m))_{m \in \mathbb{N}} \longrightarrow (\Hom(\Delta^m \times I,V_n(\mathcal{H})))_{m \in \mathbb{N}}$$
is a morphism of simplicial $U_n(F)$-spaces for every $n \in \mathbb{N}$, which satisfies a set of compatibility conditions, and
such that, for every  $(n,m) \in \mathbb{N}^2$, $(\Psi_{n,m})_{|V_n(m) \times \{0\}}=\psi_{n,m}$
and $(\Psi_{n,m})_{|V_n(m) \times \{1\}}=\psi'_{n,m}$.

\vskip 2mm
We will actually describe a construction that will fulfill both needs at once.
We fix $p \in \mathbb{N}$, and we assume that we have a family $(\psi_n^\partial)_{n\in \mathbb{N}}$, such that
$$\psi_n^\partial : (V_n(m))_{m \in \mathbb{N}} \longrightarrow (\Hom(\Delta^m \times \partial \Delta^p,V_n(\mathcal{H})))_{m \in \mathbb{N}}$$
is a morphism of simplicial $U_n(F)$-spaces for every $n \in \mathbb{N}$, which satisfies compatibility conditions (i) to (vi) detailed
in the next paragraph.
We want to construct a family $(\psi_n)_{n\in \mathbb{N}}$, such that
$$\psi_n : (V_n(m))_{m \in \mathbb{N}} \longrightarrow (\Hom(\Delta^m \times \Delta^p,V_n(\mathcal{H})))_{m \in \mathbb{N}}$$
is a morphism of simplicial $U_n(F)$-spaces for every $n \in \mathbb{N}$, which satisfies compatibility conditions (i) to (vi) detailed in the next paragraph, and, for
every $(n,m)\in \mathbb{N}^2$, $(\psi_{n,m})_{|V_n(m) \times \Delta^m \times \partial \Delta^p}=\psi_{n,m}^\partial$.

\vskip 2mm
In the case $p=0$, $\psi_n^\partial$ is trivial, and the construction
will yield the family $(\psi_n)_{n \in \mathbb{N}}$
we are looking for. In the case $p=1$, if we have two sequences of morphisms $\psi$ and $\psi'$, they define a family of morphisms
$(\Psi_n^\partial)_{n\in \mathbb{N}}$, with
$$\Psi_n^\partial : (V_n(m))_{m \in \mathbb{N}} \longrightarrow
(\Hom(\Delta^m \times \{0,1\},V_n(\mathcal{H})))_{m \in \mathbb{N}},$$
and the construction of $\Psi$ from $\Psi^\partial$ will yield a homotopy from $\psi$ to $\psi'$.

\subsubsection{The compatibility conditions}

Let $K$ be a compact space, and $\psi=(\psi_n)_{n\in \mathbb{N}}$ be a family such that
$\psi_n : (V_n(m))_{m \in \mathbb{N}} \longrightarrow \Hom(\Delta^n \times K,V_n(\mathcal{H}))$
is a morphism of simplicial $U_n(F)$-spaces for every $n\in \mathbb{N}$. We define the following conditions on $\psi$, some of which
depend on three integers $n$, $n'$ and $m$, and a non-decreasing map $f : [m] \rightarrow \mathbb{N}$.

\begin{enumerate}[(i)]
\item $\psi_{n,0}: V_n(0) \times \Delta^0 \times K \rightarrow V_n(\mathcal{H})$
is the composite of the projection on the first factor and the map
$V_n(0) \rightarrow V_n(\mathcal{H})$ induced by the inclusion
$\underset{i \in \mathbb{N}}{\oplus} \mathcal{H}_i \subset \mathcal{H}$.

\vskip 3mm
\item
$$\forall (\mathbf{B},\mathbf{B}',t)\in V_n(f) \times V_{n'}(f+n) \times (\Delta^m \times K), \quad
\psi_{n,m}(\mathbf{B},t) \,\bot\, \psi_{n',m}(\mathbf{B}',t).$$

\vskip 3mm
\item The diagram
$$\xymatrix{
V_n(f) \times V_{n'}(f+n) \ar[dd] \ar[drr]^{\psi_{n,m} \times \psi_{n',m}}  \\
& &  \Hom(\Delta^m \times K,\mathcal{H}^n\times \mathcal{H}^{n'}) \\
V_{n+n'}(f) \ar[urr]_{\psi_{n+n',m}}
}$$
is commutative.

\vskip 3mm
\item If $g : [k] \rightarrow \mathbb{N}$ denotes the root of $f$ and $\sigma : [m] \twoheadrightarrow [k]$ its reduction, then, for every $u\in S(k)$,
$$\psi_{n,m}\bigl(V_n(f) \times (\sigma^*)^{-1}(\Delta_u) \times K\bigr) \subset V_n(\mathcal{H}_{n,u,g}).$$

\vskip 3mm
\item If $g : [k] \rightarrow \mathbb{N}$ denotes the root of $f$ and $\sigma:[m] \twoheadrightarrow [k]$ its reduction,
and $\forall i \in [k], n_i:=\# \sigma^{-1}\{i\}$, then,
$$\forall \mathbf{B} \in V_n(f),
\forall i \in [k],
\forall t \in \Delta^{n_i-1} \times K, \;
 \psi_{n,n_i-1}((\sigma \natural i)_*(\mathbf{B}),t) \in V_n(\mathcal{H}_{g(i)}),$$
and, for every $\delta : [k'] \hookrightarrow [k]$,
the composite of $\psi_{n,m}$
with \\
$\nu_\sigma : V_n(f) \times (G_\sigma(\delta) \times \Delta^{k'}) \times K
 \overset{\delta_* \times \nu_\sigma \times \id_K}{\longrightarrow} V_k(f) \times \Delta^m \times K$
is the composite map
\begin{multline*}
V_n(f) \times (G_\sigma(\delta) \times \Delta^{k'}) \times K
\overset{(\sigma \natural \delta)_*\times \id}{\longrightarrow}
\left(\underset{i=0}{\overset{k'}{\prod}}(V_n(\mathcal{H}_{g(i)}))^{n_{\delta(i)}}
\times \Delta^{n_{\delta(i)}-1} \times K\right)
\times \Delta^{k'} \times K \\
\overset{\underset{i=0}{\overset{k'}{\prod}}\psi_{n,n_{\delta(i)}-1}}{\longrightarrow}
V_n(\delta_*(g)) \times \Delta^{k'} \times K
\overset{\psi_{n,k'}}{\longrightarrow} V_n(\mathcal{H}).
\end{multline*}
\vskip 3mm
\item In the case $f$ is non-degenerate: for every non-trivial increasing sequence \\
$u=\delta_1<\dots<\delta_{k-1}<[m]$ in $S(m)$, with $\delta_{k-1} : [m'] \hookrightarrow [m]$,
and every quadruple $(\mathbf{B},y,t,z)\in V_n(f) \times \Delta_{u'} \times [0,1] \times K$:
if we set $\mathbf{B}':=\psi_{n,m'}((\delta_{k-1})_*(\mathbf{B}),y,z)$
and $x=r_m(\delta_{k-1}^*(y),t)$, then $\mathbf{B}'\in V_n(\mathcal{H}_{n,u',(\delta_{k-1})_*(f)})$ and
$$\psi_{n,m}(\mathbf{B},x,z)=\cos\left(\frac{\pi}{2}t\right).\mathbf{B}'
+\sin\left(\frac{\pi}{2}t\right).\varphi_{n,u,f}(\mathbf{B}') .$$
\end{enumerate}

\noindent
\textbf{Remarks:}
\begin{itemize}
\item Conditions (ii) and (iii) hold for all $m$ and $f: [m] \rightarrow \mathbb{N}$ when $n=0$ or $n'=0$.
\item When $f$ is constant, condition (iv) simply means that
$\psi_{n,m}(V_n(f) \times \Delta^m \times K) \subset V_{n,f(0).}$
\item Conditions (iv), (v) and (vi) are only there so that we can carry out the construction, and they will be
useless when the construction is over.
\end{itemize}

\subsubsection{Relationships between the conditions}

\begin{itemize}
\item In condition (v), the first requirement holds if and only if condition (iv) holds for every
triple $(n,n_i-1,f \circ (\sigma \natural i))$ with $i\in [k]$.
\item In the case $f$ is constant, condition (v) for $(n,m,f)$ is logically equivalent to condition (iv) for the same triple.
\item In the case $f$ is non-degenerate and (i) holds for $n$,
condition (v) holds if and only if
the square
$$\begin{CD}
V_n(f) @>{\psi_{n,k}}>> \Hom(\Delta^m \times K,V_n(\mathcal{H})) \\
@V{\delta_*}VV @V{\delta_*}VV \\
V_n(\delta_*(f)) @>{\psi_{n,k'}}>> \Hom(\Delta^k \times K,V_n(\mathcal{H}))
\end{CD}$$
is commutative for every $\delta : [k] \hookrightarrow [m]$.
\item In condition (vi), only the second requirement is interesting
since the first one is obviously true when condition (iv) holds for $(n,m',f\circ \delta_{k-1})$.
\item Assume that condition (iv) holds for the two triples $(n,m,f)$ and $(n',m,f)$. \\
Let $g$ be the root of $f$ and $\sigma : [m] \twoheadrightarrow [k]$ its reduction.
Let $(\mathbf{B},\mathbf{B}') \in V_n(f) \times V_{n'}(f+n)$, $t \in \Delta^m$ and $z\in K$.
By Proposition \ref{glue}, we may choose $u\in S(k)$
such that $\sigma^*(t) \in \Delta_u$.

Since condition (iv) is true for both triples $(n,m,f)$ and $(n',m,f)$, we deduce
that $\psi_{n,m}(\mathbf{B},x,z) \in V_n(\mathcal{H}_{n,u,g})$
and $\psi_{n',m}(\mathbf{B}',x,z) \in V_n'(\mathcal{H}_{n',u,g+n})$. Since
$\mathcal{H}_{n,u,g} \bot \mathcal{H}_{n',u,g+n}$, we deduce that
$\psi_{n,m}(\mathbf{B},x,z) \bot \psi_{n',m}(\mathbf{B}',x,z)$.
Hence condition (ii) holds for $(n,n',m,f)$.
\end{itemize}

\subsubsection{Main result}

We may now state our main result:

\begin{prop}\label{mighty}
Let $p \in \mathbb{N}$, and $\psi^\partial=(\psi_n^\partial)_{n\in \mathbb{N}}$ be a family
such that:
\begin{itemize}
\item For every $n \in \mathbb{N}$, $\psi_n^\partial : (V_n(m))_{m\in \mathbb{N}} \longrightarrow
(\Hom(\Delta^m \times \partial \Delta ^p,V_n(\mathcal{H})))_{m\in \mathbb{N}}$
is a morphism of simplicial $U_n(F)$-spaces;
\item Conditions (i) to (vi) are satisfied by $\psi^\partial$ for every compatible $4$-tuple $(n,n',m,f)$.
\end{itemize}
\vskip 2mm
Then there exists a family $\psi=(\psi_n)_{n\in \mathbb{N}}$ such that
\vskip 2mm
\begin{itemize}
\item For every $n \in \mathbb{N}$, $\psi_n : (V_n(m))_{m\in \mathbb{N}} \longrightarrow
(\Hom(\Delta^m \times \Delta ^p,V_n(\mathcal{H})))_{m\in \mathbb{N}}$
is a morphism of simplicial $U_n(F)$-spaces.
\item Conditions (i) to (vi) are satisfied by $\psi$ for every compatible $4$-tuple $(n,n',m,f)$.
\item On has $(\psi_{n,m})_{|V_n(m) \times \Delta^m \times \partial \Delta^p}=\psi_{n,m}^\partial$ for every $(n,m)\in \mathbb{N}^2$.
\end{itemize}
\end{prop}

\subsection{The proof of Proposition \ref{mighty}}\label{10.4.5}

Our proof of Proposition \ref{mighty} will be done by induction.
An essential is played by a general conjecture on relative triangulations that
we were not able to prove: we will begin by stating it and drawing the consequences
that will be necessary in our proof.

\subsubsection{A conjecture on relative triangulations, and some consequences}\label{conjecturesection}

\begin{Def}
Let $M$ be an $n$-dimensional smooth manifold, and $(M_i)_{i \in I}$ be a finite family of closed subspaces of
$M$. For every $x \in M$, set $I_x:=\{i\in I :\; x \in M_i\}$.
We say that the family $(M_i)_{i \in I}$ \textbf{intersects cleanly} if, for every $x \in M$,
there is an open neighborhood $U_x$ of $x$ in $M$, an open neighborhood $V$ of $0$ in $\mathbb{R}^n$,
a family $(F_i)_{i\in I_x}$ of
linear subspaces of $\mathbb{R}^n$, and a smooth diffeomorphism $\varphi : U_x \overset{\cong}{\longrightarrow} V$
such that
$$\forall i \in I \setminus I_x, \quad U_x \cap M_i =\emptyset$$
and
$$\forall J \in \mathcal{P}(I_x), \quad \varphi\left(U \cap \left(\underset{j  \in J}{\cap}M_j\right)\right) =
V \cap\left(\underset{j  \in J}{\cap}F_j\right).$$
\end{Def}

In this case, $(M_j)_{j \in J}$ obviously intersects cleanly for every $J \subset I$,
and the $M_j$'s are all smooth submanifolds of $M$.

\vskip 2mm
\noindent \textbf{Remarks:}
\begin{itemize}
\item Let $M$ be a smooth manifold, $I$ be a finite set, and
$(M_i)_{i \in I}$ be a family of closed smooth submanifolds of
$M$ indexed over $I$. Assume that, for every $x \in M$, there is a smooth manifold $N_x$, a family
$(N_{x,i})_{i \in I}$ of closed smooth submanifolds of $N_x$ which intersects cleanly, an open neighborhood
$U_x$ of $x$ in $M$, an open subset $V_x$ of $N$, and a smooth diffeomorphism
$\varphi_x : U_x \overset{\cong}{\longrightarrow} V_x$ such that
$U_x \cap (\underset{i \in I  \setminus I_x}{\bigcup}M_i)=\emptyset$ and
$\forall i \in I_x,\; \varphi_x(U_x \cap M_i)=V_x \cap N_{i,x}$. Then
$(M_i)_{i \in I}$ intersects cleanly.

\item Let $M$ and $N$ be two smooth manifolds, $I$ a finite set, and $(M_i)_{i\in I}$
(resp.\ $(N_i)_{i\in I}$) a family of closed smooth submanifolds of $M$ (resp.\ of $N$) indexed
over $I$ which intersects cleanly. Then $(M_i \times N_i)_{i \in I}$ intersects cleanly in $M \times N$.
\end{itemize}

\noindent
\textbf{Example 1:}
Let $E$ be an affine variety, and $(E_i)_{i \in I}$ be a finite family of affine subvarieties of $E$. Then
$(E_i)_{i \in I}$ intersects cleanly.

\vskip 2mm
\noindent
\textbf{Example 2:}
Let $G$ be a Lie group and $(H_i)_{i \in I}$ be a finite family of closed subgroups of $G$.
Then $(H_i)_{i \in I}$ intersects cleanly.

\begin{proof}
Let $x \in G$, $I_x:=\{i\in I : x \in H_i\}$, and let
$V$ be an open neighborhood of $x$ in $G \setminus \underset{i \in I \setminus I_x}{\bigcup}H_i$. We set
$\varphi_x : \begin{cases}
G & \longrightarrow G \\
g & \longmapsto g.x^{-1}
\end{cases}$
and choose an open neigborhood $V_1$ of $0$ in $LG$ and an open neighborhood $V_2$ of $1_G$ in $G$
such that $\exp_{|V_1} : V_1 \overset{\cong}{\longrightarrow} V_2$ is a diffeomorphism. We set $V':=(\varphi_x)^{-1}(V_2) \cap V$ and
$\varphi : \begin{cases}
V' & \overset{\cong}{\longrightarrow} V_1 \\
z & \longmapsto  (\exp_{|V_2})^{-1}(\varphi_x(z)).
\end{cases}$ \\
Then, for every $J \subset I_x$, one has
$$\varphi\left(V' \cap \left(\underset{j\in J}{\cap}H_j\right) \right)=V_1 \cap\left(\underset{j\in J}{\cap}LH_j\right).$$
Then Example 1 and the previous remarks show that $(H_i)_{i\in I}$ intersects cleanly.
\end{proof}

We may now state our conjecture:

\begin{conj}[Triangulation conjecture for clean intersections]\label{unionofsubmanifolds}
Let $M$ be a smooth compact manifold and $(M_i)_{i \in I}$ be a finite family of closed subspaces of $M$ which intersects cleanly.
Then the pair $\left(M, \underset{i \in I}{\bigcup}M_i\right)$ is a finite relative CW-complex.
\end{conj}

Let us now draw some important consequences of this conjecture. \\
Let $m$ and $n$ be two positive integers. Let $(E_0,\dots,E_m)$ be an $(m+1)$-tuple of
finite dimensional inner product spaces (with ground field $F$).
For every $k \in [m-1]$, we consider an isometry $\varphi_k : E_k \overset{\cong}{\longrightarrow} E_{k+1}$
For every $k \in [m]$, we consider a linear subspace $F_k$ of $E_k$, and a decomposition
$E_k=\underset{1\leq i \leq n}{\overset{\bot}{\bigoplus}}E_k^{(i)}$
such that: \begin{itemize}
\item for all $k \in [m-1], \; \varphi_k(F_k)=F_{k+1}$.
\item for all $k \in [m-1], \; \forall i \in \{1,\dots,n\}, \varphi_k(E_k^{(i)})=E_{k+1}^{(i)}$.
\item for all $k \in [m], \; F_k= \underset{1\leq i \leq n}{\overset{\bot}{\bigoplus}}F_k^{(i)}$,
where $F_k^{(i)}:=F_k\cap E_k^{(i)}$ for all $i\in \{1,\dots,n\}$.
\end{itemize}
For every $k \in [m]$, every subspace $V$ of $E_k$, and every subset $A$ of $\{1,\dots,n\}$, we define
$$V^{[A]}:=V \cap \bigl(\underset{i \in A}{\oplus}E_k^{(i)}\bigr).$$
Set then
$$\mathcal{M}:= \left(\bigl(E_k,F_k,(E_k^{(i)})_{1 \leq i \leq n}\bigr)_{0 \leq k \leq m},
(\varphi_k)_{0 \leq k \leq m-1} \right).$$

We define $V_N(\mathcal{M})$ as the product space
$\underset{k=0}{\overset{m}{\prod}}V_n(E_k)$ with the diagonal right-action of $U_n(F)$.
For every $(p,q) \in (\mathbb{N}^*)^2$ such that $p+q=n$, we define
$$V^{(p,q)}_n(\mathcal{M}) :=
\left(\underset{k=0}{\overset{m}{\prod}}\left[V_p\left(\underset{1\leq i \leq p}{\overset{\bot}{\bigoplus}}E_k^{(i)}\right)
\times V_q\left(\underset{p+1\leq j \leq n}{\overset{\bot}{\bigoplus}}E_k^{(j)}\right)\right]\right).U_n(F) \subset
V_n(\mathcal{M}),$$ and we set
$$V_n^{\text{prod}}(\mathcal{M}):=\underset{p \geq 1,q \geq 1, p+q=n}{\bigcup}
V^{(p,q)}_n(\mathcal{M}) \subset V_n(\mathcal{M}).$$
For every $k \in [M-1]$, set
$$V_n^k(\mathcal{M}):=\bigl\{(\mathbf{B}_i)_{0 \leq i \leq m} \in V_n(\mathcal{M}) :
\varphi_k(\mathbf{B}_k)=\mathbf{B}_{k+1}\bigr\},$$
$$V_n^{\text{deg}}(\mathcal{M}):=\underset{0 \leq k \leq m-1}{\bigcup}V_n^k(\mathcal{M}) \quad \text{and}
\quad V'_n(\mathcal{M}):=\underset{k=0}{\overset{m}{\prod}}V_n(F_k).$$
We now define $G_n(\mathcal{M})$ as the quotient space $V_n(\mathcal{M})/U_n(F)$. Since
$V_n(\mathcal{M})$ is a smooth compact manifold with a smooth free action of $U_n(F)$, we deduce that
there is a unique structure of smooth manifold on $G_n(\mathcal{M})$ such that the canonical projection
$V_n(\mathcal{M}) \longrightarrow G_n(\mathcal{M})$ is a smooth principal $U_n(F)$-bundle.
For every $(p,q)\in (\mathbb{N}^*)^2$ such that $p+q=n$ (resp.\ for every $k \in [m-1]$), we define
$G_N^{(p,q)}(\mathcal{M})$ as the direct image of $V^{(p,q)}_N(\mathcal{M})$ (resp.\ of $V_N^k(\mathcal{M})$) by the canonical projection
$V_n(\mathcal{M}) \rightarrow G_n(\mathcal{M})$. Finally, we define
$G'_n(\mathcal{M})$ as the direct image of $V'_n(\mathcal{M})$ by the canonical projection.
Obviously, all those subspaces of $G_n(\mathcal{M})$ are compact, and therefore closed.

\vskip 2mm
\noindent \textbf{Remark :}
An element $x$ of $G_n(\mathcal{M})$ may be identified with a diagram $x_0 \overset{f_0}{\rightarrow} x_1 \overset{f_1}{\rightarrow}
 \dots \overset{f_{m-1}}{\rightarrow} x_m$, where, for every $k \in [m]$, $x_k$ is an $n$-dimensional subspace
of $E_k$, and, for every $k \in [m-1]$, $f_k$ is a linear isometry from $x_k$ to $x_{k+1}$. \\
Let $(p,q)\in (\mathbb{N}^*)^2$ such that $p+q=n$. Then $x \in G_n^{(p,q)}(\mathcal{M})$ if and only if, for every
$k \in [m]$, $x_k=x_k^{[\{1,\dots,p\}]} \oplus x_k^{[\{p+1,\dots,n\}]}$,
and, for every $k \in [m-1]$, $f_k\left(x_k^{[\{1,\dots,p\}]}\right)=x_{k+1}^{[\{1,\dots,p\}]}$
and $f_k\left(x_k^{[\{p+1,\dots,n\}]}\right)=x_{k+1}^{[\{p+1,\dots,n\}]}$. \\
Let $k \in [m-1]$. Then $x \in G_n^k(\mathcal{M})$ if and only if $f_k=\varphi_k\big|_{x_k.}^{x_{k+1}}$ \\
Finally, $x \in G'_n(\mathcal{M})$ if and only if $x_k \subset F_k$ for every $k \in [m]$.

\vskip 2mm
\noindent We are now ready to state the two consequences of conjecture \ref{unionofsubmanifolds}
that will be used in the proof of Proposition \ref{mighty}.

\begin{cor}\label{clean}
The family $\left((G_n^{(p,q)}(\mathcal{M}))_{p+q=n},\,(G^k_n(\mathcal{M}))_{0 \leq k \leq m},\, G'_n(\mathcal{M})\right)$
intersects cleanly in $G_n(\mathcal{M})$.
\end{cor}

\begin{proof}
Obviously, the subspaces considered here are all closed in $G_n(\mathcal{M})$ since they are compact. \\
Let $x=x_0 \overset{f_0}{\rightarrow} x_1 \overset{f_1}{\rightarrow}
 \dots \overset{f_{m-1}}{\rightarrow} x_m$ in $G_n(\mathcal{M})$, and set:
$$I_x:=\{(p,q) \in (\mathbb{N}^*)^2 : x \in G_n^{(p,q)}(\mathcal{M})\} \cup \{k \in [m-1] : x \in G_n^k(\mathcal{M})\}$$
and
$$J_x:=\{(p,q) \in (\mathbb{N}^*)^2 : x \not\in G_n^{(p,q)}(\mathcal{M})\} \cup \{k \in [m-1] : x \not\in G_n^k
(\mathcal{M})\}.$$
Denote by $U_x$ the subset of $G_n(\mathcal{M})$ consisting of those elements
$x'=x'_0 \overset{f'_0}{\rightarrow} x'_1 \overset{f'_1}{\rightarrow} \dots \overset{f'_{m-1}}{\rightarrow} x'_m$
such that $x'_k \cap x_k^\bot =\{0\}$ for all $k\in [m]$. If $x'$ is such an element, we
set $\psi_{x_k,x'_k}:= \pi_{x'_k} \circ (\pi_{x_k}^{x'_k} \circ \pi_{x'_k}^{x_k})^{-\frac{1}{2}}$ for every $k \in [m]$
(this is a well-defined isometry from $x_k$ to $x'_k$). Obviously, $U_x$ is an open neighborhood of $x$ in $G_N(\mathcal{M})$. \\
We finally set
$$\overline{M}:=\underset{k=0}{\overset{m}{\prod}}L(x_k,x^\bot_k) \quad ,
\quad \overline{N}:=\underset{k=0}{\overset{m-1}{\prod}}U(x_k)$$
and
$$\varphi_x : \begin{cases}
U_x & \longrightarrow \overline{M} \times \overline{N}  \\
x'_0 \overset{f'_0}{\rightarrow} x'_1 \overset{f'_1}{\rightarrow} \dots \overset{f'_{m-1}}{\rightarrow} x'_m
& \longmapsto
\left(\left(\pi_{x_k^\bot} \circ (\pi_{x_k}^{x'_k})^{-1}\right)_{k \in [m]},
\left(f_k^{-1} \circ (\psi_{x_{k+1},x'_{k+1}})^{-1}
\circ f'_k \circ \psi_{x_k,x'_k}\right)_{k \in [m-1]}\right).
\end{cases}$$
For every $(p,q)\in I_x$, we set
$$M_{(p,q)}:=\underset{k=0}{\overset{m}{\prod}}
\left(L\left(x^{[\{1,\dots,p\}]}_k,(x^\bot_k)^{[\{1,\dots,p\}]}\right)\oplus
L\left(x^{[\{p+1,\dots,n\}]}_k,(x^\bot_k)^{[\{p+1,\dots,n\}]}\right)\right)
$$
and
$$N_{(p,q)}:=\underset{k=0}{\overset{m-1}{\prod}}\left(U(x_k^{[\{1,\dots,p\}]})\times
U(x_k^{[\{p+1,\dots,n\}]})\right).$$
For every $k \in I_x$, we set
$$M_k:=\left\{(\alpha_i)_{0\leq i \leq m}\in \overline{M} : \;\alpha_k=\alpha_{k+1}\right\},$$
and
$$N_k:=\left\{(\beta_i)_{0\leq i \leq m-1}\in \overline{N} : \; \beta_k=\id_{x_k}\right\}.$$
Finally, we set $M':=\underset{k=0}{\overset{m}{\prod}}L(x_k,x^\bot_k \cap F_k)$.
By Example 1, $((M_i)_{i \in I_x},M')$ intersects cleanly in $\overline{M}$.
By Example 2, $((N_i)_{i \in I_x},N)$ intersects cleanly in $\overline{N}$.
Therefore $((M_i \times N_i)_{i \in I_x},M' \times \overline{N})$
is a family of subsets of $\overline{M} \times \overline{N}$
which intersects cleanly.
It is then a straightforward task to check that $\varphi_x(U_x \cap G_n^i(\mathcal{M}))=M_i \times N_i$ for
every $i \in I_x$, and $\varphi_x(U_x \cap G_n'(\mathcal{M}))=M' \times \overline{N}$ if $x \in G_n'(\mathcal{M})$. \\
Set finally
$$V_x:=
\begin{cases}
G_N(\mathcal{M}) \setminus \underset{j \in J_x}{\bigcup}G_N^j(\mathcal{M}) & \text{if} \quad x \in G'_N(\mathcal{M}) \\
G_N(\mathcal{M}) \setminus \left(G'_N(\mathcal{M}) \cup \underset{j \in J_x}{\bigcup}G_N^j(\mathcal{M})\right) &
\text{otherwise}.
\end{cases}$$
By restricting $\varphi_x$ to the open neighborhood $U_x \cap V_x$ of $x$ in $G_n(\mathcal{M})$,
we deduce from a previous remark that
the family $\left((G_n^{(p,q)}(\mathcal{M}))_{p+q=n},(G^k_n(\mathcal{M}))_{0 \leq k \leq m}, G'_n(\mathcal{M})\right)$
of subsets of $G_n(\mathcal{M})$ intersects cleanly. \end{proof}

\begin{cor}\label{CWconjecture}
$(V_n(\mathcal{M}),V_n^{\text{prod}}(\mathcal{M}) \cup V_n^{\text{deg}}(\mathcal{M}) \cup V_n'(\mathcal{M}))$
and $(V_n(\mathcal{M}),V_n^{\text{prod}}(\mathcal{M}) \cup V_n'(\mathcal{M}))$
are finite relative $U_n(F)$-CW-complexes.
\end{cor}

\begin{proof}
By Corollary \ref{clean}, each pair
$\bigl(G_n(\mathcal{M}),G_n^{\text{prod}}(\mathcal{M}) \cup G_n^{\text{deg}}(\mathcal{M}) \cup G_n'(\mathcal{M})\bigr)$
and $\bigl(G_n(\mathcal{M}),G_n^{\text{prod}}(\mathcal{M}) \cup G_n'(\mathcal{M})\bigr)$
is a finite relative CW-complex. Since
the respective inverse images of $G_n^{\text{prod}}(\mathcal{M}) \cup G_n^{\text{deg}}(\mathcal{M}) \cup
G_n'(\mathcal{M})$ and $G_n^{\text{prod}}(\mathcal{M}) \cup G_n'(\mathcal{M})$ by the canonical projection
$V_n(\mathcal{M}) \rightarrow G_n(\mathcal{M})$ are
$V_n^{\text{prod}}(\mathcal{M}) \cup V_n^{\text{deg}}(\mathcal{M}) \cup V_n'(\mathcal{M})$ and
$V_n^{\text{prod}}(\mathcal{M}) \cup V_n'(\mathcal{M})$, and since the projection is a $U_n(F)$-principal bundle
with a compact total space,
there exists some barycentric subdivision of each relative CW-complex structure that lifts, and
we can thus define a structure of finite relative
$U_n(F)$-CW-complex for each pair
$\bigl(V_n(\mathcal{M}),V_n^{\text{prod}}(\mathcal{M}) \cup V_n^{\text{deg}}(\mathcal{M}) \cup V_n'(\mathcal{M})\bigr)$
and $\bigl(V_n(\mathcal{M}),V_n^{\text{prod}}(\mathcal{M}) \cup V_n'(\mathcal{M})\bigr)$.
\end{proof}

We finish with another technical result. Let $(p,q,r)\in (\mathbb{N}^*)^3$ such that $p+q+r=n$. We define
$$V_N^{(p,q,r)}(\mathcal{M}) :=
\left(\underset{k=0}{\overset{M}{\prod}}\left[V_p\left(\underset{1\leq i \leq p}{\overset{\bot}{\bigoplus}}E_k^{(i)}\right)
\times V_q\left(\underset{p+1\leq i \leq p+q}{\overset{\bot}{\bigoplus}}E_k^{(i)}\right)
\times V_r\left(\underset{p+q+1\leq i \leq n}{\overset{\bot}{\bigoplus}}E_k^{(i)}\right)
\right]\right).U_n(F).$$

\begin{prop}\label{productintersection}
Let $(p,q) \in (\mathbb{N}^*)^2$ and $(p',q') \in (\mathbb{N}^*)^2$ such that
$n=p+q=p'+q'$ and $p<p'$. Then: $V_n^{(p,q)}(\mathcal{M}) \cap V_n^{(p',q')}(\mathcal{M}) = V_n^{(p,p'-p,q')}(\mathcal{M})$.
\end{prop}

\begin{proof}
We let $G_n^{(p,p'-p,q')}(\mathcal{M})$ denote the image
of $V_n^{(p,p'-p,q')}(\mathcal{M})$
by the canonical projection $\pi : V_n(\mathcal{M}) \rightarrow G_n(\mathcal{M})$.
It suffices to check that $G_n^{(p,q)}(\mathcal{M}) \,\cap\, G_n^{(p',q')}(\mathcal{M})
=G_n^{(p,p'-p,q')}(\mathcal{M})$.

Let $x=x_0 \overset{f_0}{\rightarrow} x_1 \overset{f_1}{\rightarrow}
 \dots \overset{f_{m-1}}{\rightarrow} x_m$ be an element of $G_n(\mathcal{M})$.
Then $x \in G_n^{(p,p'-p,q')}$ if and only if,
$x_k=x_k^{[\{1,\dots,p\}]} \oplus x_k^{[\{p+1,\dots,p'\}]} \oplus  x_k^{[\{p'+1,\dots,n\}]}$ for every $k \in [m]$, and,
for every $k \in [m-1]$, one has
$$f_k\bigl(x_k^{[\{1,\dots,p\}]}\bigr) =x_{k+1}^{[\{1,\dots,p\}]}, \quad
f_k\bigl(x_k^{[\{p+1,\dots,p'\}]}\bigr) =x_{k+1}^{[\{p+1,\dots,p'\}]} \quad \text{and} \quad
f_k\bigl(x_k^{[\{p'+1,\dots,n\}]}\bigr) =x_{k+1}^{[\{p'+1,\dots,n\}]}.$$
We are thus reduced to the following easy lemma. \end{proof}

\begin{lemme}
Let $E$ be a vector space, and $E=E_1 \oplus E_2 \oplus E_3$
be a decomposition of $E$. Let $F$ be a linear subspace of $E$
such that $F=(F\cap E_1)\oplus (F\cap(E_2 \oplus E_3))$ and
$F= (F\cap(E_1 \oplus E_2))\oplus (F\cap E_3)$.
Then $F=(F\cap E_1) \oplus (F\cap E_2) \oplus (F\cap E_3)$.
\end{lemme}

\begin{proof}
We denote by $\pi_1$ (resp.\ $\pi_2$, resp.\ $\pi_3$) the projection on $E_1$ alongside
$E_2 \oplus E_3$ (resp.\ on $E_2$ alongside $E_1 \oplus E_3$, resp.\ on $E_3$ alongside $E_1\oplus E_2$). \\
Then $\pi_1+\pi_2$ is the projection on $E_1\oplus E_2$ alongside $E_3$,
and $\pi_2+\pi_3$ is the projection on $E_2\oplus E_3$ alongside $E_1$.
The assumptions on $F$ show that, for any $x \in F$,
all the vectors $\pi_1(x)$, $(\pi_2+\pi_3)(x)$, $\pi_3(x)$ and $(\pi_1+\pi_2)(x)$ belong to $F$.
It follows that $\forall x \in F,\; (\pi_1(x),\pi_2(x),\pi_3(x))\in F^3$, which yields the claimed result.
\end{proof}

\subsubsection{Starting the induction}

We let $p \in \mathbb{N}$, and $\psi^\partial=(\psi_n^\partial)_{n\in \mathbb{N}}$ be a family which satisfies
the conditions of conjecture \ref{mighty}.

We are going to construct a family $(\psi_n)_{n \in \mathbb{N}}$ by a double induction process on both
$m$ and $n$. Since $V_0(\mathcal{H})=*$, we define $\psi_{0,m}$ for every $m \in \mathbb{N}$ as the trivial map . This yields
a morphism $\psi_0$ of simplicial $U_0(F)$-spaces which clearly satisfies conditions (i) to (vi), and
of course $\psi_0^\partial$ is a restriction of $\psi_0$.

For every non negative integer $n \in \mathbb{N}$, we define $\psi_{n,0}$ by condition (i), then conditions (iii), (iv), (v) and (vi)
are easily seen to be true, and, since $\psi^\partial_{n,0}$ also satisfies condition (i),
$\psi^\partial_{n,0}$ is a restriction of $\psi_{n,0}$.

\subsubsection{The induction hypothesis}
We now fix a pair $(N,M) \in (\mathbb{N}^*)^2$,
we define $I(N,M):=\bigl([N-1] \times \mathbb{N}\bigr) \cup \bigl(\{N\} \times [M-1]\bigr)$, and we assume that we have a
family $(\psi_{n,m})_{(n,m)\in I(N,M)}$ such that: \vskip 2mm
\begin{itemize}
\item $\psi_{n,m} : V_n(m) \times \Delta^m \times \Delta^p \longrightarrow V_n(\mathcal{H})$
is a $U_n(F)$-map for every $(n,m)\in I(N,M)$;
\item $\psi_n$ is a morphism of simplicial sets for every $n<N$;
\item For every $(m,m')\in [M-1]^2$ and every morphism $\tau : [m'] \rightarrow [m]$ in $\Delta$, the
square
$$\begin{CD}
V_n(m) @>{\psi_{n,m}}>> \Hom(\Delta^m \times \Delta^p,V_n(\mathcal{H})) \\
@VV{\tau_*}V @VV{\tau_*}V \\
V_n(m') @>{\psi_{n,m'}}>> \Hom(\Delta^{m'} \times \Delta^p,V_n(\mathcal{H}))
\end{CD}$$
is commutative;
\item The restriction of $\psi_{n,m}$ to $V_n(m) \times \Delta^m \times \partial\Delta^p$ is
$\psi_{n,m}^\partial$ for every $(n,m)\in I(N,M)$;
\item Condition (i) is satisfied;
\item Condition (ii) is satisfied for every $4$-tuple $(n,n',m,f)$ such that $(n,m) \in I(N,M)$ and
$(n',m) \in I(N,M)$;
\item Condition (iii) is satisfied for every $4$-tuple $(n,n',m,f)$ such that $(n,m) \in I(N,M)$,
$(n',m) \in I(N,M)$ and $n+n'\leq N$;
\item Conditions (iv), (v) and (vi) are satisfied for every triple $(n,m,f)$ such that $(n,m) \in I(N,M)$.
\end{itemize}

\subsubsection{The requirements}

We need to build a $U_N(F)$-map
$\psi_{N,M}: V_N(M) \times \Delta^M \times \Delta^p \longrightarrow V_N(\mathcal{H})$
such that
\vskip 2mm
\begin{itemize}
\item The restriction of $\psi_{N,M}$ to $V_N(M) \times \Delta^M \times \partial\Delta^p$ is $\psi_{N,M}^\partial$;
\item For every $i \in [M-1]$, the square
$$\begin{CD}
V_N(M-1) @>{\psi_{N,M-1}}>> \Hom(\Delta^{M-1} \times \Delta^p,V_N(\mathcal{H})) \\
@VV{s_i^M}V @VV{s_i^M}V \\
V_N(M) @>{\psi_{N,M}}>> \Hom(\Delta^M \times \Delta^p,V_N(\mathcal{H}))
\end{CD}$$
is commutative;
\item For every $i \in [M]$, the square
$$\begin{CD}
V_N(M) @>{\psi_{N,M}}>> \Hom(\Delta^M \times \Delta^p,V_N(\mathcal{H})) \\
@VV{d_i^{M-1}}V @VV{d_i^{M-1}}V \\
V_N(M-1) @>{\psi_{N,M-1}}>> \Hom(\Delta^{M-1} \times \Delta^p,V_N(\mathcal{H}))
\end{CD}$$
is commutative.
\item Condition (iii) is satisfied for every $4$-tuple
$(n,n',m,f)$ such that $m=M$, $n \leq N$, $n' \leq N$ and $n+n'= N$.
\item Conditions (iv), (v) and (vi) are satisfied by the family $(\psi_{n,m})_{(n,m)\in I(N,M+1)}$
for every triple $(n,m,f)$ with
$(n,m)=(N,M)$.
\end{itemize}

\subsubsection{The basic strategy to construct $\psi_{N,M}$}

We fix a map $f : [M] \rightarrow \mathbb{N}$, and we build
$\psi_{N,M}$ on $V_N(f)\times \Delta^M \times \Delta^p$. We have to distinguish between three cases,
whether $f$ is degenerate but non-constant, $f$ is constant, or $f$ is non-degenerate.

In the first case, the definition is completely forced by condition (v), and all that needs to be done is check that
the requirements are satisfied by this definition. In the other two cases, the requirements force the definition on
subsets of $V_N(f)\times \Delta^M \times \Delta^p$. What we first do is check that those definitions are compatible. We then find that
they help define $\psi_{N,M}$ on a subset of $V_N(f)\times \Delta^M \times \Delta^p$. In order to complete the definition, we use
an extension argument which relies on practical consequences of our conjecture on triangulations.

\subsubsection{The case $f$ is degenerate and non-constant}

We assume here that $f$ is degenerate and non-constant. We denote by $g : [k] \rightarrow \mathbb{N}$
its root and by $\sigma : [M] \twoheadrightarrow [k]$ its reduction, and for all $i \in [k]$, we set $n_i:=\# \sigma^{-1}\{i\}$.
Notice, since $0<k<n$, that the definition of
$\psi_{N,M}$ on $V_N(f)\times \Delta^M \times \Delta^p$ is forced by condition (v) for $(N,M,f)$.

Since $k>0$, we have $n_i-1 <M$ for all $i\in [k]$. Since
$(\sigma \natural i)_*(f)$ is constant and its value is $g(i)$ for every $i \in [k]$,
we deduce from condition (iv) applied to $(N,n_i,(\sigma \natural i)_*(f))$ that
the first requirement in condition (v) holds for $(N,M,f)$. \\
It follows that we may define a map
\begin{multline*}
V_N(f) \times (G_\sigma(\delta) \times \Delta^{k'}) \times \Delta^p
\Right5{(\sigma \natural \delta)_*\times \id}
\left(\underset{i=0}{\overset{k'}{\prod}}(V_N(\mathcal{H}_{g(\delta(i))}))^{n_{\delta(i)}}
\times \Delta^{n_{\delta(i)}-1} \times \Delta^p \right)
\times \Delta^{k'} \times \Delta^p \\
\Right5{\underset{i=0}{\overset{k'}{\prod}}\psi_{N,n_{\delta(i)}-1}}
V_N(\delta_*(g)) \times \Delta^{k'} \times \Delta^p
\overset{\psi_{N,k'}}{\longrightarrow} V_N(\mathcal{H})
\end{multline*}
for every $\delta : [k'] \hookrightarrow [k]$. By the induction hypothesis applied to
$\psi_{N,k'}$ and $\psi_{N,n_i-1}$ for every $i \in [k']$, this map
is a $U_N(F)$-map. \\
Let $\varphi=\xymatrix{
[k'] \ar@{^{(}->}[dr]^{\delta} \ar@{^{(}->}[rr]^{\delta''} & & [k''] \ar@{_{(}->}[dl]^{\delta'} \\
& [k]}$ be a morphism in $\Delta^*\downarrow [k]$. Let then $\mathbf{B} \in V_N(f)$, and for every $i \in [k]$,
set $\mathbf{B}_i:=(\sigma \natural i) _*(\mathbf{B})$. Let also $(t_0,\dots,t_{k'})\in \Delta^{k'}$,
$\mathbf{t}=(\mathbf{t}_0,\dots,\mathbf{t}_{k''}) \in G_\sigma(\delta')$ and $x \in \Delta^p$.
Then $(\delta'')_*(\mathbf{t})=(\mathbf{t}_{\delta''(i)})_{0 \leq i \leq k'}$. \\
By the induction hypothesis, the square
$$\begin{CD}
V_N(k'') @>{\psi_{N,k''}}>> \Hom(\Delta^{k''} \times \Delta^p,V_N(\mathcal{H})) \\
@VV{\delta''_*}V @VV{\delta_*''}V \\
V_N(k') @>{\psi_{N,k'}}>> \Hom(\Delta^{k'} \times \Delta^p,V_N(\mathcal{H}))
\end{CD}$$
is commutative, and we deduce that
 \begin{align*}
\psi_{N,k'}\left(\left(\psi_{N,n_{\delta(i)}-1}\left(\mathbf{B}_{\delta(i)},\mathbf{t}_{\delta''(i)},x\right)\right)_{0 \leq i  \leq k'}
,t,x \right) & = \psi_{N,k'}\left((\delta'')_*\left((\psi_{N,n_{\delta'(i)}-1}
\left(\mathbf{B}_{\delta'(i)},\mathbf{t}_i,x)\right)_{0 \leq i  \leq k''}
,t,x \right)\right) \\
& = \psi_{N,k''}\left(\left(\psi_{N,n_{\delta'(i)}-1}
(\mathbf{B}_{\delta'(i)},\mathbf{t}_i,x)\right)_{0 \leq i  \leq k''}
,(\delta'')^*(t),x \right).
\end{align*}
Hence the previous maps are compatible, and it follows from Proposition \ref{superglue}
that they yield an equivariant map $\psi_{N,M} : V_N(f) \times \Delta^M \times \Delta^p \longrightarrow V_N(\mathcal{H})$.
Since $V_N(f)$ is filtered by an increasing sequence of compact spaces, we also deduce from
Proposition \ref{superglue} that $\psi_{N,M}$ is actually a continuous map.

\vskip 2mm
Since condition (v) is now checked for $(n,M,f)$ for any $n\leq N$, it follows from the induction hypothesis\footnote{
and, more specifically,
from the part concerning condition (iii), applied to $(n,n',k,g)$, and
$(n,n',n_i-1,(\sigma \natural i)_*(f))$, for every $i\in [k]$ and
every $(n,n')$ such that $n+n'=N$.}
that condition (ii) holds for
$(n,n',M,f)$, for every $(n,n')$ such that $n+n'=N$.

\vskip 2mm
Since condition (v) is satisfied by both $\psi_{N,M}$ and $\psi_{N,M}^\partial$, it follows
from the induction hypothesis\footnote{and, more specifically, the part concerning the compatibility of $\psi$ with $\psi^\partial$
for $(N,k)$, and $(N,n_i-1)$ for every $i\in [k]$.}
that the restriction of $\psi_{N,M}$ to $V_N(f) \times \Delta^M \times \partial \Delta^p$ is $\psi^\partial_{N,M}$.

\vskip 2mm
Let $u \in S(k)$, $\mathbf{B} \in V_N(f)$, $t \in (\sigma^*)^{-1}(\Delta_u)$ and $x \in \Delta^p$.
Let $\mathbf{t}=(\mathbf{t}_i)_{0\leq i \leq n} \in \underset{i=0}{\overset{k}{\prod}}\Delta^{n_i-1}$ such that
$t=\lambda_\sigma(\mathbf{t},\sigma^*(t))$. By the induction hypothesis, condition (iv) holds for the triple
$(N,k,g)$; we deduce that
$$\psi_{N,k}\left(\left(\psi_{N,n_i-1}\left(\mathbf{B}_i,\mathbf{t}_i,x\right)\right)_{0 \leq i  \leq k},\sigma^*(t),x \right)
\in V_N(\mathcal{H}_{n,u,g})$$
and we conclude that $\psi_{N,M}(\mathbf{B},t,x) \in V_N(\mathcal{H}_{n,u,g})$. This proves that condition (iv) holds for
the triple $(N,M,f)$.

\vskip 2mm
We finish by checking the compatibility with face and degeneracy maps.
Let $\tau : [m] \rightarrow [M]$ be any morphism in $\Delta$, with $m<M$.
Let $\tau=\delta' \circ \sigma'$ be the decomposition of $\sigma \circ \tau$ into the
composite of an epimorphism and a monomorphism.
$$\xymatrix{
[k'] \ar@{^{(}->}[r]^{\delta'} & [k] \ar@{^{(}->}[r]^{g} & \mathbb{N}. \\
[m] \ar@{>>}[u]^{\sigma'} \ar[r]^{\tau} & [M] \ar@{>>}[u]^{\sigma} \ar[ur]_f
}
$$
Then $g \circ \delta'$ is the root of $f \circ \tau$, and $\sigma'$ is its reduction.
We set $n'_i:=\#(\sigma')^{-1}\{i\}$ for every $i\in [k']$.
Let $(\tau_i)_{0 \leq i \leq k'}$ denote the decomposition of $\tau$ over $\sigma$.
Let $\mathbf{B} \in V_N(f)$.  For every $i \in [k]$,
set $\mathbf{B}_i:=(\sigma \natural i) _*(\mathbf{B})$, and for every $i \in [k']$, set
$\mathbf{B}'_i:=(\sigma' \natural i) _*(\tau_*(\mathbf{B}))=(\tau_i)_*((\delta'\natural \sigma)_*(\mathbf{B}))$.
Let also $(t_0,\dots,t_m)\in \Delta^m$, $\mathbf{t}=(\mathbf{t}_0,\dots,\mathbf{t}_{k'}) \in G_{\sigma'}(\id_{[k']})$ and $x \in \Delta^p$. \\
By the induction hypothesis, the square
$$\begin{CD}
V_N(n_{\delta'(i)}-1) @>{\psi_{N,n_{\delta(i)}-1}}>> \Hom(\Delta^{n_{\delta'(i)}-1} \times \Delta^p,V_N(\mathcal{H})) \\
@VV{(\tau_i)_*}V @VV{(\tau_i)_*}V \\
V_N(n'_i-1) @>{\psi_{N,n'_i-1}}>> \Hom(\Delta^{n'_i-1} \times \Delta^p,V_N(\mathcal{H}))
\end{CD}$$
is commutative for every $i \in [k']$. \\
It follows that
$$\psi_{N,k'}\left(\left(\psi_{N,n'_i-1}\left(\mathbf{B}'_i,\mathbf{t}_i,x\right)\right)_{0 \leq i  \leq k'},t,x \right)
= \psi_{N,k'}\left(\left(\psi_{N,n_{\delta'(i)}-1}\left(((\delta' \natural \sigma)\natural i)_*(\mathbf{B}),\tau_i^*(\mathbf{t}_i)
,x\right)\right)_{0 \leq i  \leq k'},t,x \right),$$
and we deduce from Proposition \ref{splitting} and condition (v) that
$\psi_{N,m}(\tau_*(\mathbf{B}),t',x)=\psi_{N,M}(\mathbf{B},\tau^*(t'),x)$, where
$t'=\lambda_{\sigma'}(\mathbf{t},t) \in \Delta^m$. This proves that the square
$$\begin{CD}
V_N(f) @>{\psi_{N,M}}>> \Hom(\Delta^M \times \Delta^p,V_N(\mathcal{H})) \\
@VV{\tau_*}V @VV{\tau_*}V \\
V_N(\tau_*(f)) @>{\psi_{N,m}}>> \Hom(\Delta^m \times \Delta^p,V_N(\mathcal{H}))
\end{CD}$$
is commutative.

\vskip 2mm
Finally, let $\tau : [M] \twoheadrightarrow [m]$ and $f_1 : [m] \rightarrow \mathbb{N}$ such that
$f =f_1 \circ \tau$. Then $\sigma=\sigma' \circ \tau$, where $\sigma' : [m] \twoheadrightarrow [k]$
is the reduction of $f_1$ (and $g$ is its root):

$$\xymatrix{
[M] \ar@{>>}[r]^{\tau} \ar[rd]_f & [m] \ar@{>>}[r]^{\sigma'} \ar[d]_{f_1} & [k] \ar[dl]_g \\
& \mathbb{N}.
}$$
For every $i\in [k]$, set $n'_i:=\#(\sigma')^{-1}\{i\}$.
Let $(\tau_i)_{0 \leq i \leq k}$ denote the decomposition of $\tau$ over $\sigma'$.
Let $\mathbf{B} \in V_N(f_1)$.  For every $i \in [k]$,
set $\mathbf{B}_i:=(\sigma' \natural i) _*(\mathbf{B})$ and
$\mathbf{B}'_i:=(\sigma \natural i) _*(\tau_*(\mathbf{B}))=(\tau_i)_*(\sigma_*(\mathbf{B}))$.
Let also $(t_0,\dots,t_M)\in \Delta^M$, $\mathbf{t}=(\mathbf{t}_0,\dots,\mathbf{t}_{k}) \in G_{\sigma}(\id_{[k]})$ and $x \in \Delta^p$. \\
By the induction hypothesis, the square
$$\begin{CD}
V_N(n'_i-1) @>{\psi_{N,n'_i-1}}>> \Hom(\Delta^{n'_i-1} \times \Delta^p,V_N(\mathcal{H})) \\
@VV{(\tau_i)_*}V @VV{(\tau_i)_*}V \\
V_N(n_i-1) @>{\psi_{N,n_i-1}}>> \Hom(\Delta^{n_i-1} \times \Delta^p,V_N(\mathcal{H}))
\end{CD}$$
is commutative for every $i \in [k]$.
It follows that
$$\psi_{N,k}\left(\left(\psi_{N,n_i-1}\left(\mathbf{B}'_i,\mathbf{t}_i,x\right)\right)_{0 \leq i  \leq k},t,x \right)
= \psi_{N,k}\left(\left(\psi_{N,n'_i-1}\left(\mathbf{B}_i,\tau_i^*(\mathbf{t}_i),x\right)\right)_{0 \leq i  \leq k},t,x \right),$$
and we deduce from Proposition \ref{splitting} and condition (v) that
$\psi_{N,M}(\tau_*(\mathbf{B}),t',x)=\psi_{N,m}(\mathbf{B},\tau^*(t'),x)$, where
$t'=\lambda_\sigma(\mathbf{t},t) \in \Delta^M$.
This proves that the square
$$\begin{CD}
V_N(f_1) @>{\psi_{N,m}}>> \Hom(\Delta^m \times \Delta^p,V_N(\mathcal{H})) \\
@VV{\tau_*}V @VV{\tau_*}V \\
V_N(f) @>{\psi_{N,M}}>> \Hom(\Delta^M \times \Delta^p,V_N(\mathcal{H}))
\end{CD}$$
is commutative.

\subsubsection{The case $f$ is constant}

We now assume $f$ is constant and we set $n_0:=f(0)$.
The root of $f$ is the map $[0] \rightarrow \mathbb{N}$ whose image is $\{n_0\}$, and its reduction is the
canonical map $[M] \twoheadrightarrow [0]$. For conditions (iv) and (v) to be satisfied, it suffices to
build $\psi_{N,M} : V_N(f) \times \Delta^M \times \Delta^p \longrightarrow V_{N,n_0}$. Notice that condition (vi)
is irrelevant here.

\vskip 2mm
Let $i\in [M]$. We define $\psi_{N,M}$ on $V_N(f) \times (\delta_i^M)^*(\Delta^{M-1}) \times \Delta^p$ by
$$\forall (\mathbf{B},t,x) \in V_N(f) \times \Delta^{M-1} \times \Delta^p, \quad
\psi_{N,M}(\mathbf{B},(\delta_i^{M-1})^*(t),x):=\psi_{N,M-1}(d_i^{M-1}(\mathbf{B}),t,x).$$
By compatibility with the face maps at lower levels and condition (iv) at lower levels,
these definitions are compatible, and
they yield a $U_N(F)$-map
$$\boxed{\psi_{N,M} : V_N(f) \times \partial \Delta^M \times \Delta^p \longrightarrow V_{N,n_0}}.$$

\vskip 2mm
Let $i \in [M-1]$. We define $\psi_{N,M}$ on $s_i^M((V_{N,n_0})^{[M-1]}) \times \Delta^M \times \Delta^p$
by
$$\forall (\mathbf{B},t,x) \in s_i^M((V_{N,n_0})^{[M-1]}) \times \Delta^M \times \Delta^p,
\psi_{N,M}(\mathbf{B},t,x):=\psi_{N,M-1}(d_i^{M-1}(\mathbf{B}),(\sigma_i^M)^*(t),x).$$
By the induction hypothesis, this is a $U_N(F)$-map.
Those maps are compatible. Let $(i,j)\in [M-1]^2$ be such that $i<j$, and
let
$(\mathbf{B},t,x) \in \left(s_i^{M-1}((V_{N,n_0})^{[M-1]}) \cap s_j^{M-1}((V_{N,n_0})^{[M-1]})\right)
\times \Delta^M \times \Delta^p$. \\
Let $\mathbf{B'}\in (V_{N,n_0})^{[M-2]}$ be such that
$\mathbf{B}=s_i^{M-1}(s_{j-1}^{M-2}(\mathbf{B'}))=s_j^{M-1}(s_i^{M-2}(\mathbf{B'}))$.
Then, by the induction hypothesis,
\begin{align*}
\psi_{N,M-1}(d_i^M(\mathbf{B}),(\sigma_i^{M-1})^*(t),x) & = \psi_{N,M-1}(s_{j-1}^{M-2}(\mathbf{B}'),(\sigma_i^{M-1})^*(t),x) \\
 & = \psi_{N,M-1}(\mathbf{B}',(\sigma_{j-1}^{M-2} \circ \sigma_i^{M-1})^*(t),x) \\
 & = \psi_{N,M-1}(\mathbf{B}',(\sigma_{i}^{M-2} \circ \sigma_j^{M-1})^*(t),x) \\
 & =  \psi_{N,M-1}(s_i^{M-2}(\mathbf{B}'),(\sigma_j^{M-1})^*(t),x) \\
 & = \psi_{N,M-1}(d_j^M(\mathbf{B}),(\sigma_j^{M-1})^*(t),x).
\end{align*}
We deduce that the preceding maps yield an equivariant map
$$\boxed{\psi_{N,M} : V_N^{\text{deg}}(f) \times \Delta^M \times \Delta^p \longrightarrow V_{N,n_0},}$$
where $V_N^{\text{deg}}(f):=\underset{i \in [M-1]}{\prod}s_i^M((V_{N,n_0})^{[M-1]})$ (the continuity of this map will be checked later on).

\vskip 2mm
We now check that this map is compatible with the preceding one. \\
Let $(i,j)\in [M-1]\times [M]$, and
$(\mathbf{B},t,x) \in s_i^M((V_{N,n_0})^{[M-1]}) \times (\delta_j^{M-1})^*(\Delta^{M-1}) \times \Delta^p$,
$t'\in \Delta^{M-1}$ such that $t=(\delta_j^{M-1})^*(t')$, and
$\mathbf{B}'\in (V_{N,n_0})^{[M-1]}$ such that $\mathbf{B}=s_i^M(\mathbf{B'})$.
Then $d_i^{M-1}(\mathbf{B})=\mathbf{B}'$, and we deduce from the induction hypothesis that
\begin{align*}
\psi_{N,M-1}(d_i^{M-1}(\mathbf{B}),(\sigma_i^M)^*(t),x) & =
\psi_{N,M-1}(\mathbf{B}',(\sigma_i^M\circ \delta_j^{M-1})^*(t'),x) \\
& = \psi_{N,M-1}((d_j^{M-1} \circ s_i^M)(\mathbf{B}'),t',x) \\
& = \psi_{N,M-1}(d_j^{M-1}(\mathbf{B}),t',x).
\end{align*}
This proves that the two previous definitions are compatible.

\vskip 2mm
Now, let $p \in \{1,\dots,N-1\}$ and set $q:=N-p$.
Should we identify $V_p(f) \times V_q(f+p)$ with a subspace of $V_N(f)$, we may use
the induction hypothesis concerning condition (ii) to
define $\psi_{N,M}$ on $V_p(f) \times V_q(f+p) \times (\Delta^M \times \Delta^p)$ by:
$$\forall (\mathbf{B},\mathbf{B}',t)\in V_p(f) \times V_q(f+p) \times (\Delta^M \times \Delta^p), \quad
\psi_{N,M}(\mathbf{B},\mathbf{B}',t):=\bigl(\psi_{p,M}(\mathbf{B},t),\psi_{q,M}(\mathbf{B}',t)\bigr).$$
Since $\psi_p$ (resp.\ $\psi_q$) is a morphism of simplicial $U_p(F)$-spaces (resp.\ of simplicial
$U_q(F)$-spaces), we deduce that we have just defined a $(U_p(F) \times U_q(F))$-map. By the induction hypothesis, it is compatible with the
preceding ones.
Set
$$V_N^{(p,q)}(f):=(V_p(f) \times V_q(f+p)).U_N(F) \subset V_N(f).$$
The canonical $U_N(F)$-map
$$(V_p(f) \times V_q(f+p)) \times U_N(F) \longrightarrow V_N^{(p,q)}(f)$$
is an identification map (indeed, for every $l\in \mathbb{N}^*$,
$(V_p^{(l)}(f) \times V_q^{(l)}(f+p)) \times U_N(F) \longrightarrow
 V_N^{(p,q)}(f)\cap V_N^{(l)}(f)$ is a continuous surjection between compact spaces). We deduce that the canonical map
$$(V_p(f) \times V_q(f+p)) \times_{U_p(F) \times U_q(F)} U_N(F) \longrightarrow V_N^{(p,q)}(f)$$
is an equivariant homeomorphism.
It follows that the following definition extends $\psi_{N,M}$ as a $U_N(F)$-map on
$V_N^{(p,q)}(f) \times \Delta^M \times \Delta^p$:
$$\forall (\mathbf{B},\mathbf{B}',t,\mathbf{M}) \in V_p(f) \times V_q(f+p)
\times (\Delta^M \times \Delta^p) \times U_N(F), \quad
\psi_{N,M}((\mathbf{B},\mathbf{B}').\mathbf{M},t)=\psi_{N,M}((\mathbf{B},\mathbf{B}'),t).\mathbf{M}.$$

\vskip 2mm
Let $p'\in \{1,\dots,N-1\}$ such that $p<p'$, and set $q':=N-q$.
We have to check that the respective definitions of $\psi_{N,M}$ on $V_N^{(p,q)}(f) \times \Delta^M \times \Delta^p$
and $V_N^{(p',q')}(f) \times \Delta^M \times \Delta^p$ are compatible. \\
Let $(\mathbf{B},t,x)\in (V_N^{(p,q)}(f) \cap V_N^{(p',q')}(f))\times \Delta^M \times \Delta^p$.
Proposition \ref{productintersection} then yields
a triple $(\mathbf{B}_1,\mathbf{B}_2,\mathbf{B}_3)\in V_p(f) \times V_{p'-p}(f+p) \times V_{q'}(f+p')$ and
$\mathbf{M} \in U_N(F)$ such that $\mathbf{B}=(\mathbf{B}_1,\mathbf{B}_2,\mathbf{B}_3).\mathbf{M}$.

We deduce from the induction hypothesis that
\begin{align*}
(\psi_p(\mathbf{B}_1,t,x),\psi_q((\mathbf{B}_2,\mathbf{B}_3),t,x)) & = (\psi_p(\mathbf{B}_1,t,x),\psi_{p'-p}(\mathbf{B}_2,t,x),
\psi_{q'}(\mathbf{B}_3,t,x)) \\
(\psi_p(\mathbf{B}_1,t,x),\psi_q((\mathbf{B}_2,\mathbf{B}_3),t,x)) & =
(\psi_{p'}((\mathbf{B}_1,\mathbf{B}_2),t,x),\psi_{q'}(\mathbf{B}_3,t,x))
\end{align*}
This proves that the previous maps
are all compatible, and that they therefore yield a $U_N(F)$-map: $$\boxed{\psi_{N,M} : V_N^{\text{prod}}(f) \times \Delta^M \times \Delta^p \longrightarrow V_{N,n_0},}$$
where $V_N^{\text{prod}}(f):= \underset{p \in \{1,\dots,N-1\}}{\bigcup} V_N^{(p,N-p)}(f)$. This map
is compatible with the preceding ones.

\vskip 2mm
Finally, we may define $\psi_{N,M}$ on $V_N(f)\times \Delta^M \times \partial\Delta^p$ as $\psi_{N,M}^\partial$.
Since $\psi^\partial$ satisfies conditions (i) to (iv) and is a family of morphisms of simplicial spaces, we
deduce from the induction hypothesis that this definition is compatible with the preceding ones. \\
We have just constructed an equivariant map
$$\boxed{\psi_{N,M} : \left(V_N(f) \times \partial(\Delta^M \times \Delta^p)\right) \cup
\left((V_N^{\text{prod}}(f) \cup V_N^{\text{deg}}(f)) \times \Delta^M \times \Delta^p\right) \longrightarrow
V_{N,n_0}}.$$
For every $l \in \mathbb{N}^*$, its restriction to the intersection of
$V_N^{(l)}(f) \times \Delta^M \times \Delta^p$ and its domain is obtained
by gluing together a finite family of continuous maps, where each map is defined on a compact subspace of
$V_N^{(l)}(f) \times \Delta^M \times \Delta^p$, and is therefore continuous. Since
$V_N(f)=\underset{l\in \mathbb{N}^*}{\Indlim}\, V_N^{(l)}$, we deduce that $\psi_{N,M}$ is continuous.

\vskip 2mm
By construction, it suffices to extend $\psi_{N,M}$ to a $U_N(F)$-map from $V_N(f) \times \Delta^M \times \Delta^p$
to $V_{N,n_0}$, and such an extension will fulfill all the expected conditions: indeed, the compatibility with face maps follows from
the definition of $\psi_{N,M}$ on $V_N(f) \times \partial \Delta^M \times \Delta^p$; the compatibility with degeneracy maps follows from
the definition of $\psi_{N,M}$ on $V_N^{\text{deg}}(f) \times \Delta^M \times \Delta^p$; it follows from the definition of $\psi_{N,M}$ on $V_N^{\text{prod}}(f) \times \Delta^M \times \Delta^p$ that
condition (iii) holds; finally, the compatibility with $\psi^\partial$ comes from the definition of
$\psi_{N,M}$ on $V_N(f) \times \Delta^M \times \partial\Delta^p$.

The extension is based upon the following result:

\begin{prop}
The pair
$$\left(V_N(f)\times \Delta^M \times \Delta^p,\left(V_N(f) \times \partial(\Delta^M \times \Delta^p)\right) \cup
\left((V_N^{\text{prod}}(f) \cup V_N^{\text{deg}}(f)) \times (\Delta^M \times \Delta^p)\right)
\right)$$ is filtered by a sequence of relative $U_N(F)$-CW-complexes.
\end{prop}

\begin{proof}
To start with, $(\Delta^M \times \Delta^p,\partial(\Delta^M \times \Delta^p))$ is a relative CW-complex.
Also, for every $l \in \mathbb{N}^*$, we set
$V_N^{\text{prod} (l)}(f):=V_N^{\text{prod}}(f)\cap V_N^{(l)}(f)$,
$V_N^{\text{deg} (l)}(f):=V_N^{\text{deg}}(f)\cap V_N^{(l)}(f)$ and
$$\mathcal{M}_l:=\left(\left(\underset{i=0}{\overset{N-1}{\oplus}}\mathcal{H}_{n_0+i}^{(l+1)},
\underset{i=0}{\overset{N-1}{\oplus}}\mathcal{H}_{n_0+i}^{(l)},(\mathcal{H}_{n_0+i-1}^{(l)})_{1 \leq i \leq N}\right)_{0 \leq k \leq M}
,(\id)_{0 \leq k \leq M-1}\right).$$
We then deduce from corollary \ref{CWconjecture} - applied to
$\mathcal{M}_l$ - that the pair \\
$(V_N^{(l+1)}(f), V_N^{(l)}(f) \cup V_N^{\text{prod}(l+1)}(f) \cup
V_N^{\text{deg}(l+1)}(f))$ is a relative $U_N(F)$-CW-complex for every positive integer $l$.
Since $V_N^{(1)}(f)=V_N^{\text{prod}(1)}(f)$, we deduce that the pair
$(V_N(f), V_N^{\text{prod}}(f) \cup V_N^{\text{deg}}(f))$
is filtered by a sequence of $U_N(F)$-CW-complexes, which proves the claimed result.
\end{proof}

Finally, $V_{N,n_0}$ is contractible, and $U_N(F)$ acts freely on
$V_N(f) \times \Delta^M \times \Delta^p$; we may thus extend $\psi_{N,M}$ to obtain a $U_N(F)$-map
$\psi_{N,M} : V_N(f) \times \Delta^M \times \Delta^p \longrightarrow V_{N,n_0}$.
We choose such an extension, and this finishes the construction in the case $f$ is constant.

\subsubsection{The case $f$ is non-degenerate}

We finally assume that $f:[M] \rightarrow \mathbb{N}$ is non-degenerate.

\vskip 2mm
Let $i\in [M]$. We define $\psi_{N,M}$ on $V_N(f) \times (\delta_i^{M-1})^*(\Delta^{M-1}) \times \Delta^p$ by:
$$\forall (\mathbf{B},t,x) \in V_N(f) \times \Delta^{M-1} \times \Delta^p, \quad
\psi_{N,M}(\mathbf{B},(\delta_i^{M-1})^*(t),x):=\psi_{N,M-1}(d_i^{M-1}(\mathbf{B}),t,x).$$
By compatibility with the face maps at lower levels and since condition (iv) is satisfied at lower levels, those definitions are compatible
and, since condition (iv) is satisfied at the lower levels, yield a $U_N(F)$-map:
$$\boxed{\psi_{N,M} : V_N(f) \times \partial \Delta^M \times \Delta^p \longrightarrow V_N(\mathcal{H}).}$$
Let $u=\delta_1<\dots<\delta_{k-1}<[M]$ be a non-trivial sequence in $S(M)$, where
$\delta_{k-1} : [m] \hookrightarrow [M]$. Let $\mathbf{B} \in V_N(f)$, $x \in \Delta_u$, $z \in \Delta^p$, and
$(y,t)\in \Delta^m \times [0,1]$ such that $x=r_M(\delta_{k-1}^*(y),t)$. Then
$y \in \Delta_{u'}$ and we deduce from the induction hypothesis
that $\psi_{N,m}((\delta_{k-1})_*(\mathbf{B}),y,z) \in V_N(\mathcal{H}_{N,u',(\delta_{k-1})_*(f)})$.
We deduce from the previous definition and the compatibility with face maps at lower levels that
$\mathbf{B}':=
\psi_{N,M}(\mathbf{B},(\delta_{k-1})^*(y),z)=\psi_{N,m}((\delta_{k-1})_*(\mathbf{B}),y,z) \in V_N(\mathcal{H}_{N,u',(\delta_{k-1})_*(f)})$.
It follows that we may define
$$\psi_{N,M}(\mathbf{B},x,z):=\cos\left(\frac{\pi t}{2}\right)\mathbf{B'}+\sin\left(\frac{\pi t}{2}\right).\varphi_{N,u,f}(\mathbf{B'})
\in V_N(\mathcal{H}_{N,u,f}),$$
and this yields a $U_N(F)$-map $V_N(f) \times \Delta_u \times \Delta^p \longrightarrow V_N(\mathcal{H}_{N,u,f})$ which
is compatible with the previous one.

\vskip 2mm
Let $v=\delta'_1<\dots<\delta_{k'-1}<[M]\in S(M)$ be another non-trivial class.
We need to check that the respective definitions of $\psi_{N,M}$
on $V_N(f) \times \Delta_u \times \Delta^p$ and $V_N(f) \times \Delta_v \times \Delta^p$ agree.
By Proposition \ref{intersection}, it suffices to do this when $u \subset v$.
Let $\mathbf{B} \in V_N(f)$, $x \in \Delta_u$, and $(y,t)\in \partial \Delta^M \times \Delta^p$ such that $x=r_M(y,t)$, and
$\mathbf{B}' :=\psi_{N,M}(\mathbf{B},y,z)$. Then $y \in \Delta_u$, and we deduce from the preceding remarks
that $\mathbf{B}' \in V_N(\mathcal{H}_{N,u',(\delta_{k'-1})_*(f)})$. From the remarks at the end of Section
\ref{10.4.3}, we derive that $\varphi_{N,u,f}(\mathbf{B'})=\varphi_{N,v,f}(\mathbf{B'})$, and
it follows that
$$\cos\left(\frac{\pi t}{2}\right)\mathbf{B'}+\sin\left(\frac{\pi t}{2}\right).\varphi_{N,u,f}(\mathbf{B'})
=\cos\left(\frac{\pi t}{2}\right)\mathbf{B'}+\sin\left(\frac{\pi t}{2}\right).\varphi_{N,v,f}(\mathbf{B'}).$$
We deduce that all the definitions are compatible so far, and it follows from Proposition \ref{glue} that they yield a $U_N(F)$-map: $$\boxed{\psi_{N,M} : V_N(f) \times \left(\underset{0 \leq i \leq M}{\cup}\Delta_i^M\right)\times \Delta^p \longrightarrow V_N(\mathcal{H})
.}
$$
By construction, this map is compatible with face maps (and therefore condition (v) holds),
and condition (vi) also holds. Moreover, condition (iv) holds for every non-trivial sequence $u \in S(M)$.
The compatibility with degeneracy maps is not relevant here.

\vskip 2mm
Let now $p \in \{1,\dots,N-1\}$ and $q:=N-p$.
Identifying $V_p(f) \times V_q(f+p)$ with a subspace of $V_N(f)$, we may use
the induction hypothesis concerning condition (ii) to
define $\psi_{N,M}$ on $(V_p(f) \times V_q(f+p)) \times (\Delta^M \times \Delta^p)$ by:
$$\forall (\mathbf{B},\mathbf{B}',t)\in V_p(f) \times V_q(f+p) \times (\Delta^M \times \Delta^p), \;
\psi_{N,M}(\mathbf{B},\mathbf{B}',t)=(\psi_{p,M}(\mathbf{B},t),\psi_{q,M}(\mathbf{B}',t)).
$$
Since $\psi_p$ (resp.\ $\psi_q$) is a morphism of simplicial $U_p(F)$-spaces (resp.\ of simplicial
$U_q(F)$-spaces), we deduce that we have just defined a $(U_p(F) \times U_q(F))$-map. \\
Set
$$V_N^{(p,q)}(f):=(V_p(f) \times V_q(f+p)).U_N(F) \subset V_N(f).$$
The canonical $U_N(F)$-map
$$(V_p(f) \times V_q(f+p)) \times U_N(F) \longrightarrow V_N^{(p,q)}(f)$$
is an identification map (since for every positive integer $l$,
$(V_p^{(l)}(f) \times V_q^{(l)}(f+p)) \times U_N(F) \longrightarrow
 V_N^{(p,q)}(f)\,\cap\, V_N^{(l)}(f)$ is a continuous surjection between compact spaces). We deduce that the canonical map
$$(V_p(f) \times V_q(f+p)) \times_{U_p(F) \times U_q(F)} U_N(F) \longrightarrow V_N^{(p,q)}(f)$$
is an equivariant homeomorphism.
It follows that we can extend $\psi_{N,M}$ as a $U_N(F)$-map on
$V_N^{(p,q)}(f) \times \Delta^M \times \Delta^p$ by: $$\forall (\mathbf{B},\mathbf{B}',t,\mathbf{M}) \in V_p(f) \times V_q(f+p) \times (\Delta^M \times \Delta^p) \times U_N(F),
\; \psi_{N,M}((\mathbf{B},\mathbf{B}').\mathbf{M},t)=\psi_{N,M}((\mathbf{B},\mathbf{B}'),t).\mathbf{M}.$$
With the same line of reasoning as in the case $f$ is constant, those definitions are compatible
and yield a $U_N(F)$-map: $$\boxed{\psi_{N,M} : V_N^{\text{prod}}(f) \times \Delta^M \times \Delta^p \longrightarrow V_N(\mathcal{H}),}$$
where
$$V_N^{\text{prod}}(f):= \underset{1 \leq p \leq N-1}{\bigcup} V_N^{(p,N-p)}(f).$$
We now check that this map is compatible with the previous one. The compatibility with the definition
on $V_n(F) \times \partial \Delta^M \times \Delta^p$ is a straightforward consequence of the induction hypothesis.
Let $u \in S(M)$ be a non trivial class,
$p \in \{1,\dots,N-1\}$ and $q:=N-p$. Let $\mathbf{B} \in V_N(f)$, $x \in \Delta_u$, $z \in \Delta^p$, and
$(y,t)\in \partial\Delta^M \times [0,1]$ such that $x=r_M(y,t)$.

Let $(\mathbf{B}',\mathbf{B}'') \in V_p(f)\times V_q(f+p)$ and $\mathbf{M} \in U_N(F)$ such that
$\mathbf{B}=(\mathbf{B}',\mathbf{B}'').\mathbf{M}$. By the induction hypothesis,
$$
\psi_{p,M}(\mathbf{B}',x,z)=\cos \left(\frac{\pi t}{2}\right)\psi_{p,M}(\mathbf{B}',y,z)
+\sin \left(\frac{\pi t}{2}\right)\varphi_{p,u,f}(\psi_{p,M}(\mathbf{B}',y,z)),
$$
$$
\psi_{q,M}(\mathbf{B}'',x,z)=\cos \left(\frac{\pi t}{2}\right)\psi_{q,M}(\mathbf{B}'',y,z)+
\sin \left(\frac{\pi t}{2}\right)\varphi_{q,u,f+p}(\psi_{q,M}(\mathbf{B}'',y,z)),
$$
and we deduce from the last remark in Section \ref{10.4.3} that
$$
\varphi_{N,u,f}(\psi_{p,M}(\mathbf{B}',y,z),\psi_{q,M}(\mathbf{B}'',y,z))=
\left(\varphi_{p,u,f}(\psi_{p,M}(\mathbf{B}',y,z)),\varphi_{q,u,f+p}(\psi_{q,M}(\mathbf{B}'',y,z))\right).
$$
We then deduce from the compatibility of the definitions on $V_N(f) \times \partial \Delta^M \times \Delta^p$ that
$$\bigl(\psi_{p,M}(\mathbf{B}',x,z),\psi_{q,M}(\mathbf{B}'',x,z)\bigr).\mathbf{M}=
\cos \left(\frac{\pi}{2}t\right)\psi_{N,M}(\mathbf{B},y,z)+
\sin \left(\frac{\pi}{2}t\right)\varphi_{N,u,f}(\psi_{N,M}(\mathbf{B},y,z)),$$
where $\psi_{N,M}(\mathbf{B},y,z)$ is given by the earlier definition.

\vskip 2mm
Finally, we may define $\psi_{N,M}$ on $V_N(f)\times \Delta^M \times \partial\Delta^p$ as $\psi_{N,M}^\partial$.
Since $\psi^\partial$ satisfies condition (i) to (vi) and is a family of morphisms of simplicial spaces, we
deduce from the induction hypothesis that this definition is compatible with the above ones.

\vskip 2mm
Proposition \ref{glue} then shows that we have just constructed an equivariant map
$$\boxed{\psi_{N,M} : \left(V_N(f) \times \left((\Delta^M \times \partial \Delta^p)\cup\left((\underset{0\leq i \leq M}{\bigcup}\Delta_i^M) \times \Delta^p
\right)\right)\right) \bigcup\left(V_N^{\text{prod}}(f) \times \Delta^M \times \Delta^p\right)
\longrightarrow V_N(\mathcal{H}).}$$
For every positive integer $l$ and every non-trivial $u \in S(M)$, the restriction of $\psi_{N,M}$ to the intersection of
$V_N^{(l)}(f) \times \Delta_u \times \Delta^p$ and its domain is obtained
by gluing together a finite family of continuous maps, where each map is defined on a compact subspace of
$V_N^{(l)}(f) \times \Delta_u \times \Delta^p$: therefore, this restriction is continuous. Since
$V_N(f)=\underset{l\in \mathbb{N}^*}{\Indlim}\, V_N^{(l)}(f)$, we deduce from Proposition \ref{glue} that $\psi_{N,M}$ is continuous.

\vskip 2mm
Notice that $\psi_{N,M}$ maps $\Bigl(V_N(f) \times \partial (C(\Delta^M) \times \Delta^p)\Bigr)
\bigcup \left(V_N^{\text{prod}}(f) \times C(\Delta^M) \times \Delta^p\right)$ into $V_N(\mathcal{H}_{N,[M],f})$.
If we can extend this map to a $U_N(F)$-map $V_N(f)\times C(\Delta^M) \times \Delta^p \longrightarrow V_N(\mathcal{H}_{N,[M],f})$, condition
(iv) will be checked, and we will recover a $U_N(F)$-map: $V_N(f) \times \Delta^M \times \Delta^p \longrightarrow V_N(\mathcal{H})$
which fulfills all the requirements. That $\psi_{N,M}$ may be extended in this manner is a consequence of the following result:

\begin{prop}
The pair $$\left(V_N(f)\times C(\Delta^M) \times \Delta^p,\left(V_N(f) \times \partial(C(\Delta^M) \times \Delta^p)\right) \cup
\left(V_N^{\text{prod}}(f)\times C(\Delta^M) \times \Delta^p\right)
\right)$$ is filtered by a sequence of relative $U_N(F)$-CW-complexes.
\end{prop}

\begin{proof}
First of all, $(C(\Delta^M) \times \Delta^p,\partial(C(\Delta^M) \times \Delta^p))$ is a relative CW-complex.
For every positive integer $l$, set
$V_N^{\text{prod} (l)}(f):=V_N^{\text{prod}}(f)\cap V_N^{(l)}(f)$, and
$$\mathcal{N}_l:=\left(\left(\underset{i=0}{\overset{N-1}{\oplus}}\mathcal{H}_{f(k)+i}^{(l+1)},
\underset{i=0}{\overset{N-1}{\oplus}}\mathcal{H}_{f(k)+i}^{(l)},(\mathcal{H}_{f(k)+i-1}^{(l)})_{1 \leq i \leq N})\right)
_{0 \leq k \leq M},(f_k)_{0 \leq k \leq M-1}\right),$$
where, for every $k \in [M-1]$, we have chosen a linear isometry $f_k$ from $\underset{i=0}{\overset{N-1}{\oplus}}\mathcal{H}_{f(k)+i}^{(l+1)}$  to
$\underset{i=0}{\overset{N-1}{\oplus}}\mathcal{H}_{f(k+1)+i}^{(l+1)}$ which maps
$\underset{i=0}{\overset{N-1}{\oplus}}\mathcal{H}_{f(k)+i}^{(l)}$  to
$\underset{i=0}{\overset{N-1}{\oplus}}\mathcal{H}_{f(k+1)+i.}^{(l)}$
We then deduce from conjecture \ref{CWconjecture}, applied to
$\mathcal{N}_l$, that the pair
$(V_N^{(l+1)}(f), V_N^{(l)}(f) \cup V_N^{\text{prod}(l+1)}(f))$ is a relative $U_N(F)$-CW-complex for every positive integer $l$.
Since $V_N^{(1)}(f)=V_N^{\text{prod}(1)}(f)$, we deduce that the pair $(V_N(f), V_N^{\text{prod}}(f))$
is filtered by a sequence of $U_N(F)$-CW-complexes: this proves the claimed result.
\end{proof}

Finally, $V_N(\mathcal{H}_{N,[M],f})$ is contractible, and $U_N(F)$ acts freely on
$V_N(f) \times C(\Delta^M) \times \Delta^p$. It follows that we may extend $\psi_{N,M}$ to recover a $U_N(F)$-map
$$\psi_{N,M} : V_N(f) \times C(\Delta^M) \times \Delta^p \longrightarrow V_N(\mathcal{H}_{N,[M],f}).$$
This finishes the construction in the case $f$ is non-degenerate (up to a choice of extension, of course),
and we have thus completed our proof of Proposition \ref{mighty}.

\subsection{The map $s\vEc_{G,L^2}^{F,\infty *}[1] \longrightarrow \sub(L^2(G,\mathcal{H}))$}\label{10.4.6}

For every $n \in \mathbb{N}$, set
$$\varphi_n : \underset{k \in \mathbb{N}}{\coprod}E_n\bigl((F^{(\infty)})^{\{k,\dots,k+n-1\}}\bigr) \longrightarrow
\underset{k \in \mathbb{N}}{\coprod}G_n\bigl((F^{(\infty)})^{\{k,\dots,k+n-1\}}\bigr),$$
so that $\text{Fib}^{F^{(\infty)} *}(\mathbf{1})=\underset{n\in \mathbb{N}}{\coprod} \varphi_n$. \\
For every object $x$ in $\varphi_n \sframe$, we let $k_x$ denote the only integer such that
$x \in V_n((F^{(\infty)})^{\{k_x,\dots,k_x+n-1\}})$.

\vskip 2mm
Let $n \in \mathbb{N}^*$.
For every object $\mathbf{F}$ of $\Func_{\uparrow L^2}(\mathcal{E}G,\varphi_n \sframe)$, we define
$$\alpha(\mathbf{F}) : \begin{cases}
G & \longrightarrow \mathbb{R}_*^+ \\
g & \longmapsto \|\mathbf{F}(g)\|
\end{cases} \quad ; \quad
\beta(\mathbf{F}) : \begin{cases}
G & \longrightarrow V_n(\mathcal{H}) \\
g & \longmapsto \frac{1}{\|\mathbf{F}(g)\|}\mathbf{F}(g)
\end{cases}
\quad \text{and} \quad
k(\mathbf{F}) : \begin{cases}
G & \longrightarrow \mathbb{N} \\
g & \longmapsto k_{\beta(\mathbf{F})[g]}.
\end{cases}
$$
Clearly, those maps are continuous and yield continuous maps
$$\alpha : \Ob(\Func_{\uparrow L^2}(\mathcal{E}G,\varphi_n \sframe)) \rightarrow C_*(G) \cap L^2_*(G),$$
$$\beta : \Ob(\Func_{\uparrow L^2}(\mathcal{E}G,\varphi_n \sframe)) \rightarrow \mathcal{H}^n \quad \text{and} \quad
k :  \Ob(\Func_{\uparrow L^2}(\mathcal{E}G,\varphi_n \sframe)) \rightarrow \mathbb{N}^G.$$
Let $m \in \mathbb{N}$, and $\mathbf{F}:F_0 \rightarrow \dots \rightarrow F_m$ be an element of
$\mathcal{N}(\Func_{\uparrow L^2}(\mathcal{E}G,\varphi_n \sframe))_m$.
The definition of $\Func_{\uparrow L^2}(\mathcal{E}G,\varphi_n \sframe)$ shows that the list $k(\mathbf{F})[g]:=(k(F_i)[g])_{0 \leq i \leq m}$
is a non-decreasing map from $[m]$ to $\mathbb{N}$ for every $g \in G$.
We then define
$$\beta(\mathbf{F}) :
\begin{cases}
G & \longrightarrow V_n(m) \\
g & \longmapsto (\beta(F_i)[g])_{0\leq i \leq m},
\end{cases}
$$
where, for every $g \in G$, $(\beta(F_i)[g])_{0\leq i \leq m}$ is considered as an element of
$V_n(k(\mathbf{F})[g])$. Notice that $\beta(\mathbf{F})$ is continuous. This yields
a continuous map
$$\beta : \mathcal{N}(\Func_{\uparrow L^2}(\mathcal{E}G,\varphi_n \sframe))_m \longrightarrow V_n(m).$$
Let now $K$ be a compact space, and assume that we have a family
$\psi=(\psi_n)_{n\in \mathbb{N}}$ such that
$\psi_n : (V_n(m))_{m \in \mathbb{N}} \longrightarrow (\Hom(\Delta^n \times K,V_n(\mathcal{H})))_{m \in \mathbb{N}}$
is a morphism of simplicial $U_n(F)$-spaces for every $n\in \mathbb{N}$, which satisfies conditions (i), (ii) and (iii).

Let $n$ be a positive integer. Let $\mathbf{F}:F_0 \rightarrow \dots \rightarrow F_m$ be an element
of $\mathcal{N}(\Func_{\uparrow L^2}(\mathcal{E}G,\varphi_n \sframe))_m$, $t=(t_0,\dots,t_m) \in \Delta^m$ and $z \in K$.
Set
$$\chi_{n,m}(\mathbf{F},t,z):
\begin{cases}
G & \longrightarrow \mathcal{H}^n \\
g & \longmapsto \left(\underset{0 \leq i \leq m}{\sum}t_i.\alpha(F_i)[g]\right).
\psi_{n,m}(\beta(\mathbf{F})[g],t,z).
\end{cases}$$

For every $g \in G$, $\chi_{n,m}(\mathbf{F},t,z)[g]$ is the product of a positive real number with an orthonormal $n$-tuple.
Also $\underset{0 \leq i \leq m}{\sum}t_i.\alpha(F_i) \in L^2(G)$ since $L^2(G)$ is convex. It follows
that the factors $\pi_i\circ \chi_{n,m}(\mathbf{F},t,z) :G \rightarrow \mathcal{H}$ of $\chi_{n,m}(\mathbf{F},t,z)$
(where $\pi_i : \mathcal{H}^n \rightarrow \mathcal{H}$ denotes the projection on the $i$-th factor)
form a (linearly independent) orthogonal $n$-tuple of elements of $L^2(G,\mathcal{H})$, and we conclude that
$\chi_{n,m}(\mathbf{F},t,z)$ defines an element of $B_n(L^2(G,\mathcal{H}))$.

Using the fact that $\psi_{n,m}$ is continuous and the respective definitions of the various left $G$-actions considered here,
we find that we have just defined a $G$-map:
$$\chi_{n,m} : \mathcal{N}(\Func_{\uparrow L^2}(\mathcal{E}G,\varphi_n \sframe))_m \times \Delta^m \times K \longrightarrow
B_n(L^2(G,\mathcal{H})).$$
Note that $\chi_{n,m}$ is also a $\Sim_n(F)$-map. Indeed, let
$\mathbf{F}:F_0 \rightarrow \dots \rightarrow F_m$ be an element of $\mathcal{N}(\Func_{\uparrow L^2}(\mathcal{E}G,\varphi_n \sframe))_m$,
$t=(t_0,\dots,t_m) \in \Delta^m$ and $z \in K$. Let also $(\lambda,\mathbf{M}) \in \mathbb{R}_+^* \times U_n(F)$. It then follows from the
$U_n(F)$-equivariance of $\psi_{n,m}$ that
\begin{align*}
\chi_{n,m}(\mathbf{F}.(\lambda\mathbf{M}),t,z)[g]
& =\left(\underset{0 \leq i \leq m}{\sum}t_i.\alpha(F_i.(\lambda\mathbf{M}))[g]\right).\psi_{n,m}(\beta(\mathbf{F}.(\lambda\mathbf{M})),t,z) \\
& = \left(\underset{0 \leq i \leq m}{\sum}t_i\lambda \,\alpha(F_i)[g]\right).
\psi_{n,m}(\beta(\mathbf{F}).\mathbf{M},t,z) \\
 & = \chi_{n,m}(\mathbf{F},t,z)[g].(\lambda\mathbf{M}).
\end{align*}
We conclude that $\chi_{n,m}$ is a $(G \times \Sim_n(F))$-map.

\vskip 2mm
\noindent
We now check that the maps $\chi_{n,0},\dots,\chi_{n,m},\dots$ are compatible with the simplicial structures.
Let $\tau : [m'] \rightarrow [m]$ be a morphism in $\Delta$. Let
$\mathbf{F}:F_0 \rightarrow \dots \rightarrow F_m$ be an element of $\mathcal{N}(\Func_{\uparrow L^2}(\mathcal{E}G,\varphi_n \sframe))_m$.
Let $t=(t_0,\dots,t_m) \in \Delta^m$ and $z \in K$. Clearly
$\beta(\tau_*(\mathbf{F}))=\tau_*(\beta(\mathbf{F}))$, so we deduce from the compatibility of
$\psi_n$ with the simplicial structure that, for any $g \in G$,
\begin{align*}
\chi_{n,m}(\mathbf{F},\tau^*(t),z)[g]
& = \left(\underset{0 \leq i \leq m}{\sum}\left(\underset{j \in \tau^{-1}(\{i\})}{\sum}t_j\right).\alpha(F_i)[g]\right).
\psi_{n,m}(\beta(\mathbf{F})[g],\tau^*(t),z) \\
& = \left(\underset{0 \leq j \leq m'}{\sum}t_j.\alpha(F_{\tau(j)})[g]\right).\psi_{n,m'}(\tau_*(\beta(\mathbf{F})[g]),t,z) \\
 & = \chi_{n,m}(\tau_*(\mathbf{F}),t,z)[g].
\end{align*}
We deduce that the family $(\chi_{n,m})_{m\in \mathbb{N}}$ yields a $(G\times \Sim_n(F))$-map:
$$\chi_n : s\widetilde{\vEc}_{G,L^2}^{n,F,\infty *} \times K \longrightarrow B_n(L^2(G,\mathcal{H})).$$
We define $\chi_0$ as the trivial map.

\vskip 2mm
Since the canonical map $\underset{n \in \mathbb{N}}{\coprod} s\widetilde{\vEc}_{G,L^2}^{n,F,\infty *} \longrightarrow
s\vEc_{G,L^2}^{F,\infty *}[1]$ is an identification map and $K$ is compact,
we may finally define $\chi$ as the unique $G$-map which renders commutative the diagram
$$\begin{CD}
\left(\underset{n \in \mathbb{N}}{\coprod} s\widetilde{\vEc}_{G,L^2}^{n,F,\infty *}\right) \times K
@>{\underset{n\in \mathbb{N}}{\coprod}\chi_n}>>
\underset{n \in \mathbb{N}}{\coprod} B_n(L^2(G,\mathcal{H})) \\
@VVV @VVV \\
 s\vEc_{G,L^2}^{F,\infty *}[1] \times K @>{\chi}>> \underset{n \in \mathbb{N}}{\bigcup}\sub_n(L^2(G,\mathcal{H})).
\end{CD}$$

\subsection{The morphism $s\vEc_{G,L^2}^{F,\infty *} \longrightarrow \Fred(G,\mathcal{H}^\infty)$}\label{10.4.7}

We define
$$\alpha_G^\psi(0) : s\vEc_{G,L^2}^{F,\infty *}[0] \times K \longrightarrow \Fred(G,\mathcal{H}^\infty)[0]$$
as the trivial map. \\
For every positive integer $n$, we define $\alpha_G^\psi(n)$ as the composite $G$-map
$$s\vEc_{G,L^2}^{F,\infty *}[n] \times K
\longrightarrow \left(s\vEc_{G,L^2}^{F,\infty *}[1]\times K\right)^n
\overset{\chi^n}{\longrightarrow}
\left(\underset{k \in \mathbb{N}}{\bigcup}\sub_k(L^2(G,\mathcal{H}))\right)^n \overset{\Shift^n}{\longrightarrow}
\Fred(G,\mathcal{H}^\infty)[n],$$
where the first map is the product of the simplicial maps $(\alpha_i^n)_* : s\vEc_{G,L^2}^{F,\infty *}[n] \rightarrow
s\vEc_{G,L^2}^{F,\infty *}[1]$ for $i \in \{1,\dots,n\}$ (where
$\alpha_i^n : [1] \rightarrow [n]$ maps $0$ to $i-1$ and $1$ to $i$).
The compact space $K$ is considered here as the simplicial space with $K$ as the space at every level
and all morphisms equal to $\id_K$.

\begin{prop}
$\alpha_G^\psi : s\vEc_{G,L^2}^{F,\infty *}\times K \longrightarrow
\Fred(G,\mathcal{H}^\infty)$ is a morphism of hemi-simplicial $G$-spaces.
\end{prop}

\noindent \textbf{Remark :} In fact, $\alpha_G^\psi$ is even a morphism of simplicial $G$-spaces, but this is not relevant here since
we are only interested in thick geometric realizations.

\begin{proof}
Obviously, $\alpha_G^\psi(n)$ is a $G$-map for every $n\in \mathbb{N}$.
In order to prove that $\alpha_G^\psi$ is compatible with degeneracy maps, it suffices to
check that the square
$$\begin{CD}
s\vEc_{G,L^2}^{F,\infty *}[2] \times K @>{\alpha_G^\psi(2)}>> \Fred(L^2(G,\mathcal{H})^\infty)^2 \\
@VV{d_1^2 \times \id_K}V @VVV \\
s\vEc_{G,L^2}^{F,\infty *}[1]\times K  @>{\alpha_G^\psi(1)}>> \Fred(L^2(G,\mathcal{H})^\infty)
\end{CD}
$$ is commutative.

Let $x \in s\vEc_{G,L^2}^{F,\infty *}[2]$ and $a \in K$.
Let $(y,t) \in \mathcal{N}(\Func_{\uparrow L^2}(\mathcal{E}G,\text{Fib}^{F^{(\infty)}*}(\mathbf{2}) \smod))_m \times
\Delta^m$ such that $x=[y,t]$. We write $t=(t_0,\dots,t_m)\in \Delta^m$.
 We set $y_i:= (\alpha_i^2)_*(y)$ for every
$i \in \{1,2\}$. Set $(n,n')\in \mathbb{N}^2$ such that
$y_1 \in \mathcal{N}(\Func_{\uparrow L^2}(\mathcal{E}G,\varphi_n \smod))_m$
and
$y_2 \in \mathcal{N}(\Func_{\uparrow L^2}(\mathcal{E}G,\varphi_{n'} \smod))_m$.

We choose
$z_1 \in \mathcal{N}(\Func_{\uparrow L^2}(\mathcal{E}G,\varphi_n \sframe))_m$
and $z_2 \in \mathcal{N}(\Func_{\uparrow L^2}(\mathcal{E}G,\varphi_{n'} \sframe))_m$
in the respective fibers of $y_1$ and $y_2$, and we write $z_1=F_0 \rightarrow \dots \rightarrow F_m$ and
$z_2=F'_0 \rightarrow \dots \rightarrow F'_m$.
Set $z:=(F_0,F'_0) \rightarrow \dots \rightarrow (F_m,F'_m)$, and define
$y' \in \mathcal{N}(\Func_{\uparrow L^2}(\mathcal{E}G,\varphi_{n+n'} \smod))_m$
as the image of $z$ by the canonical map. Then $[y',t]=d_1^2(x)$.

\vskip 2mm
The definition of $\Func_{\uparrow L^2}(\mathcal{E}G,\text{Fib}^{F^{(\infty)}*}(\mathbf{2}) \smod)$,
shows that $\|F_i(g)\|=\|F'_i(g)\|$ and $k(F'_i)[g]=k(F_i)[g]+n$ for any $i \in [m]$ and any $g \in G$.
Also, $\beta(z)=(\beta(z_1),\beta(z_2))$.
Hence $\alpha(F_i)=\alpha(F'_i)=\alpha(F_i,F'_i)$ for every $i\in [m]$, and $k(z_2)=k(z_1)+n=k(z)+n$.
It follows that $(\beta(z_1)[g],\beta(z_2)[g]) \in V_n(k(z_1)[g]) \times V_{n'}(k(z_1)[g]+n)$
for every $g \in  G$.

\vskip 2mm
Since condition (iii) holds for $\psi$, we deduce that, for any $g \in G$,
\begin{align*}
\chi_{n+n',m}(z,t,a)[g] & = \left(\underset{0 \leq i \leq n}{\sum} t_i.\alpha((F_i,F'_i))[g]\right).
\psi_{n+n',m}(\beta(z)[g],t,a) \\
& = \left(\underset{0 \leq i \leq n}{\sum} t_i.\alpha((F_i,F'_i))[g]\right).
\left(\psi_{n,m}(\beta(z_1)[g],t,a),\psi_{n',m}(\beta(z_2)[g],t,a)\right) \\
& = \left(\left(\underset{0 \leq i \leq n}{\sum} t_i.\alpha(F_i)[g]\right).
\psi_{n,m}(\beta(z_1)[g],t,a), \left(\underset{0 \leq i \leq n}{\sum} t_i.\alpha(F'_i)[g]\right).\psi_{n',m}(\beta(z_2)[g],t,a)\right) \\
& = \bigl(\chi_{n,m}(z_1,t,a)[g],\chi_{n',m}(z_2,t,a)[g]\bigr).
\end{align*}
It follows that the subspaces of $L^2(G,\mathcal{H})$ respectively generated by
$\chi_{n,m}(z_1,t,a)$ and $\chi_{n',m}(z_2,t,a)$ are orthogonal
and that their direct sum is the subspace generated by $\chi_{n+n',m}(z,t,a)$. \\
We deduce that $\chi([y_2,t],a) \overset{\bot}{\oplus} \chi([y_1,t],a)=\chi([y',t],a)$, i.e.
$\chi((\alpha_{2,2})_*(x),a)  \overset{\bot}{\oplus} \chi((\alpha_{1,2})_*(x),a)
=\chi(d_1^2(x),a)$. We finally derive from Proposition \ref{propertiesofshift} that
$\Shift(\chi(d_1^2(x),a))=\Shift(\chi((\alpha_{2,2})_*(x),a)) \circ \Shift(\chi((\alpha_{1,2})_*(x),a))$.
 \end{proof}

Since $\alpha_G^\psi : s\vEc_{G,L^2}^{F,\infty *}\times K \longrightarrow
\Fred(G,\mathcal{H}^\infty)$ is a morphism of hemi-simplicial $G$-spaces, it
yields a $G$-map: $Bs\vEc_{G,L^2}^{F,\infty *}\times K \longrightarrow B\Fred(G,\mathcal{H}^\infty)$, and, furthermore,
a $G$-map
$$\Omega B\alpha_G^\psi : \bigl(\Omega Bs\vEc_{G,L^2}^{F,\infty *}\bigr)\times K \longrightarrow \Omega B\Fred(G,\mathcal{H}^\infty).$$
Apply now the result of Proposition \ref{mighty}. For $p=0$, we recover a family
$(\psi_n)_{n\in \mathbb{N}}$ such that $\psi_n : (V_n(m))_{m \in \mathbb{N}} \longrightarrow
(\Hom(\Delta^m,V_n(\mathcal{H})))_{m \in \mathbb{N}}$ is a morphism of simplicial $U_n(F)$-spaces
for every $n \in \mathbb{N}$, and which satisfies conditions (i) to (vi). This family yields a $G$-map
$$\Omega B\alpha_G^\psi : \Omega Bs\vEc_{G,L^2}^{F,\infty *} \longrightarrow \Omega B\Fred(G,\mathcal{H}^\infty).$$
If $\psi'$ is another such family, it yields another $G$-map
$$\Omega B\alpha_G^{\psi'} : \Omega Bs\vEc_{G,L^2}^{F,\infty *} \longrightarrow \Omega B\Fred(G,\mathcal{H}^\infty).$$
Those two maps are $G$-homotopic. Indeed, by Proposition \ref{mighty} applied in the case
$p=1$ to $\psi \coprod \psi'$, we recover a family
$(\Psi_n)_{n\in \mathbb{N}}$ such that $\Psi_n : (V_n(m))_{m \in \mathbb{N}} \longrightarrow
(\Hom(\Delta^m \times [0,1],V_n(\mathcal{H})))_{m \in \mathbb{N}}$ is a morphism of simplicial $U_n(F)$-spaces
for every $n \in \mathbb{N}$, which satisfies conditions
(i) to (vi), and such that $\Psi_{|\{0\}}=\psi$ and $\Psi_{|\{1\}}=\psi'$.
This family then yields an equivariant homotopy
$$\Omega B\alpha_G^\Psi : \bigl(\Omega Bs\vEc_{G,L^2}^{F,\infty *}\bigr) \times I \longrightarrow \Omega B\Fred(G,\mathcal{H}^\infty)$$
from $\Omega B\alpha_G^\psi$ to $\Omega B\alpha_G^{\psi'}$.

\section{The natural transformation $KF_G(-) \rightarrow KF_G^{\text{Ph}}(-)$}\label{10.5}

\subsection{The definition of $\eta$}

We denote by $KF_G^{\text{Ph}}(-)$ the equivariant K-theory of Phillips (as defined and studied in \cite{Phillips})
on the category of finite proper $G$-CW-complexes.

Since $\mathcal{H}^\infty$ is a separable Hilbert space, we may consider, for every
finite proper $G$-CW-complex $X$, the map: $$\begin{cases}
[X,\underline{\Fred}(L^2(G,\mathcal{H}^\infty))]_G & \longrightarrow KF_G^{\text{Ph}}(X)
\\
\left[f:X \rightarrow \underline{\Fred}(L^2(G,\mathcal{H}^\infty))\right]
& \longmapsto \left[(X \times L^2(G,\mathcal{H}^\infty),
X \times L^2(G,\mathcal{H}^\infty),f)\right].
\end{cases}$$
This yields a natural transformation of group-valued functors: $$[-,\underline{\Fred}(L^2(G,\mathcal{H}^\infty))]_G \longrightarrow KF_G^{\text{Ph}}(-).$$
Since we know from Proposition \ref{fredweak} that the canonical map
$\underline{\Fred}(L^2(G,\mathcal{H}^\infty)) \longrightarrow \Omega B\Fred(G,\mathcal{H}^\infty)$
is a $G$-weak equivalence, it yields a natural equivalence
$$[-,\underline{\Fred}(L^2(G,\mathcal{H}^\infty))]_G
\overset{\cong}{\longrightarrow} [-,\Omega B\Fred(G,\mathcal{H}^\infty)]_G$$
on the category of proper $G$-CW-complexes. \\
The $G$-map $\alpha_G^\psi : \Omega Bs\vEc_{G,L^2}^{F,\infty *} \longrightarrow \Omega B\Fred(G,\mathcal{H}^\infty)$
constructed earlier then induces a uniquely defined natural transformation
$$[-,\Omega Bs\vEc_{G,L^2}^{F,\infty *}]_G
\overset{\alpha_G^\psi \circ -}{\longrightarrow} [-,\Omega B\Fred(G,\mathcal{H}^\infty)]_G$$
on the category of $G$-spaces.

Moreover, the composite
$\Omega Bs\vEc_{G,L^2}^{F,\infty *} \longrightarrow
\Omega Bs\underline{\vEc}_{G,L^2}^{F,\infty} \longrightarrow KF_G^{[\infty]}$
of the previously discussed $G$-weak equivalences is a $G$-weak equivalence and yields
a natural equivalence
$$[-,\Omega Bs\vEc_{G,L^2}^{F,\infty *}]_G \overset{\cong}{\longrightarrow}
[-,KF_G^{[\infty]}]_G=KF_G(-)$$
on the category of proper $G$-CW-complexes.

\vskip 2mm
We may now define a natural transformation
$\eta : KF_G(-) \longrightarrow KF_G^{\text{Ph}}(-)$  on the category of finite
proper $G$-CW-complexes by composing from left to right: $$KF_G(-) \overset{\cong}{\longleftarrow}
[-,\Omega Bs\vEc_{G,L^2}^{F,\infty *}]_G
\overset{\alpha_G^\psi \circ -}{\longrightarrow}
[-,\Omega B\Fred(G,\mathcal{H}^\infty)]_G
\overset{\cong}{\longleftarrow}
[-,\underline{\Fred}(L^2(G,\mathcal{H}^\infty))]_G \longrightarrow KF_G^{\text{Ph}}(-).$$
Obviously, $\eta_X$ is a homomorphism of abelian groups for every finite proper $G$-CW-complex $X$.

\vskip 2mm
The definition of Phillips' K-theory in negative degrees \cite{Phillips} pp 80-81
makes it clear that $\eta$ may be extended to negative degrees as a natural transformation
$KF_G^*(-) \longrightarrow KF_G^{*\text{Ph}}(-)$ which preserves the long exact sequence of a
finite proper $G$-CW-pair.

\subsection{Properties of $\eta$}
Recall from \cite{Phillips} example 3.4 page 40 that there is a natural transformation
$\mathbb{K}F_G(-) \longrightarrow  KF_G^{\text{Ph}}(-)$ on the category of locally
compact $G$-spaces, defined by mapping every finite dimensional $G$-Hilbert bundle $E \rightarrow X$
to the class of the cocycle $(E,0,0)$ in $KF_G^{Ph}(X)$.

\begin{prop}\label{transformationcommute}
The diagram
$$\xymatrix{
KF_G(-) \ar[rr]^{\eta} & &  KF_G^{\text{Ph}}(-) \\
& \mathbb{K}F_G(-) \ar[ul]^{\gamma} \ar[ur]
}$$ is commutative on the category of finite proper $G$-CW-complexes.
\end{prop}

\begin{proof}

It suffices to prove that the triangle
$$\xymatrix{
KF_G(-) \ar[rr]^{\eta} & &  KF_G^{\text{Ph}}(-) \\
& \mathbb{V}\text{ect}_G^F(-) \ar[ul]^{\gamma} \ar[ur]
}$$ is commutative on the category of finite proper $G$-CW-complexes. \\
However $\gamma$ is the composite of the natural transformations:
$$\mathbb{V}\text{ect}_G^F(-) \overset{\cong}{\longrightarrow} [-,s\vEc_{G,L^2}^{F,\infty *}]_G
\longrightarrow [-,\Omega Bs\vEc_{G,L^2}^{F,\infty *}]_G
\longrightarrow [-,KF_G^{[\infty]}]_G$$
where the first morphism is the one induced by pulling back the universal bundle
$Es\vEc_{G,L^2}^{F,\infty *} \longrightarrow s\vEc_{G,L^2}^{F,\infty *}[1]$. \\
Since the square $$\begin{CD}
\Omega Bs\vEc_{G,L^2}^{F,\infty *} @>{\alpha_G^\psi}>> \Omega B\Fred(G,\mathcal{H}^\infty) \\
@AAA @AAA \\
s\vEc_{G,L^2}^{F,\infty *}[1] @>{\chi}>> \underline{\Fred}(L^2(G,\mathcal{H}^\infty))
\end{CD}$$
is commutative, the composite natural transformation
$\mathbb{V}\text{ect}_G^F(-) \longrightarrow KF_G(-) \overset{\eta}{\longrightarrow} KF_G^{\text{Ph}}(-)$
is simply the composite natural transformation
$$\mathbb{V}\text{ect}_G^F(-) \overset{\cong}{\longrightarrow} [-,s\vEc_{G,L^2}^{F,\infty *}]_G
\overset{\chi \circ -}{\longrightarrow}
[-,\underline{\Fred}(L^2(G,\mathcal{H}^\infty))]_G \longrightarrow KF_G^{\text{Ph}}(-).$$
Moreover,  $\chi$ defines a morphism of Hilbert $G$-bundles
$$\overline{\chi} : \begin{cases}
s\vEc_{G,L^2}^{F,\infty *}[1] \times L^2(G,\mathcal{H}^\infty) & \longrightarrow
s\vEc_{G,L^2}^{F,\infty *}[1] \times L^2(G,\mathcal{H}^\infty) \\
(x,y) & \longmapsto \bigl(x,\chi(x)[y]\bigr).
\end{cases}$$
and the maps
$$\begin{cases}
s\widetilde{\vEc}_{G,L^2}^{n,F,\infty*} \times \Sim_n(F) & \longrightarrow L^2(G,\mathcal{H}) \\
(x,M) & \longmapsto \tilde{\chi_n}(x).M,
\end{cases}
$$ for $n \in \mathbb{N}$, yield a morphism of $G$-simi-Hilbert bundles
$$\xymatrix{
Es\vEc_{G,L^2}^{F,\infty*} \ar[dr] \ar[rr]^{\underline{\chi}} & &
s\vEc_{G,L^2}^{F,\infty *}[1] \times L^2(G,\mathcal{H}^\infty)
\ar[dl] \\
&
s\vEc_{G,L^2}^{F,\infty *}[1]
}$$
which maps $Es\vEc_{G,L^2}^{F,\infty*}$ onto the kernel of $\overline{\chi}$.

\vskip 2mm
Let now $X$ be a finite proper $G$-CW-complex,
and $X \overset{f}{\longrightarrow} s\vEc_{G,L^2}^{F,\infty *}$
be a $G$-map. \\
Pulling back the diagram of bundle morphisms
$$\xymatrix{
Es\vEc_{G,L^2}^{F,\infty*} \ar[dr] \ar[r]^(0.35){\underline{\chi}} &
s\vEc_{G,L^2}^{F,\infty *}[1] \times L^2(G,\mathcal{H}^\infty) \ar[r]^{\overline{\chi}}
\ar[d] & s\vEc_{G,L^2}^{F,\infty *}[1] \times L^2(G,\mathcal{H}^\infty) \ar[dl] \\
& s\vEc_{G,L^2}^{F,\infty *}[1]
}$$
by $f$ yields a diagram of bundle morphisms
$$\xymatrix{
E \ar[dr] \ar[r]^(0.3){f_*(\underline{\chi})} &
X \times L^2(G,\mathcal{H}^\infty) \ar[r]^{f_*(\overline{\chi})} \ar[d] & X \times L^2(G,\mathcal{H}^\infty), \ar[dl] \\
& X
}$$
where $f_*(\underline{\chi})$ maps $E$ onto the kernel of $f_*(\overline{\chi})$.

\vskip 2mm
Thus $\eta(\gamma([f]))$ is the class of $f_*(\overline{\chi})$ in $KF_G^{\text{Ph}}(X)$, and
we need to prove that it is also the class $[E \rightarrow 0]$. \\
Since the dimension of the kernel of $f_*(\overline{\chi})$ is locally constant,
it follows that $\Ker f_*(\overline{\chi})$ is a finite-dimensional $G$-vector bundle over $X$ (see \cite{Janich}).
The above diagram thus induces a strong morphism of $G$-vector bundles:
$$\xymatrix{
E \ar[rr]^{\cong} \ar[dr] & & \Ker f_*(\overline{\chi}). \ar[dl] \\
& X
}$$
Hence $\Ker f_*(\overline{\chi})$ and $E$ are isomorphic.

\vskip 2mm
Moreover $(\Ker f_*(\overline{\chi}))^\bot$ is a sub-$G$-Hilbert bundle of
$X \times L^2(G,\mathcal{H}^\infty)$, and $f_*(\overline{\chi})$ induces an isomorphism of $G$-Hilbert bundles
$$\Ker f_*(\overline{\chi}) ^\bot \overset{\cong}{\longrightarrow} X \times L^2(G, \mathcal{H}^\infty).$$
We deduce that the class
of $\Ker f_*(\overline{\chi})^\bot \longrightarrow X \times L^2(G, \mathcal{H}^\infty)$
in $KF_G^{\text{Ph}}(X)$ is $0$. \\
Finally
$$[f_*(\overline{\chi})]=\left[\Ker f_*(\overline{\chi}) \rightarrow 0\right] +
\left[\Ker f_*(\overline{\chi})^\bot \rightarrow X \times L^2(G,
  \mathcal{H}^\infty) \right]=\left[E \rightarrow 0\right]$$
in $KF_G^{\text{Ph}}(X)$.  \end{proof}

\begin{cor}
For every finite proper $G$-CW-complex $X$
$$\eta_X : KF_G(X) \overset{\cong}{\longrightarrow} KF_G^{\text{Ph}}(X)$$
is an isomorphism of abelian groups.
\end{cor}

\begin{proof}
Since $\eta$ is a natural transformation between good equivariant cohomology theories in
negative degrees, it suffices to establish the result in the case
$X=(G/H) \times Y$, where $H$ is a compact subgroup of $G$, and $Y$ is a finite CW-complex on which $G$ acts trivially
(by an argument that is similar to that of the proof of Proposition 1.5 in \cite{Bob2}).
In this case, we deduce from Proposition
\ref{transformationcommute} that the diagram
$$\xymatrix{
KF_G((G/H) \times Y) \ar[rr]_{\cong}^{\eta_{(G/H) \times Y}} & &   KF_G^{\text{Ph}}((G/H) \times Y) \\
& \mathbb{K}F_G((G/H) \times Y) \ar[ul]_{\gamma_{(G/H) \times Y}}^{\cong} \ar[ur]_{\cong}
}$$
is commutative. By \cite{Phillips}, the map
$\mathbb{K}F_G((G/H) \times Y) \longrightarrow  KF_G^{\text{Ph}}((G/H) \times Y)$ is an isomorphism.
By Proposition 4.4 of \cite{Ktheo1}, $\gamma_{(G/H) \times Y}$ is an isomorphism.
Hence $\eta_{(G/H) \times Y}$ is an isomorphism. \end{proof}

\vskip 2mm
\noindent \textbf{An open problem :}
there remains essentially one problem to be solved here: is $\eta$
compatible with the product maps (or with Bott-periodicity maps)? This has proved out of our reach, for two main reasons:
\begin{enumerate}
\item it is not clear at all how the multiplicative structure on $KF_G^{\text{Ph}}(-)$
may be understood in terms of the space $\underline{\Fred}(L^2(G,\mathcal{H}^\infty))$;
\item in constructing $\eta$, we have dumped the very categorical structures that
helped constructed the product maps, i.e. the $\Gamma-G$-space structure.
\end{enumerate}

\section*{Acknowledgements}

I would like to thank Bob Oliver for his constant support and the countless good advice he gave me
during my research on this topic.

\bibliographystyle{plain}

\end{document}